\documentclass[1 [leqno,11pt]{amsart}
\usepackage{amsfonts}
\usepackage{amssymb, amsmath}
\newcommand{\newcom}{\newcommand}

\def\eqdefa{\buildrel\hbox{\footnotesize def}\over =}

\def\fd{\frak{\Delta}}
\newcom{\al}{\alpha}
\newcom{\be}{\beta}
\newcom{\eps}{\varepsilon}
\newcom{\ga}{\gamma}
\newcom{\Ga}{\Gamma}
\newcom{\ka}{\kappa}
\newcom{\Lam}{\Lambda}
\newcom{\lam}{\lambda}
\newcom{\Om}{\Omega}
\newcom{\om}{\omega}
\newcom{\Si}{\Sigma}
\newcom{\si}{\sigma}
\newcom{\tht}{\theta}
\newcom{\dtri}{\nabla}
\newcom{\tri}{\triangle}
\newcom{\oo}{\infty}
\newcom{\vphi}{\varphi}
\newcom{\calb}{{\mathcal B}}
\newcom{\calc}{{\mathcal C}}
\newcom{\cD}{{\mathcal D}}
\newcom{\cF}{{\mathcal F}}
\newcom{\cG}{{\mathcal G}}
\newcom{\cI}{{\mathcal I}}
\newcom{\cJ}{{\mathcal J}}
\newcom{\cL}{{\mathcal L}}
\newcom{\cM}{{\mathcal M}}
\newcom{\cP}{{\mathcal P}}
\newcom{\cR}{{\mathcal R}}
\newcom{\cS}{{\mathcal S}}
\newcom{\cQ}{{\mathcal Q}}
\newcom{\caly}{{\mathcal Y}}
\newcom{\calz}{{\mathcal Z}}
\newcom{\bfz}{{\bf Z}}
\newcom{\R}{\Bbb R}
\newcom{\N}{\Bbb N}
\newcom{\Z}{\Bbb Z}
\newcom{\C}{\Bbb C}
\newcom{\E}{\Bbb E}

\newcom{\f}{\frac}
\newcom{\di}{\displaystyle\int}
\newcom{\ds}{\displaystyle\sum}
\newcom{\dl}{\displaystyle\lim}
\newcom{\ov}{\overline}
\newcom{\sset}{\subset}
\newcom{\wt}{\widetilde}
\newcom{\p}{\partial}
\newcom\na{\nabla}
\newcom{\co}{\cdot}
\newcom{\suml}{\sum\limits}
\newcom{\supl}{\sup\limits}
\newcom{\intl}{\int\limits}
\newcom{\infl}{\inf\limits}
\newcom{\disp}{\displaystyle}
\newcom{\non}{\nonumber}
\newcom{\no}{\noindent}
\newcom{\QED}{$\square$}
\def\ef{\hphantom{MM}\hfill\llap{$\square$}\goodbreak}

\newtheorem{athm}{\bf \t}[section]
\newenvironment{thm} [1] {\def\t{#1}\begin{athm} \bf \rm} {\end {athm}}
\newcom{\bthm}{\begin{thm}}\newcom{\ethm}{\end{thm}}

\newcom{\beq}{\begin{equation}}
\newcom{\eeq}{\end{equation}}
\newcom{\ben}{\begin{eqnarray}}
\newcom{\een}{\end{eqnarray}}
\newcom{\beno}{\begin{eqnarray*}}
\newcom{\eeno}{\end{eqnarray*}}
\newtheorem{defi}{Definition}[section]
\newtheorem{lem}{Lemma}[section]
\newtheorem{rmk}{Remark}[section]
\newtheorem{col}{Corollary}[section]
\newtheorem{prop}{Proposition}[section]
\renewcommand{\theequation}{\thesection.\arabic{equation}}

\begin{document}

\title[Large time wellposedness  of  Capillary-Gravity Waves ]
{Large time wellposdness   to the 3-D Capillary-Gravity Waves in the
long wave regime}

\author[M. Ming]{Mei Ming}%
\address[M. Ming]
 {Academy of
Mathematics $\&$ Systems Science, CAS\\
Beijing 100190, CHINA.} \email{meiming@amss.ac.cn}
\author[P. Zhang]{Ping Zhang}%
\address[P. Zhang]
 {Academy of Mathematics and Systems Science and  Hua Loo-Keng Key Laboratory of Mathematics,
  Chinese Academy of Sciences, Beijing 100190, CHINA} \email{zp@amss.ac.cn}
\author[Z. Zhang]{Zhifei Zhang}
\address[Z.  Zhang]{School of Mathematical Sciences and LMAM, Peking University\\
Beijing 100871, CHINA.} \email{zfzhang@math.pku.edu.cn}

\date{April 28, 2011}



\begin{abstract}
In the  regime of weakly transverse long waves, given long-wave
initial data, we  prove that the nondimensionalized water wave
system in an infinite strip under influence of gravity and  surface
tension on the upper free interface has a unique solution on
$[0,{T}/\eps]$ for some $\eps$ independent of constant $T.$ We shall
prove in the subsequent paper \cite{MZZ2} that  on the same time
interval, these solutions can be accurately approximated by sums of
solutions of two decoupled Kadomtsev-Petviashvili (KP) equations.
\end{abstract}

  \maketitle

  \tableofcontents

  \renewcommand{\theequation}{\thesection.\arabic{equation}}
  \setcounter{equation}{0}
  \section{Introduction}

\subsection{General setting}

The aim of this paper is to prove that in the  regime of weakly
transverse long waves, given long-wave initial data, the
nondimensionalized water wave system in an infinite strip under
influence of gravity and  surface tension on the upper free
interface has a unique solution on $[0,{T}/\eps]$ for some $\eps$
independent of constant $T.$ More precisely, we consider the
irrotational flow of an incompressible, inviscid fluid in an
infinite strip with impermeable bottom under the influence of
gravity and  surface tension on the upper free interface. In this
setting, we may assume that the free surface is described by the
graph $z=\zeta(t,x,y),$ and $z=-d+b(x,y)$ (the constant $d>0$)
describes the bottom of the infinite strip. Since the fluid is
incompressible and irrotational, there exists a velocity potential
$\phi$ such that the velocity  field is given by $v=\na\phi.$ Then
one can reduce the motion of the fluid to a system in terms of the
velocity potential $\phi$ and $\zeta$:
\beq\label{eq:waterwave}\left\{
\begin{array}{ll}
\p_x^2\phi+\p_y^2\phi+\p_z^2\phi=0,\qquad -d+b<z<\zeta,\\
\p_n\phi=0,\qquad z=-d+b,\\
\p_t\zeta+\na_h\phi\cdot\na_h\zeta=\p_z\phi,\qquad z=\zeta,\\
\p_t\phi+\f12(|\na_h\phi|^2+(\p_z\phi)^2)+g\zeta=\kappa\na_h\cdot\big(\f{\na_h
\zeta}{\sqrt{1+|\na_h\zeta|^2}}\big), \ z=\zeta,
\end{array}\right.
\eeq where $\na_h=(\p_x,\p_y)$ and $\p_n\phi$ denotes the outward
normal
 derivative at the bottom  of the fluid region,
 $g, \kappa>0$ denotes the gravitational force constant and the
surface tension coefficient respectively.

It is well-known \cite{CSS, CSS2, zak} that the water wave system
(\ref{eq:waterwave}) can be  reduced to a system  of two evolution
equations coupling the parametrization of the free surface $\zeta$
and the trace of the velocity potential $\phi$ at the free surface.
More precisely, let $n_+$ be the outward unit normal vector to the
free surface,
\[ \psi(t,x_h){\eqdefa}\phi(t,x_h,\zeta(t,x_h))\quad\mbox{with}\quad x_h\eqdefa (x,y) \]
and the (rescaled) Dirichlet-Neumann operator $G(\zeta,b)$ (or
simply $G(\zeta)$) \beno
G(\zeta)\psi{\eqdefa}\sqrt{1+|\na_h\zeta|^2}\p_{n_+}\phi|_{z=\zeta(t,x_h)}.
\eeno Taking the trace of (\ref{eq:waterwave}) on the free surface
$z=\zeta(t,x_h),$  the system (\ref{eq:waterwave}) is equivalent to
(see \cite{CSS, CSS2, zak} for instance) \beq \label{eq:Hamiltonian
form} \left\{
\begin{array}{ll}
\p_t \zeta-G(\zeta)\psi=0, \\
\p_t \psi+g\zeta+\frac{1}{2}|\na_h
\psi|^2-\f{(G(\zeta)\psi+\na_h\zeta\cdot \na_h
\psi)^2}{2{(1+|\na_h\zeta|^2)}}=\kappa\na_h\cdot(\f{\na_h
\zeta}{\sqrt{1+|\na_h\zeta|^2}}).
\end{array}\right.
\eeq

Recently this subject of water wave problem has attracted the interest of lots of
mathematicians. Concerning 2-D water wave system, when the surface
tension is neglected and the motion of free surface is a small
perturbation of still water, one could check Nalimov \cite{nal},
Yosihara \cite{yos} and  Craig \cite{wal}. In general, the local
wellposedness of 2-D full water wave problem was solved by Wu
\cite{Wu1} and see also Ambrose and Masmoudi \cite{amb1}, where they
firstly studied the 2-D irrotational water wave problem with nonzero
surface tension and proved the local wellposedness of the problem,
then they showed that as the surface tension goes to zero, the
solutions of nonzero surface tension problem goes to solutions of
the corresponding zero surface tension problem. (See similar result
by the same authors for the 3-D problem in \cite{amb2}). One may
also check \cite{Wu3} for the most recent almost global wellposedness to the 2-D full water wave system
without surface tension.

Concerning the 3-D water wave problem without surface tension,
 Wu \cite{Wu2} proved its local wellposedness  under the assumptions
that the fluid is irrotational and there is no self-intersection
point on the initial surface.  Lannes \cite{lan1} considered the
same problem in the case of finite depth under Eulerian coordinates.
More recently, following Lannes \cite{lan1}'s framework, Ming and
Zhang \cite{MZ} proved the local wellposedness of water wave system
in an infinite strip and under the influence of surface tension on
the free interface. Recently,  Alazard, Burq and  Zuily
\cite{Alazard} studied the regularities to the local  solutions of
3-D water wave system,  Germain, Masmoudi and Shatah \cite{Germain},
Wu \cite{Wu4} independently proved the global wellposedness of the
3-D water system without surface tension. D. Lannes proved very
recently a more general well-posedness result on the two-fluid
system with surface tension on the interface
\cite{Lannes-stability}, and he also stated a stability criterion
for these two-fluid interfaces and some applications.

When the initial vorticity does not equal zero, Iguchi,  Tanaka and
Tani \cite{Ig-Ta} proved the local wellposedness of the  free
boundary problem for an incompressible ideal fluid in two space
dimensions without surface tension. Similar result was proved by
Ogawa and Tani \cite{Og-Ta1} to the case with surface tension. And
in \cite{Og-Ta2},  Ogawa and Tani generalized the wellposedness
result in \cite{Og-Ta1} to the case of finite depth. One may check
\cite{Co-Sh, lin, Sh-Ze, zzc} for some recent study on the local
wellposedness of free boundary problem of 3-D Euler equations under
the Taylor sign condition on the initial interface.

\subsection{Nondimensionalized water-wave system and main results}

The complexity of the full water wave system led physicists and
mathematicians to derive simpler sets of equations likely to
describe the dynamics of \eqref{eq:waterwave} in some specific
physical regimes. Yet the mathematical analysis of the these models
on their relevance as approximate models for the water wave
equations only began three decades ago.

In the particular regime of weakly transverse long waves, Craig
\cite{wal} and Kano and Nishida \cite{KN} gave a first justification
of the 1-D Boussinesq systems. However, the convergence result
proved in \cite{KN}  is given on a time scale which is too short to
capture the nonlinear and dispersive effects for Boussinesq systems;
the correct large time convergence result was later proved by Craig
in \cite{wal}. In the 2-D case, assuming the large time
wellposedness of the dimensionless water wave equations, Bona, Colin
and Lannes \cite{BCL} justified the Boussinesq approximation. Notice
that at the first order, the Boussinesq systems reduce to two
decoupled Korteweg-de Vries (KdV) equations in 1-D case and
Kadomtsev-Petviashvili (KP) equations in 2-D case. Many papers
addressed the problem of the validity of KdV model (\cite{wal, Kano,
sw1, sw2,Iguchi}).  For the KP model, a first attempt was done in
\cite{Kano} for small and analytic initial data. But as in
\cite{KN}, the time scale considered is again unfortunately too
small for the relevant dynamics. In \cite{Lannes-non}, Lannes and
Saut proved the KP limit by assuming a large time wellposedness
theorem and a specific control of the solutions to the dimensionless
full water wave system without surface tension.

In the fundamental paper \cite{Lannes-Inven}, Alvarez-Samaniego and
Lannes systematically justified various 3-D asymptotic models,
including shallow-water equations, Boussinesq system, KP
approximation, Green-Naghdi equations, Serre approximation and
full-dispersion model for the water wave system without surface
tension.

As is well-known, the proof of large-time wellposedness of
dimensionless form of  \eqref{eq:Hamiltonian form}
 is the most delicate point in the
justification of the related approximations. The purpose of our
paper is to
 prove
that: in the long wave regime, the evolution of long wave-length
initial data  to \eqref{eq:Hamiltonian form} has a unique solution
on $[0,T/\eps]$ for some $\eps$ independent positive time $T.$ The
main idea of our proof is similar as in our previous work \cite{MZ}
but more complicated. We use the similarity between the main part of
Dirichlet-Neumann operator and the surface tension operator to
construct the energy for the linearized system, and we also use a
Nash-Moser iteration theorem to handle with the loss of derivatives
in the energy estimates. We refer to \cite{Tzvetkov,
Lannes-stability} for another way to prove the well-posedness
without using Nash-Moser iteration by taking the sufficient amount of derivatives
to the system.

Now we are going to introduce the specific regime we used in our
paper. This regime of weakly transverse long waves can be specified
in terms of relevant characters of the wave, namely, its typical
amplitude $a$, the mean depth $d$, the typical wavelength $\lambda$
along the longitudinal direction ( say , the $x$ axis), and $\frac
\lambda{\sqrt{\eps}}$ the wavelength in $y$ direction, $B$ the
amplitude of the variations of the bottom topogaphy, which satisfy
\beq \label{KPregime} {\eps}=\frac ad=\frac {d^2}{\lambda^2}=\frac B
d\ll 1.\eeq The asymptotic study becomes more transparent when
working with variables scaled in such a way that the dependent
quantities and the initial data which appear in the initial value
problem are all of order one. The relation \eqref{KPregime} which
sets the KP regime here are connected with small parameters in the
nondimensionalized equations of motion.

For simplicity, we  take gravitational constant $g=1$ in
\eqref{eq:Hamiltonian form} and denote the dimensionless variables
with a prime. We set \beq\label{scaling1}
\begin{split}
&x=\lambda x',~~y=\frac{\lambda }{\sqrt{\eps}} y',~~z=d z',~~t=\frac\lambda {\sqrt{d}}t',\\
&\zeta=a\zeta',~~\phi=\frac{a }{\sqrt{d}}
\lambda\phi',~~b=Bb',~~\psi=\frac{a}{\sqrt{d}}  \lambda\psi'.
\end{split}
\eeq Then we write the dimensionless form of (\ref{eq:waterwave}) as
follows (by neglecting the prime)
\beq\label{eq:waterwave-non}\left\{
\begin{array}{ll}
\eps\p_x^2\phi+\eps^2\p_y^2\phi+\p_z^2\phi=0,& -1+\eps b<z<\eps\zeta,\\
-\eps\na_h^{\eps}(\eps b)\cdot\na_h^{\eps}\phi+\p_z\phi=0,& z=-1+\eps b,\\
\p_t\zeta+\eps\na_h^{\eps}\zeta\cdot\na_h^{\eps}\phi=\f 1\eps\p_z\phi,&z=\eps\zeta,\\
\p_t\phi+\f12\big(\eps|\na_h^{\eps}\phi|^2+(\p_z\phi)^2\big)+\zeta\\
\ =({\kappa\eps}/{d^2})\na_h^{\eps}\cdot\big({\na_h^{\eps}
\zeta}/{\sqrt{1+\eps^3|\na_h^{\eps}\zeta|^2}}\big),& z=\eps\zeta
\end{array}\right.
\eeq where $\na_h^{\eps}\eqdefa (\p_x,\sqrt{\eps}\p_y)$. We define
the scaled Dirichlet-Neumann operator $G[{\eps}\zeta]$ by \beq
\label{DNO} G[{\eps}\zeta]\psi{\eqdefa}
\bigl(-{\eps}\na^{\eps}_h({\eps}\zeta)\cdot
\na^{\eps}_h\Phi+\p_z\Phi\bigr)|_{z={\eps}\zeta},\eeq with $\Phi$
solving  \beq\label{eq:elliptic equation-DN}
\left\{\begin{array}{ll}
\p_z^2\Phi+\eps\p_x^2\Phi+\eps^2\p_y^2\Phi=0,\quad-1+\eps b<z<{\eps}\zeta,\\
\Phi|_{z={\eps}\zeta}=\psi,\quad\p_n^{P_0}\Phi|_{z=-1+\eps b}=0.
\end{array}\right.
\eeq  Here $\p_n^{P_0}\Phi$ is the outward conormal derivative
associated to the elliptic equation \eqref{eq:elliptic equation-DN}, i.e.,
\[\p_n^{P_0}\Phi|_{z=-1+\eps b}{\eqdefa} n\cdot
P_0\na\Phi|_{z=-1+\eps b}\quad\hbox{with}\quad P_0=\left(\begin{matrix}\eps & 0 & 0\\
0&\eps^2&0\\
0& 0&1
\end{matrix}\right) \]
where $\na\eqdefa(\na_h,\partial_z)$  and $n$ stands for the outward
unit normal vector to the bottom of the infinite strip $\{(x,y,z)| \
-1+\eps b(x,y)<z<{\eps}\zeta(t,x,y)\}.$

 Then similar to (\ref{eq:Hamiltonian form}), the system \eqref{eq:waterwave-non} becomes
\beq\label{eq:Hamiltonian form-non} \left\{
\begin{array}{ll}
\p_t \zeta-\frac 1 {{\eps}}G[{\eps}\zeta]\psi=0, \\
\p_t \psi+\zeta+\frac{{\eps}}{2}|\na^{\eps}_h \psi|^2-\f{\eps^2}2{
(\frac 1{\eps} G[{\eps}\zeta]\psi+{\eps}\na^{\eps}_h\zeta\cdot
\na^{\eps}_h \psi)^2}/{(1+{\eps}^3|\na^
\eps_h\zeta|^2)}\\
\ =\alpha\eps\na^{\eps}_h \cdot \big({\na^{\eps}_h
\zeta}/{\sqrt{1+{\eps}^3|\na^{\eps}_h \zeta|^2}}\big),\\
\zeta|_{t=0}=\zeta_0^\eps,\quad \psi|_{t=0}=\psi_0^\eps,
\end{array}\right.
\eeq where $\alpha={\kappa}/{d^2}$ is the so-called Bond number.

The uniform energy estimates for the solutions to the linearized
system of \eqref{eq:Hamiltonian form-non} plays an essential role in
the proof of the large time well-posedness for the nonlinear system.
Compared with \cite{Lannes-Inven}, there is an additional term on
the left hand side of the linearized system \eqref{eq:linearized
system} due to the appearance of surface tension term  in
\eqref{eq:Hamiltonian form-non}. Then the  ordinary energy
functional given in \cite{Lannes-Inven} will not work for
\eqref{eq:linearized system}, otherwise, there will be a loss of one
order derivative in the energy estimates. The key point here is that
we observed the resemblance between the principle part of the
Dirichlet-Neumann operator and the linearized surface tension
operator, and based on this fact we constructed an effective energy
functional to obtain the uniform energy estimate for the linearized
system \eqref{eq:linearized system}. This new energy functional
leads to the use of a parameterized Sobolev space and some
complicated pseudo-differential operator estimates in the process of
the energy estimates. With these preparations, we can use a modified
version of Nash-Moser iteration theorem in \cite{Lannes-IUMJ} to
prove the large time existence of solutions to \eqref{eq:Hamiltonian
form-non}.

Before presenting our main results, we introduce the following
function space

\begin{defi}\label{def1.1}
{\sl We define the space $\frak{X}^s$ as
\[
\frak{X}^s{\eqdefa}\left\{U=(\zeta,\psi)^T: \zeta\in
H^{2s+1}(\R^2),\na_h \psi\in
 H^{2s-\f12}(\R^2)^2\right\}
\]
endowed with the semi-norm
\[
|U|_{\frak{X}^s}{\eqdefa}\sqrt{\eps}|\zeta|_{H^{2s+1}_\eps}+|\zeta|_{H^{2s}_\eps}+\sqrt{\eps}|\na_h^{\eps}\zeta|
_{H^{s}}+|\zeta|_{H^s}+|\mathfrak P\psi|_{H^{2s}_\eps}+|\mathfrak
P\psi|_{H^{s}}
\]
for $\mathfrak {P}\eqdefa
{|D^{\eps}_{h}|}/{(1+\sqrt{\eps}|D^{\eps}_{h}|)^\frac12},$
$|D^{\eps}_h|$ the Fourier multiplier with the symbol
$(\xi_1^2+\eps\xi_2^2)^\f12,$ and $H^s_\eps(\R^2)$ is the space of
tempered distributions $v$ so that \beq \label{defsob}
|v|_{H^s_\eps}\eqdefa |(1+|D_h^\eps|^2)^{\f{s}2}v|_{L^2}<\infty. \eeq}
\end{defi}

\begin{rmk}
The scaled Sobolev space $H^s_\eps$ is naturally connected with the equivalent form for the energy functional
introduced in Section 6, which is crucial to obtain an uniform energy estimates
for the linearized water-wave system.
\end{rmk}

Our result of this paper is as follows.
\bthm{Theorem}\label{thm:KP approximation}  {\sl Let the Bond number
$\al>0$ and $\al\neq \f13, s\ge m_0$ for some $m_0\in (9,10).$
Assume that there exist $P>D>0$ such that $b\in H^{2s+2P+1}(\R^2)$
and bounded initial data $(\zeta_0^\eps,\psi_0^\eps)\in
\frak{X}^{s+P}$
 satisfy
\[
\inf_{\R^2}\bigl(1+{\eps}\zeta_0^\eps-{\eps}
b\bigr)>0\qquad\hbox{uniformly for}\quad \eps\in(0,1).
\]
 Then there exits $T>0$ such that
(\ref{eq:Hamiltonian form-non}) has a unique family of solutions
$(\zeta^\eps,\psi^\eps)_{0<\eps<1}$ on $[0,\f{T}\eps]$ with
 $(\zeta^\eps)_{0<\eps<1},
(\p_x\psi^\eps)_{0<\eps<1}$, and
$(\sqrt{\eps}\p_y\psi^\eps)_{0<\eps<1}$ being uniformly bounded in
$C([0,T/\eps];$ $H^{s+D-\f12}(\R^2))$. }
 \ethm

\begin{rmk}\label{rmk1.0a}
Let \beno \zeta_{KP}^\eps(t,x,y)\eqdefa \f1
{\sqrt{2}}\big(\zeta_+(\eps t,x-t,y)-\zeta_-(\eps t,x+t,y)\big),
\eeno where $\zeta_\pm(\tau,X,Y)$ solve the uncoupled KP equations
\beno (KP)^\pm\quad \p_\tau
\zeta_\pm\pm\f12\p_{X}^{-1}\p_Y^2\zeta_\pm\pm\big(\f16-\f\al2\big)\p_{X}^3\zeta_\pm
+\f{3\sqrt{2}}4\zeta_\pm\p_{X}\zeta_\pm=0. \eeno We shall prove in
\cite{MZZ2} that: in addition to the assumptions in Theorem
\ref{thm:KP approximation}, we assume moreover  \beno \lim_{\eps\to
0}|\zeta^\eps_0-\zeta_0|_{\p_xH^{s+D-\f12}}=0\quad\mbox{and}\quad
\lim_{\eps\to 0}|\p_x\psi^\eps_0-\p_x\psi_0|_{\p_xH^{s+D-\f12}}=0
\eeno with  $(\p_x\psi_0,\zeta_0)\in \p_xH^{s+D-\f12}(\R^2)$ and
$(\p_y^2\p_x\psi_0,\p_y^2\zeta_0)\in \p_x^2H^{s+D-\f92}(\R^2)$. Then
$(KP)_\pm$ with initial data $(\p_x\psi_0\pm\zeta_0)/\sqrt{2}$ has a
unique solution $\zeta_\pm\in C([0,T/\eps];H^{s+D-\f32}(\R^2))$.
Furthermore, there holds\beno \lim_{\eps\rightarrow
0}|\zeta^\eps-\zeta^\eps_{KP}|_{L^\infty([0,T/\eps]\times \R^2)}=0.
\eeno
\end{rmk}

In case when the Bond number $\al=\frac13,$ the coefficients of the
third order dispersion terms in $(KP)^\pm$ vanish and the resulting
equations become illposed. These third order terms in $(KP)^\pm$
equations represent the leading order dispersive effects in the
water-wave problem and their disappearance means that in this
parameter regime the water waves are almost dispersionless. To model
interesting behaviors and capture the dispersive nature of the
water-wave problem for this parameter regime in our following
paper\cite{MZZ2}, we need to modify the scaling in \eqref{scaling1}
firstly and then prove the large-time existence for the new
water-wave system.  More precisely, we set \beq \label{scaling2}
\begin{split}
&x=\lam x',\quad y=\f{\lam}{\eps}y',\quad z=dz',\quad
t=\f{\lam}{\sqrt{d}}t',\\
& \zeta=a\zeta',\quad \phi=\f{a}{\sqrt{d}}\lam\phi',\quad
b=Bb',\quad \psi=\f{a}{\sqrt{d}}\lam\psi', \end{split} \eeq with
\[
\eps=\sqrt{\f ad}=\f{d^2}{\lam^2}=\sqrt{\f Bd}.
\]
Then similar to \eqref{eq:waterwave-non}, we obtain the following
dimensionless form of the original system (by neglecting the prime)
\ben\label{eq:water wave-rescaled system} \left\{
\begin{array}{ll}
\eps\p_x^2\phi+\eps^3\p_y^2\phi+\p_z^2\phi=0, & -1+\eps^2 b<z<\eps^2\zeta,\\
-\eps\wt\na_h^{\eps}(\eps^2 b)\cdot\wt\na_h^{\eps}\phi+\p_z\phi=0, & z=-1+\eps^2 b,\\
\p_t\zeta+\eps^2\wt\na_h^{\eps}\zeta\cdot\wt\na_h^{\eps}\phi=\f
1\eps\p_z\phi,
& z=\eps^2\zeta,\\
\p_t\phi+\f12\big(\eps^2|\wt\na_h^{\eps}\phi|^2+\eps(\p_z\phi)^2\big)+\zeta &\\
\ =\al\eps\wt\na_h^{\eps}\cdot\big({\wt\na_h^{\eps}
\zeta}/{\sqrt{1+\eps^5|\wt\na_h^{\eps}\zeta|^2}}\big),&
z=\eps^2\zeta,
\end{array}\right.
\een where $\wt\na_h^{\eps}\eqdefa (\p_x,\eps\p_y)$ and
$\al=\kappa/d^2$ is still the Bond number.
 We define a new
scaled Dirichlet-Neumann operator $\wt G[{\eps^2}\zeta]$ by
\beq\label{newdi} \wt G[{\eps^2}\zeta]\psi:=
\bigl(-{\eps}\wt\na^{\eps}_h({\eps^2}\zeta)\cdot
\wt\na^{\eps}_h\phi+\p_z\phi\bigr)|_{z={\eps^2}\zeta},\eeq with
$\phi$ solving  \[ \left\{\begin{array}{ll}
\p_z^2\phi+\eps\p_x^2\phi+\eps^3\p_y^2\phi=0,\quad-1+\eps^2 b<z<{\eps^2}\zeta,\\
\phi|_{z={\eps^2}\zeta}=\psi,\quad\p_n^{\wt{P}_0}\phi|_{z=-1+\eps^2
b}=0,
\end{array}\right.\]
and
 \[\p_n^{\wt{P}_0}\phi|_{z=-1+\eps^2 b}{\eqdefa} n\cdot
\wt{P}_0\na\phi|_{z=-1+\eps^2 b}\quad\hbox{with}\quad \wt{P}_0=\left(\begin{matrix}\eps & 0 & 0\\
0&\eps^3&0\\
0& 0&1
\end{matrix}\right) \]
Let $\psi(t,x_h)\eqdefa
\phi|_{z=\eps^2\zeta}=\phi(t,x_h,\eps^2\zeta)$. Then similar to
\eqref{eq:Hamiltonian form-non}, the new dimensionless system of
$(\phi, \zeta)$ can be reformulated as a system of $(\psi, \zeta)$:
\beq\label{eq:Hamiltonian form-non-5kp} \left\{
\begin{array}{ll}
\p_t \zeta-\frac 1 {{\eps}}\wt G[{\eps^2}\zeta]\psi=0, \\
\p_t \psi+\zeta+\frac{{\eps^2}}{2}|\wt\na^{\eps}_h
\psi|^2-\f12{\eps}^3 {(\frac 1{\eps} \wt
G[{\eps^2}\zeta]\psi+{\eps^2}\wt\na^{\eps}_h\zeta\cdot
\wt\na^{\eps}_h \psi)^2}/{(1+{\eps}^5|\wt\na^
\eps_h\zeta|^2)}\\
\ =\alpha\eps\wt\na^{\eps}_h \cdot \big({{\wt\na^{\eps}_h
\zeta}}/{{\sqrt{1+{\eps}^5|\wt\na^{\eps}_h \zeta|^2}}}\big),\\
\zeta|_{t=0}=\zeta_0^\eps,\quad \psi|_{t=0}=\psi_0^\eps.
\end{array}\right.
\eeq

To describe the function space for the initial data such that
\eqref{eq:Hamiltonian form-non-5kp} has a unique solution on
$[0,{T}/{\eps^2}],$ we need to modify Definition \ref{def1.1} as
below:

\begin{defi}\label{def1.2}
{\sl We define the space $\wt{\frak{X}}^s$ as
\[
\wt{\frak{X}}^s{\eqdefa}\left\{U=(\zeta,\psi)^T: \zeta\in
H^{2s+1}(\R^2),\na_h \psi\in
 H^{2s-\f12}(\R^2)^2\right\}
\]
endowed with the semi-norm
\[
|U|_{\wt{\frak{X}}^s}{\eqdefa}\sqrt{\eps}|\zeta|_{\wt
H^{2s+1}_\eps}+|\zeta|_{\wt
H^{2s}_\eps}+\sqrt{\eps}|\wt\na_h^{\eps}\zeta|
_{H^{s}}+|\zeta|_{H^s}+|\wt{\mathfrak P}\psi|_{\wt
H^{2s}_\eps}+|\wt{\mathfrak P}\psi|_{H^{s}}
\]
for a new regularizing Poisson operator $\wt{\mathfrak {P}}\eqdefa
{|\wt D^{\eps}_{h}|}/{(1+\sqrt{\eps}|\wt D^{\eps}_{h}|)^\frac12}$
with  $\wt D^{\eps}_h=\f 1i\wt\na^\eps_h$,  and $\wt H^s_\eps(\R^2)$
is the space of tempered distributions $v$ so that \beq |v|_{\wt
H^s_\eps}\eqdefa |(1+|\wt D_h^\eps|^2)^{\f s2}v|_{L^2}<\infty. \eeq}
\end{defi}

Our second main result is as follows.

\bthm{Theorem}\label{thm:5-KP approximation} ({\bf Degenerate case})
{\sl Let $\al=\f13+\eps\tht$, $\theta\ge 0$ fixed and $s\ge m_0$ for
some $m_0\in (9,10).$ Assume that there exists $P>D>0$ such that for
all $b\in H^{2s+2P+1}(\R^2)$ and bounded initial data
$(\zeta^\eps_0,\psi^\eps_0) \in \wt{\frak X}^{S+P}$ satisfying
\[
\inf_{\R^2}\bigl(1+{\eps^2}\zeta^\eps_0-{\eps^2}
b\bigr)>0\quad\hbox{uniformly for}\quad \eps\in(0,1).
\]
 Then there exits $T>0$ such that
(\ref{eq:Hamiltonian form-non-5kp}) has a unique family of solutions
$(\zeta_\eps,\psi_\eps)_{0<\eps<1}$ on $[0,{T}/{\eps^2}]$ which
satisfy
 $(\zeta_\eps)_{0<\eps<1},
(\p_x\psi_\eps)_{0<\eps<1}$, and $({\eps}\p_y\psi_\eps)_{0<\eps<1}$
are uniformly bounded in $C([0,T/\eps^2];H^{s+D-\f12}(\R^2)).$ }
 \ethm

\begin{rmk}\label{rmk1.2} In fact, these two theorems above are two
particular results of a more general existence theorem. First of
all, define as in \cite{Lannes-Inven} that\[ \epsilon=\f ad,\quad
\mu=\f{d^2}{\lam^2},\quad \beta=\f Bd,
\] and set the dimensionless variable(with prime) as below\beq \label{scaling3}
\begin{split}
&x=\lam x',\quad y=\f{\lam}{\gamma}y',\quad z=dz',\quad
t=\f{\lam}{\sqrt{d}}t',\\
& \zeta=a\zeta',\quad \phi=\f{a}{\sqrt{d}}\lam\phi',\quad
b=Bb',\quad \psi=\f{a}{\sqrt{d}}\lam\psi', \end{split} \eeq One can
derive a more general water-wave system of $(\psi, \zeta)$:
\beq\label{eq:Hamiltonian form-general} \left\{
\begin{array}{ll}
\p_t \zeta-\frac 1 {{\mu}} G[{\epsilon}\zeta]\psi=0, \\
\p_t \psi+\zeta+\frac{{\eps^2}}{2}|\na^{\gamma}_h
\psi|^2-\f12{\epsilon\mu {(\frac 1{\mu}
G[{\epsilon}\zeta]\psi+{\epsilon}\na^{\gamma}_h\zeta\cdot
\na^{\gamma}_h \psi)^2}}/{(1+\epsilon^2\mu\na^
\gamma_h\zeta|^2)}\\
\ =\alpha\mu\na^{\gamma}_h \cdot \big({{\na^{\gamma}_h
\zeta}}/{{\sqrt{1+\epsilon^2\mu|\na^{\gamma}_h \zeta|^2}}}\big),\\
\zeta|_{t=0}=\zeta_0^\eps,\quad \psi|_{t=0}=\psi_0^\eps.
\end{array}\right.
\eeq with $\psi(t,x_h)\eqdefa
\phi|_{z=\epsilon\zeta}=\phi(t,x_h,\epsilon\zeta)$,
$\na^\gamma_h\eqdefa(\p_x,\gamma\p_y)^T$ and the nondimensionalized
Dirichlet-Neumann operator $ G[\epsilon\zeta]$ defined by
\beq\label{newdi} G[\epsilon\zeta]\psi:=
\bigl(-{\mu}\na^{\gamma}_h({\epsilon}\zeta)\cdot
\na^{\gamma}_h\phi+\p_z\phi\bigr)|_{z={\epsilon}\zeta},\eeq with
$\phi$ solving  \[ \left\{\begin{array}{ll}
\p_z^2\phi+\mu\p_x^2\phi+\gamma^2\mu\p_y^2\phi=0,\quad-1+\epsilon b<z<{\epsilon}\zeta,\\
\phi|_{z=\epsilon\zeta}=\psi,\quad\p_n^{{P}_0}\phi|_{z=-1+\epsilon
b}=0,
\end{array}\right.\]
and
 \[\p_n^{{P}_0}\phi|_{z=-1+\epsilon b}{\eqdefa} n\cdot
{P}_0\na\phi|_{z=-1+\epsilon b}\quad\hbox{with}
\quad {P}_0=\left(\begin{matrix}\mu & 0 & 0\\
0&\gamma^2\mu &0\\
0& 0&1
\end{matrix}\right). \]

We can have a large-time existence result similar to
Theorem\ref{thm:KP approximation} for solutions to the general
system(\ref{eq:Hamiltonian form-general}) on time interval $[0,\f
T\epsilon]$ following the proof of Theorem\ref{thm:KP
approximation}. Then Theorem\ref{thm:KP approximation} and Theorem\ref{thm:5-KP
approximation} are indeed two particular results of this result. In
fact, one can take
\[\epsilon=\mu=\eps,\quad \gamma=\sqrt \epsilon\] in system(\ref{eq:Hamiltonian
form-general}) to arrive at Theorem\ref{thm:KP approximation}, and
one can take \[\epsilon=\mu^2=\eps^2,\quad \gamma=\sqrt
\epsilon=\eps\] to arrive at Theorem\ref{thm:5-KP approximation}.

\end{rmk}

\begin{rmk}\label{rmk1.3}
Let \beno \zeta^{KP}_\eps(t,x,y)\eqdefa \f1
{\sqrt{2}}\big(\zeta_+(\eps^2 t,x-t,y)-\zeta_-(\eps^2 t,x+t,y)\big),
\eeno where $\zeta_\pm(\tau,X,y)$ solve the uncoupled fifth order KP
equations \beno (KP^{5\textrm{th }})^\pm\quad \p_\tau
\zeta_\pm+\f12\p_{X}^{-1}\p_Y^2\zeta_\pm\mp\f \tht
2\p_{X}^3\zeta_\pm\pm\f1{90}\p_{X}^5\zeta_\pm+\f{3\sqrt{2}}4\zeta_\pm\p_{X}\zeta_\pm=0.
\eeno We shall prove in \cite{MZZ2} that: under the assumptions in
Theorem \ref{thm:5-KP approximation}, we assume moreover
\beno\label{datac} &&\lim_{\eps\rightarrow
0}|\zeta^\eps_0-\zeta_0|_{H^{s+D-\f12}\cap
\p_xH^{s+D-\f12}}=0\quad\hbox{and}\quad
|\p_x\psi^\eps_0-\p_x\psi_0|_{H^{s+D-\f12}\cap \p_xH^{s+D-\f12}}=0
\eeno with $(\p_x\psi^0,\zeta^0)\in H^{s+D-\f12}(\R^2)\ \cap\
\p_xH^{s+D-\f12}(\R^2)$ and $\ (\p^2_y\p_x\psi^0, \p^2_y\zeta^0)\in\
$ $\p^2_xH^{s+D-6}(\R^2)$. Then $(KP^{5\textrm{th }})_\pm$ with
initial data $(\p_x\psi^0\pm\zeta^0)/\sqrt{2}$ has a unique solution
$\zeta_\pm\in C([0,T];$ $H^{s+D-\f12}(\R^2))$. Furthermore, there
holds\beq\label{conv5} \lim_{\eps\rightarrow
0}|\zeta_\eps-\zeta_\eps^{KP}|_{L^\infty([0,T/\eps^2]\times
\R^2)}=0. \eeq
\end{rmk}

\subsection{ Scheme of the proof and organization of the paper and notations} In
Section 2, we shall present various product laws and commutator
estimates in the scaled Sobolev spaces; We  provide  uniform
estimates for the solutions of scaled Laplacian equations in the
Section 3; While in Section 4, we modify some results from
\cite{Lannes-JFA} on the calculus of pseudo-differential operators
with rough symbols; We shall study the Dirichlet-Neumann operator in
Section 5; With the preparation in the above sections, we shall
prove large-time uniform estimates for the solutions of the
linearized system of \eqref{eq:Hamiltonian form-non}, which is the
crucial step in the proof of the large-time wellposedness result for
\eqref{eq:Hamiltonian form-non} in Section 7.  In the appendix, we
shall present a variance of Nash-Moser iteration Theorem in
\cite{Lannes-IUMJ}, which has been used in Section 7.\\

Let us complete this section by some notations, which we shall use
throughout the paper. We shall always denote by
$C(\lambda_1,\lambda_2,\cdots)$ a generic positive constant which is
a nondecreasing function of its variables,  $t_0$ a  fixed  number
in $(1,2)$, and $m_0\eqdefa t_0+8$. We denote $\na_h\eqdefa
(\p_x,\p_y),$ $\na^\eps_h \eqdefa (\p_x,\sqrt{\eps}\p_y),$
 the scaled
horizontal derivatives,  $D^{\eps}_h \eqdefa \frac1{i}\na^\eps_h$,
$\na \eqdefa (\na_h,\p_z),$ $\na^{\eps} \eqdefa (\sqrt{\eps}
\na^{\eps}_h, \p_z)$ the scaled full derivative,
$\xi^\eps=(\xi_1,\sqrt{\eps}\xi_2)$ in $\R^2$, and $|D^{\eps}_h|$
the Fourier multiplier with the symbol $|\xi^\eps|$. $\Lambda$ and
$\Lambda_{\eps}$ are Fourier multiplier with the symbol
$(1+|\xi|^2)^\f12$ and $(1+|\xi^\eps|^2)^\f12$ respectively. We
denote $|\cdot|_p$ the $L^p(\R^2)$ norm, $\|\cdot\|_p$ the
$L^p(\mathcal{S})$ norm with $\mathcal{S}=\R^2\times [-1,0],$
$H^s(\R^2)$  the usual Sobolev spaces with the norm $|f|_{H^s}
\eqdefa |\Lambda^sf|_2,$ $|f|_{H^s_\eps} \eqdefa |\Lambda^s_\eps
f|_2$  the norm in the scaled Sobolev spaces $H^s_\eps(\R^2),$ and
the regularizing Poisson operator $\mathfrak {P} \eqdefa
{|D^{\eps}_{h}|}/{(1+\sqrt{\eps}|D^{\eps}_{h}|)^\frac12}$. We use
$\p_xH^s(\R^2)$ to refer to the space of all the distributions $v$
such that there exists $\wt v\in H^s(\R^2)$ with $\p_x \wt v=v$, and
we write $|v|_{\p_xH^s}\eqdefa |\wt v|_{H^s}$. We define similarly
for $\p^2_xH^s(\R^2)$. Finally we shall always use the convention
that
$$
A_s=B_s+\langle C_s\rangle_{s>s_0}
=\left\{\begin{array}{ll}B_s&\quad \textrm{if}\quad s\le s_0,\\B_s+C_s&\quad \textrm{if}\quad s>s_0.
\end{array}\right.
$$

\renewcommand{\theequation}{\thesection.\arabic{equation}}
\setcounter{equation}{0}

\section{Preliminaries}

Recall that $|f|_{H^s_{\eps}}=|\Lambda_{\eps}^s f|_2$ is the norm of
the scaled Sobolev space $H^s_{\eps}(\R^2).$  It is easy to observe
by a scaling argument that

\begin{lem}\label{lem2.1} {\sl Let $r, s\ge 0$.  There exists an
$\eps$ independent constant $C$  such that
\begin{itemize}
 \item[(i)]
If $f\in H^s(\R^2)$ and $\frac1p=\frac12-\frac s2,$ \beq
\label{lem:Sobolev Embedding}\begin{split} & |f|_p\leq C \eps^{-\f
s4} |f|_{H^s_{\eps}}\quad\mbox{for}\ \ 0\leq s<1, \
 \textrm{and}\quad |f|_\infty\leq
C\eps^{-\frac14}|f|_{H^s_{\eps}}\quad \textrm{for } s>1;
\end{split} \eeq

 \item[(ii)]
If $f,g\in H^s\cap H^{t_0}(\R^2)$, \beq\label{lem:product estimate}
\begin{split}
&|fg|_{H^s_{\eps}}\leq C\big(|f|_{H^s_{\eps}}|g|_\infty+|f|_\infty|g|_{H^s_{\eps}}\big),\\
&|fg|_{H^s_{\eps}}\leq
C\eps^{-\f14}\Big(|f|_{H^s_{\eps}}|g|_{H^{t_0}_\eps}
+\langle|f|_{H^{t_0}_\eps}|g|_{H^s_{\eps}}\rangle_{s>t_0}\Big);\end{split}
\eeq

 \item[(iii)]
If $F\in C^\infty(\R)$ with  $F(0)=0$ and $f\in H^s\cap
L^\infty(\R^2)$,  \beq \label{lem:composition} |F(f)|_{H^s_\eps}\le
C(|f|_\infty)|f|_{H^s_\eps}; \eeq

 \item[(iv)] If $f\in H^{s+r}\cap H^{t_0+1}(\R^2)$ and $g\in H^{s+r-1}\cap
H^{t_0}(\R^2)$,  \beq \label{lem:commutator estimate}
\begin{split} |\big[\Lambda^s_{\eps}, f\big]g|_{H^r_\eps}\leq& C
\big(|\na^{\eps}_h
f|_{H^{s+r-1}_\eps}|g|_{\infty}+|\na^{\eps}_h f|_{\infty}|g|_{H^{s+r-1}_\eps}\big),\\
|\big[\Lambda^s_{\eps}, f\big]g|_{H^r_\eps}\leq &C
\eps^{-\f14}\Big(|\na^{\eps}_h
f|_{H^{t_0}_\eps}|g|_{H^{s+r-1}_\eps}+ \langle|\na^{\eps}_h
f|_{H^{s+r-1}_\eps}|g|_{H^{t_0}_\eps}\rangle_{s>t_0+1-r}\Big).
\end{split}
\eeq
\end{itemize}
 } \end{lem}

\begin{rmk}\label{rmk2.1} \eqref{lem:commutator estimate} still holds  with $\Lambda^s_{\eps}$
being replaced by $|D^{\eps}_{h}|^s$ for $s\in 2\N$. \end{rmk}

\begin{lem}\label{lem:trace theorem} {\sl Let $s\in \R$ and $\na
u\in L^2([-1,0];H^{s-1}(\R^2))$ with $u(x_h,0)=0$. Then for any
$\frak{z}\in [-1,0]$, one has
\[
|u|_{z=\frak{z}}|_{H^{s-\frac12}_\eps} \leq C\eps^{-\f14}\|
\na^{\eps}\Lambda^{s-1}_\eps u\|_2
\]
for the constant $C$ independent of $\eps$.} \end{lem}

\begin{proof} Let $\ga\eqdefa\sqrt{\eps}.$ It is easy to observe that
 \beq \label{lem2.1a}
\Lambda_{\eps}^su(x_h,z)=\ga^{-1}\Lambda^s u_{\ga}\big(x,\f y
\ga,z\big), \quad u_\ga(x,y,z)=\ga u(x,\ga y,z), \eeq  we have
\ben\label{eq:trace-est1}
|u(\cdot,\frak{z})|_{H^{s-\frac12}_\eps}=\ga^{-\f12}|\Lambda^{s-\f12}u_\ga(\cdot,\frak{z})|_2
=\ga^{-\f12}|\Lambda^{s-1}u_\ga(\cdot,\frak{z})|_{H^\f12}. \een
Whereas  as $u(x_h,0)=0,$ applying Cauchy-Schwartz inequality gives
\begin{align*}
&(1+|\xi|^2)^\f12|\widehat{\Lambda^{s-1}u}_\ga(\xi,\frak{z})|^2\\
&=-\int_{\frak{z}}^0(1+|\xi|^2)^\f12\widehat{\Lambda^{s-1}u}_\ga(\xi,z)\p_z\widehat{\Lambda^{s-1}u}_\ga(\xi,z)dz\\
&\le\Big(\int_{-1}^0(1+|\xi|^2)|\widehat{\Lambda^{s-1}u}_\ga(\xi,z)|^2dz\Big)^\f12
\Big(\int_{-1}^0|\p_z\widehat{\Lambda^{s-1}u}_\ga(\xi,z)|^2dz\Big)^\f12,
\end{align*}
which implies that
\begin{align*}
|\Lambda^{s-1}u_\ga(\cdot,t)|_{H^\f12}^2&\le \|\Lambda^{s}u_\ga\|_2\|\p_z\Lambda^{s-1}u_\ga\|_2
=\ga\|\Lam^s_{\eps} u\|_2\|\p_z\Lambda^{s-1}_\eps u\|_2\\
&\le \f12\ga^2\|\Lam^s_{\eps} u\|_2^2+\f12 \|\p_z\Lambda^{s-1}_\eps
u\|_2^2,
\end{align*}
from which and (\ref{eq:trace-est1}), we deduce that
\beno
|u(\cdot,t)|_{H^{s-\frac12}_\eps}\le
C\ga^{-\f12}\big(\ga\|\Lam^s_{\eps} u\|_2+\|\p_z\Lambda^{s-1}_\eps
u\|_2\big) \le C\ga^{-\f12}\|\na^{\eps}\Lambda^{s-1}_\eps u\|_2.
\eeno where in the last step, we used again the fact that
$u(x_h,0)=0$ such that $\|\Lambda^{s-1}_\eps u\|_2\le
C\|\p_z\Lambda^{s-1}_\eps u\|_2$.\end{proof}

We introduce the following scaled 2nd-order elliptic operator
\beq
\label{eq:ellipticoperator}
\frak{\Delta}_{\eps}(a){\eqdefa}|D^{\eps}_{h}|^2-\frac{\eps^3
\bigl((\p_xa)^2D^2_x+2\eps\p_xa\p_yaD_xD_y+\eps^2(\p_ya)^2D^2_y\bigr)
}{1+\eps^3|\na_h^{\eps} a|^2},
\eeq
which is a part of the linearized opertor for the nonlinear system
corresponding to the surface tension term.

\begin{lem}\label{lem:elliptic operator} {\sl Let $ s\in [0,1]$.
 Then for $k\in \N, s\ge 0$
and $f\in H^{2k+s}\cap H^{t_0}(\R^2)$, we have \beno
&&|\fd_{\eps}(a)^k f|_{H^s_{\eps}}\leq
M(a)\big(|f|_{H^{2k+s}_\eps}+|f|_{H^{t_0}_\eps}
|\na^{\eps}_h a|_{H^{2k+s}_\eps}\big),\\
&&|\fd_{\eps}(a)^k f|_{H^s_{\eps}}\geq
M(a)^{-1}|f|_{H^{2k+s}_\eps}-M(a) \big(1+|\na^{\eps}_h
a|_{H^{2k+s}_\eps}\big)|f|_{H^{t_0}_\eps}.\eeno Here and in what
follows $M(a)$ always denotes a constant depending on
$|a|_{H^{m_0}_\eps}$.} \end{lem}

\begin{proof}  One  can  deduce this lemma from
Lemma 3.5 and Lemma 3.6 in \cite{MZ} by a  scaling argument. For
completeness,  we shall  present the outline of the proof here.
Indeed for the first estimate, one only need to use Proposition
\ref{prop:psdo-estimate} and an interpolation argument. Now we focus
on the sketch of the proof for the second estimate.

We use an inductive argument on $k$. Let us first deal with the case
when  $k=1$. Toward this, we write $\fd_{\eps}(a)$ as
\begin{eqnarray*}
\begin{split}
\fd_{\eps}(a)=&\sum_{i,j=1,2}[\delta_{ij}-\,(1+\eps^3|\na^\eps_h
a|^2)^{-1}\eps^3\p^\eps_i a\p^\eps_j
a]D^\eps_iD^\eps_j\\
\eqdefa&\sum_{i,j=1,2}g_{ij}(\eps^{\f32}\na^\eps_h
a)D^\eps_iD^\eps_j\quad\mbox{for}\quad (D_1^\eps,D_2^\eps)=D_h^\eps.
\end{split}
\end{eqnarray*}
Then we have \beno
\begin{split}
(\fd_\eps(a) f,f)=&-\sum_{i,j=1,2}(g_{ij}(\eps^{\f32}\na^\eps_h a)\p^\eps_i\p^\eps_jf,f)\\
=&\sum_{i,j=1,2}(g_{ij}(\eps^{\f32}\na^\eps_h
a)\p^\eps_if,\p^\eps_jf)
+(\p_jg_{ij}(\eps^{\f32}\na^\eps_h a)\p^\eps_if,f)\\
\ge& M(a)^{-1}|\na^\eps_h f|_{L^2}^2-M(a)|f|_{L^2}|\na^\eps_h f|_{L^2}\\
\ge& M(a)^{-1}|f|_{H^1_\eps}^2-M(a)|f|_{L^2}^2. \end{split} \eeno
Whereas notice that \beno \begin{split} |(\fd_\eps(a)\na^\eps_h
f,\na^\eps_h f)| \le&|(\na^\eps_h\fd_\eps(a)f,\na^\eps_h f)|
+\sum_{i,j=1,2}|((\na^\eps_h g_{ij})D^\eps_iD^\eps_jf,\na^\eps_h f)|\\
\le& (|\fd_\eps(a)f|_{L^2}+M(a)|f|_{H^1_\eps})|\na^\eps_h
f|_{H^1_\eps}.\end{split} \eeno As a consequence, we obtain \beno
\begin{split}
M(a)^{-1}|\na^\eps_h f|^2_{H^1_\eps}&\le
(\fd_\eps(a)\na^\eps_h f,\na^\eps_h f)+M(a)|\na^\eps_h f|^2_{L^2}\\
& \le \f {M(a)^{-1}} 2|\na^\eps_h
f|^2_{H^1_\eps}+M(a)|\fd_\eps(a)f|^2_{L^2}+M(a)|f|^2_{L^2}.
\end{split} \eeno This ensures \[ |\fd_\eps(a)f|_{L^2}\ge
M(a)^{-1}|f|_{H^2_\eps}-M(a)|f|_{L^2}, \] from which and Kato-Ponce
type commutator estimate, we infer \beno
\begin{split} |\fd_\eps(a)f|_{H^s_\eps}\ge& |\fd_\eps(a)\Lambda^s_\eps
f|_{L^2}
-|[\Lambda^s_\eps,g_{ij}]D^\eps_iD^\eps_jf|_{L^2}\\
\ge&
M(a)^{-1}|f|_{H^{2+s}_\eps}-M(a)|f|_{H^s_\eps}-M(a)(|\na^\eps_ha|_{H^{2+s}_\eps}|f|_{H^{t_0}_\eps}+|f|_{H^{s+1}_\eps}),
\end{split}
\eeno that is \[ |\fd_\eps(a)f|_{H^s_\eps} \ge
M(a)^{-1}|f|_{H^{2+s}_\eps}-M(a)(1+|\na^\eps_ha|_{H^{2+s}_\eps})|f|_{H^{t_0}_\eps}.
\] Now we  assume inductively that for $1\le \ell\le k-1$
\[ |\fd_\eps(a)^\ell f|_{H^{s}_\eps}\ge
M_1(a)^{-1}|f|_{H^{2\ell+s}_\eps}-M_1(a)(1+|\na^\eps_h
a|_{H^{2\ell+s}_\eps})|f|_{H^{t_0}_\eps}. \] Then we deduce from the
induction assumption that \beno
\begin{split}
|\fd_\eps(a)^kf|_{H^{s}_\eps}=&|\fd_\eps(a)^{k-1}\fd_\eps(a)f|_{H^{s}_\eps}
\ge
M(a)^{-1}|\fd_\eps(a)f|_{H^{2(k-1)+s}_\eps}\\
& -M(a)(1+|\na^\eps_h
a|_{H^{2(k-1)+s}_\eps})|\fd_\eps(a)f|_{H^{t_0}_\eps}),
\end{split}\eeno which together with an interpolation argument
implies the second inequality of  the lemma.\end{proof}

\begin{lem}\label{lem:elliptic operator-commutator} {\sl Let $k\in
\N, s\ge 0$. Then for any $f\in H^{2k+s}\cap H^{t_0+2}(\R^2)$ and
$g\in H^{2k+s}\cap H^{t_0+1}(\R^2)$, there hold \beno
\begin{split}
&\bigl|[\fd_{\eps}(a)^k, f]g\bigr|_{H^s_\eps}\leq
M(a)\Big(|f|_{H^{t_0+2}}|g|_{H^{2k+s-1}_\eps}+
|f|_{H^{2k+s}_\eps}|g|_{H^{t_0}}+|a|_{H^{2k+s}_\eps}|f|_{H^{t_0+2}_\eps}|g|_{H^{t_0+1}_\eps}\Big),\\
&\bigl|[\fd_{\eps}(a)^k, \na_h^{\eps}]g\bigr|_{H^s_\eps}\leq
\eps^{2}M(a)\Big(|\na_h^{\eps} g|_{H^{2k+s-1}_\eps}
+|a|_{H^{2k+s+1}_\eps}|\na_h^{\eps} g|_{H^{t_0}_\eps}\Big),
\\
&\bigl|[\fd_{\eps}(a)^k, \mathfrak{P}]g\bigr|_{H^s_\eps}\le
\eps^2M(a)(|\na^\eps_h
g|_{H^{2k+s-1}_\eps}+|a|_{H^{2k+s+2}_\eps}|\na^\eps_h
g|_{H^{t_0}_\eps})
\end{split} \eeno }
\end{lem}

\begin{proof} Firstly let's focus on the proof for the first inequality. Indeed thanks to Lemma
\ref{lem2.1} and  Sobolev inequality, it reduces to prove that
\begin{align}\label{eq:elliptic operator-est1}
\begin{split}
\bigl|[\fd_{\eps}(a)^k-|D^{\eps}_{h}|^{2k}, f]g\bigr|_{H^s_\eps} &\
\leq M(a)\Big(|f|_{H^{t_0+2}_\eps}|g|_{H^{2k+s-1}_\eps}+
|f|_{H^{2k+s}_\eps}|g|_{H^{t_0}_\eps}\\
&\quad+|a|_{H^{2k+s}_\eps}|f|_{H^{t_0+2}_\eps}|g|_{H^{t_0+1}_\eps}\Big){\eqdefa}I_k(s,f,g).
\end{split}
\end{align} We shall use an  inductive argument on $k$ to prove
\eqref{eq:elliptic operator-est1}.
 We first infer from Lemma
\ref{lem2.1} that
\begin{align}\label{eq:elliptic operator-est2}
\begin{split}
\bigl|[\fd_{\eps}(a)-|D^{\eps}_{h}|^2, f]g\bigr|_{H^s_\eps} &\ \le
M(a)\Big(|f|_{H^{t_0+1}_\eps}|g|_{H^{s+1}_{\eps}}+
\langle|f|_{H^{s+2}_\eps}|g|_{H^{t_0}_\eps}\rangle_{s>t_0-1}\\
&\qquad
+\langle|a|_{H^{s+1}_\eps}|f|_{H^{t_0+2}_\eps}|g|_{H^{t_0+1}_\eps}\rangle_{s>t_0}\Big).
\end{split}
\end{align}
This shows (\ref{eq:elliptic operator-est1}) for $k=1$. Now we
assume that (\ref{eq:elliptic operator-est1})
 holds for $k\le \ell-1$. To prove the case of $k=\ell$, we write
\beno
\begin{split}
\big[\fd_{\eps}(a)^\ell-|D^{\eps}_{h}|^{2\ell}, f\big]g
&=\fd_{\eps}(a)\big[\fd_{\eps}(a)^{\ell-1}-|D^{\eps}_{h}|^{2(\ell-1)}, f\big]g+(\fd_{\eps}(a)-|D^{\eps}_{h}|^2)\big[|D^{\eps}_{h}|^{2(\ell-1)}, f\big]g\nonumber\\
&\ +\big[\fd_{\eps}(a)-|D^{\eps}_{h}|^{2},
f\big]\fd_{\eps}(a)^{\ell-1}g+\big[|D^{\eps}_{h}|^2,
f\big](\fd_{\eps}(a)^{\ell-1}-|D^{\eps}_{h}|^{2(\ell-1)})g.
\end{split}
\eeno We first get by applying Lemma \ref{lem:elliptic operator} and
the induction assumption that \beno
\begin{split}
&\bigl|\fd_{\eps}(a)\big[\fd_{\eps}(a)^{\ell-1}-|D^{\eps}_{h}|^{2(\ell-1)}, f\big]g\bigr|_{H^s_\eps}\\
&\ \le
M(a)\Big(|\big[\fd_{\eps}(a)^{\ell-1}-|D^{\eps}_{h}|^{2(\ell-1)},
f\big]g|_{H^{2+s}_\eps}+|a|_{H^{s+3}_\eps}|\big[\fd_{\eps}(a)^{\ell-1}-|D^{\eps}_{h}|
^{2(\ell-1)}, f\big]g|_{H^{t_0}_\eps}\Big)\\
&\ \le
I_\ell(s,f,g)+M(a)|a|_{H^{s+3}_\eps}\Big(|f|_{H^{t_0+2}_\eps}|g|_{H^{2(\ell-1)+t_0-1}_\eps}
+|f|_{H^{2(\ell-1)+t_0}_\eps}|g|_{H^{t_0}_\eps}
\\
&\quad+|a|_{H^{2(\ell-1)+t_0}_\eps}|f|_{H^{t_0+2}_\eps}|g|_{H^{t_0+1}_\eps}\Big)\\
&\ \le I_\ell(s,f,g), \end{split} \eeno where in the last inequality
we used the following interpolation inequalities \beno
&&|a|_{H^{s+3}_\eps}\le
|a|_{H^{t_0+1}_\eps}^{1-\tht}|a|_{H^{2\ell+s}_\eps}^\tht,\quad
\tht=\f{s-t_0+2}{s+2\ell-t_0-1},\\
&&|a|_{H^{2(\ell-1)+t_0}_\eps}\le
|a|_{H^{t_0+1}_\eps}^{1-\tht}|a|_{H^{2\ell+s}_\eps}^\tht,\quad
\tht=\f{2(\ell-1)-1}{s+2\ell-t_0-1},\\
&&|g|_{H^{2(\ell-1)+t_0-1}_\eps}\le
|g|_{H^{t_0}_\eps}^{1-\tht}|g|_{H^{s+2\ell-1}_\eps}^{\tht},\quad
\tht=\f{2(\ell-1)-1}{s+2\ell-t_0-1},\\
&&|f|_{H^{2(\ell-1)+t_0}_\eps}\le
|f|_{H^{t_0+1}_\eps}^{1-\tht}|f|_{H^{s+2\ell}_\eps}^{\tht},\quad
\tht=\f{2(\ell-1)-1}{s+2\ell-t_0-1}, \eeno such that for example,
\beno |a|_{H^{s+3}_{\eps}}|g|_{H^{2(\ell-1)+t_0-1}_\eps}\le
M(a)\big(|a|_{H^{2\ell+s}_\eps}|g|_{H^{t_0+1}_\eps}
+|g|_{H^{2\ell+s-1}_\eps}\big). \eeno Similarly applying Lemma
\ref{lem2.1} ensures that \beno
\bigl|(\fd_{\eps}(a)-|D^{\eps}_{h}|^2)\big[|D^{\eps}_{h}|^{2(\ell-1)},
f\big]g\bigr|_{H^s_{\eps}}\le I_\ell(s,f,g). \eeno Thanks to
(\ref{eq:elliptic operator-est2}), one has
\begin{align*}
\begin{split}
& \bigl|\big[\fd_{\eps}(a)-|D^{\eps}_{h}|^2,
f\big]\fd_{\eps}(a)^{\ell-1}g\bigr|_{H^s_\eps}\\
&\ \le M(a)\Big(|f|_{H^{t_0+1}_\eps}
|\fd_{\eps}(a)^{\ell-1}g|_{H^{s+1}_{\eps}}+
\langle|f|_{H^{s+2}_\eps}|\fd_{\eps}(a)^{\ell-1}g|_{H^{t_0}_\eps}\rangle_{s>t_0-1}\nonumber\\
&\qquad
+\langle|a|_{H^{s+1}_\eps}|f|_{H^{t_0+2}_\eps}|\fd_{\eps}(a)^{\ell-1}g|_{H^{t_0+1}_\eps}\rangle_{s>t_0}\Big),
\end{split}
\end{align*}
while it follows from Lemma \ref{lem:elliptic operator} and an
interpolation argument that for $s>t_0-1$, \beno
\begin{split}
|f|_{H^{s+2}_\eps}|\fd_{\eps}(a)^{\ell-1}g|_{H^{t_0}_\eps} \le &
M(a)|f|_{H^{s+2}_\eps}\big(|g|_{H^{2(\ell-1)+t_0}_\eps}+|a|_{H^{2(\ell-1)+t_0+1}_\eps}|g|_{H^{t_0}_\eps}\big)\\
\le& I_\ell(s,f,g),\end{split} \eeno and for $s>t_0$, \beno
\begin{split}
|a|_{H^{s+1}_{\eps}}&|f|_{H^{t_0+2}_\eps}|\fd_{\eps}(a)^{\ell-1}g|_{H^{t_0+1}_\eps}\\
&\le
M(a)|a|_{H^{s+1}_\eps}|f|_{H^{t_0+2}_\eps}|\big(|g|_{H^{2(\ell-1)+t_0+1}_\eps}
+|a|_{H^{2\ell+t_0}_\eps}|g|_{H^{t_0}_\eps}\big) \le I_\ell(s,f,g).
\end{split}
\eeno As a consequence, we obtain \beno
\bigl|\big[\fd_{\eps}(a)-|D^{\eps}_{h}|^2,
f\big]\fd_{\eps}(a)^{\ell-1}g\bigr|_{H^s_\eps}\le I_\ell(s,f,g).
\eeno Similarly we can deduce from \eqref{lem:commutator estimate}
and Lemma \ref{lem:elliptic operator} that \beno
\bigl|\big[|D^{\eps}_{h}|^2,
f\big](\fd_{\eps}(a)^{\ell-1}-|D^{\eps}_{h}|^{2(\ell-1)})g\bigr|_{H^s_\eps}\le
I_\ell(s,f,g). \eeno This proves (\ref{eq:elliptic operator-est1})
for $k=\ell$, which proves  the first estimate of the lemma.

For the second inequality, one can use a similar inductive argument
 to prove it. And it's almost the same for the third
inequality by noticing that $\fd_{\eps}(a)=f(\eps^3(\na^\eps_h
a)^2)(D^\eps_h)^2$(formally) and one has \beno \bigl|[\fd_{\eps}(a),
\mathfrak P]g\bigr|_{H^s_\eps}&=&\bigl|[f(\eps^3(\na^\eps a)^2),
\mathfrak
P](D^\eps_h)^2 g\bigr|_{H^s_\eps}\\
&\le& \eps^2 M(a)(||D^\eps_h|^2
g|_{H^s_\eps}+|a|_{H^{s+4}_\eps}||D^\eps_h|^2 g|_{L^2})\\
&\le& \eps^2M(a)(||D^\eps_h|
g|_{H^{s+1}_\eps}+|a|_{H^{s+4}_\eps}||D^\eps_h|g|_{H^1_\eps}). \eeno
This completes the proof of the lemma.
\end{proof}

Let us conclude this section by recalling a result from
\cite{Lannes-Inven} on the anisotropic Poisson regularization.

\begin{lem}\label{lem:Poisson regularization} {\sl Let $\chi\in
C^\infty_c(\R)$ with $\chi(0)=1$, and define
$u^\dagger=\chi(z\sqrt{\eps} |D^{\eps}_{h}|)u$. Then for any  $\
s\in \R$, if $u\in H^{s-\f12}(\R^2)$, we have
\[
c_1
\Big|\frac1{(1+\sqrt{\eps}|D^{\eps}_{h}|)^\frac12}u\Big|_{H^s_\eps}\leq
\|\Lambda^s_{\eps} u^\dagger\|_2\leq c_2
\Big|\frac1{(1+\sqrt{\eps}|D^{\eps}_{h}|)^\frac12}u\Big|_{H^s_\eps},
\]
and if $u\in H^{s+\f12}(\R^2)$, we have
\[
c'_1\sqrt{\eps}|\mathfrak{P}u|_{H^s_{\eps}}\leq \|\Lambda^s_{\eps}
\na^{\eps}u^\dagger\|_2\leq c'_2\sqrt{\eps}
|\mathfrak{P}u|_{H^s_{\eps}}.
\]
Here $c_1,c_2,c_1'$ and $c_2'$ are positive constants depending only
on $\chi$.} \end{lem}

\renewcommand{\theequation}{\thesection.\arabic{equation}}
\setcounter{equation}{0}
\section{Elliptic estimates on the infinite strip}

In this section, we  consider the following boundary value problem
on the infinite strip \beq\label{eq:elliptic equation}
\left\{\begin{array}{ll}
\p_z^2\Phi+\eps\p_x^2\Phi+\eps^2\p_y^2\Phi=0,\quad-1+\eps b(x_h)<z<{\eps}\zeta(t,x_h),\\
\Phi|_{z={\eps}\zeta}=\psi,\quad\p_n\Phi|_{z=-1+\eps b}=0,
\end{array}\right.
\eeq under the assumption that \beq\label{ass:free surface}
 1+\eps\zeta-\eps b\ge
h_0\quad\mbox{for some}\quad h_0>0. \eeq We denote by $S$ a
diffeomorphism from $\mathcal{S}=\R^2\times[-1,0]$ to the fluid
domain $\Omega=\R^2\times[-1+\eps b(x_h),\eps \zeta(x_h)]$ so that
\beq\label{Transform}  S:\ (x_h,z)\in \mathcal{S}\mapsto
S(x_h,z)=(x_h,z+\sigma(x_h,z))\in\Om \eeq for $ \sigma(x_h,z)=-\eps
z b(x_h)+\eps(z+1)\zeta(x_h).$

Using this diffeomorphism $S$, the elliptic equation
(\ref{eq:elliptic equation}) can be equivalently formulated as an
elliptic problem with variable coefficients on the flat strip
 so that \beq\label{eq:elliptic equation on strip} \left\{\begin{array}{ll}
\na\cdot P^\eps[\sigma]\na u=0\qquad\hbox{in}\quad
\mathcal{S},\\
u|_{z=0}=\psi,\qquad\p_n^Pu|_{z=-1}=0,
\end{array}\right.
\eeq where $u=\Phi\circ S$ and $\p_n^P$ denotes the conormal
derivative associated with $P^\eps[\sigma]$, i.e., \beno
\p_n^Pu=-{\bf e_3}\cdot P^\eps[\sigma]\na u|_{z=-1}. \eeno Here
${\bf e_3}=(0,0,1)^T$. Moreover, we write
$$\na\cdot P^\eps[\sigma]\na=
\na^{\eps}\cdot (I+Q^\eps[\sigma])\na^{\eps}$$ with \beq\label{Qd}
Q^\eps[\sigma]=\left(\begin{matrix} \p_z\sigma &
0&-\sqrt{\eps}\p_x\sigma\\
0 & \p_z\sigma & -\eps\p_y\sigma\\
-\sqrt{\eps}\p_x\sigma & -\eps\p_y\sigma &
\frac{-\p_z\sigma+\eps(\p_x\sigma)^2+\eps^2(\p_y\sigma)^2}{1+\p_z\sigma}
\end{matrix}\right).\eeq

\no{\bf Notation 3.1}\,Throughout this paper, we shall always denote
$$M(\sigma){\eqdefa}C(\f1{h_0},|b|_{H^{m_0}},|\zeta|_{H^{m_0}})$$
to be a constant which is  a nondecreasing function to all
arguments.

To transform the Dirichlet boundary data $\psi$ in (\ref{eq:elliptic
equation on strip}) to be zero,  we are led to consider the
following elliptic  problem \beq\label{eq:elliptice equation-homo}
\left\{\begin{array}{ll} \na^{\eps}\cdot
(1+Q^\eps[\sigma])\na^{\eps}u=\na^{\eps}\cdot {\bf g}\quad\hbox{in}\quad \mathcal{S},\\
u|_{z=0}=0,\qquad\p_n^Pu|_{z=-1}=-{\bf e_3}\cdot {\bf g}|_{z=-1},
\end{array}\right.
\eeq with \beno \p_n^Pu|_{z=-1}{\eqdefa}-{\bf
e_3}\cdot(1+Q^\eps[\sigma])\na^{\eps}u|_{z=-1}. \eeno

\begin{prop}\label{prop:elliptic estimate} {\sl Let $s\geq 0$,
 and $\zeta, b\in H^{m_0} \bigcap H^{s+1}(\R^2)$ satisfy
(\ref{ass:free surface}). Then for all ${\bf g}\in C([-1,0];
H^s(\R^2)^3)$, (\ref{eq:elliptice equation-homo}) has  a unique
variational solution $u\in H^1(\mathcal{S})$ so that
\[
\|\Lambda^s_{\eps}\na^{\eps} u\|_2\leq
M(\sigma)\left(\|\Lambda^s_{\eps} {\bf
g}\|_2+\langle\|\Lambda^{t_0}_\eps {\bf g}\|_2|(\zeta,
b)|_{H^{s+1}_{\eps}}\rangle_{s>t_0+1}\right).
\]}
\end{prop}

\begin{proof}\, Thanks to (\ref{lem:Sobolev
Embedding})-(\ref{lem:commutator estimate}),  one can deduce
Proposition \ref{prop:elliptic estimate} by exactly following the
same line as the proof of Proposition 2.4 in \cite{Lannes-Inven}
(see also the proof of Proposition \ref{prop:elliptic
estimate-general} below), and we omit the details here.\end{proof}

\no{\bf Notation 3.2.} For $u\in H^{\f32}(\R^2)$, we define $u^b$ as
the solution of  \beq\label{u^b} \left\{\begin{array}{ll}
\na^{\eps}\cdot (1+Q^\eps[\sigma])\na^{\eps}u^b=0\quad\hbox{in}\quad \mathcal{S},\\
u^b|_{z=0}=u,\quad\p_nu^b|_{z=-1}=0.
\end{array}\right.
\eeq

As an immediate  corollary of Proposition \ref{prop:elliptic
estimate}, we obtain

\begin{col}\label{cor:elliptic estimate} {\sl Let  $s\geq 0,$  and $\zeta, b\in
H^{m_0} \bigcap H^{s+1}(\R^2)$ satisfy (\ref{ass:free surface}).
Then for any $u \in H^{s+\f12}(\R^2)$, one has
\[
\|\Lambda^s_{\eps}\na^{\eps} u^b\|_2\leq \sqrt{\eps}
M(\sigma)\Big(|\mathfrak P u|_{H^s_{\eps}}+\langle|\mathfrak P
u|_{H^{t_0}_\eps}|(\zeta,b)|_{H^{s+1}_{\eps}} \rangle_{s>t_0}\Big).
\]}
\end{col}

\begin{proof}\, Let $v {\eqdefa} u^b-u^\dagger$ with $u^\dagger$ being given by
Lemma \ref{lem:Poisson regularization}. Then $v$ solves \beno
\left\{\begin{array}{ll} \na^{\eps}\cdot
(1+Q^\eps[\sigma])\na^{\eps}v=-\na^{\eps}\cdot
(1+Q^\eps[\sigma])\na^{\eps}u^\dagger
\quad\hbox{in}\quad \mathcal{S},\\
u|_{z=0}=0,\qquad\p^P_nu|_{z=-1}={\bf e_3}\cdot
(1+Q^\eps[\sigma])\na^{\eps}u^\dagger|_{z=-1}.
\end{array}\right.
\eeno Thanks to (\ref{lem:Sobolev Embedding})-(\ref{lem:product
estimate}), we get by applying Proposition \ref{prop:elliptic
estimate} for ${\bf g}=-(1+Q^\eps[\sigma])\na^{\eps}u^\dagger$
 that \beno
\|\Lam^s_{\eps}\na^{\eps}v\|_2\le M(\sigma)\big(\|\Lam^s_{\eps}
\na^{\eps}u^\dagger\|_2+
\langle\|\Lambda^{t_0}_\eps\na^{\eps}u^\dagger\|_{2}|(\zeta,b)|_{H^{s+1}_{\eps}}\rangle_{s>t_0}\big),
\eeno which together with Lemma \ref{lem:Poisson regularization}
completes the proof of the corollary.\end{proof}

Besides \eqref{eq:elliptice equation-homo}, we also need to deal
with a more general elliptic  problem as follows
\beq\label{eq:elliptice equation-homo-general}
\left\{\begin{array}{ll} \na^{\eps}\cdot (1+Q^\eps[\sigma])\na^{\eps}u=f\quad\hbox{in}\quad \mathcal{S},\\
u|_{z=0}=0,\qquad\p_n^Pu|_{z=-1}=g.
\end{array}\right.
\eeq

\begin{prop}\label{prop:elliptic estimate-general} {\sl Under
the assumptions of Proposition \ref{prop:elliptic estimate}, for any
given $f\in L^2((-1,0);$ $ H^{s-1}(\R^2))$ and $g\in H^s(\R^2),$
(\ref{eq:elliptice equation-homo-general})  has  a variational
solution $u\in H^1(\mathcal{S})$ so that
\begin{align*}
\begin{split}
\|\Lambda^s_{\eps}\na^{\eps} u\|_2 \leq &
M(\sigma)\Big(\eps^{-\f12}\|\Lambda^{s-1}_\eps
f\|_2+|g|_{H^s_{\eps}}+\langle(\eps^{-\f12}\|\Lambda^{t_0-1}_\eps
f\|_2+|g|_{H^{t_0}_\eps})|(\zeta, b)|_{H^{s+1}_{\eps}}
\rangle_{s>t_0+1}\Big),\\
\|\Lambda^s_{\eps}\na^{\eps} u\|_2 \leq &
M(\sigma)\Big(\eps^{-\f12}\|\Lambda^{s-1}_\eps
f\|_2+\eps^{-\f14}|g|_{H^{s-\f12}_\eps}\\
&\ +\langle(\eps^{-\f12}\|\Lambda^{t_0-1}_\eps
f\|_2+|g|_{H^{t_0}_\eps})|(\zeta,
b)|_{H^{s+1}_{\eps}}\rangle_{s>t_0+1}\Big),
\end{split}
\end{align*}
and if $s\ge 1$, we have \beno \begin{split}
\|\Lambda^{s-1}_\eps\na^{\eps} \p_zu\|_2 \leq &
M(\sigma)\Big(\|\Lambda^{s-1}_\eps f\|_2+\sqrt{\eps}|g|_{H^s_{\eps}}
+\langle(\|\Lambda^{t_0}_\eps f\|_2
+\sqrt{\eps}|g|_{H_\eps^{t_0+1}})|(\zeta,b)|_{H^{s+1}_{\eps}}\rangle_{s>t_0+1}\Big).
\end{split}\eeno }
\end{prop}

\begin{proof}\,\,Since the existence part can be obtained by a  standard argument,
here we just present the detailed proof of the estimates. Indeed
recall that what we mean by $u$ is a variational solution of
(\ref{eq:elliptice equation-homo-general}):  for any $\phi\in
C^\infty_c([-1,0)\times\R^2)$, there holds \beq \label{prop3.2}
\begin{split}
\int _{\mathcal
S}(1+Q^\eps[\sigma])\na^{\eps}u\cdot\na^{\eps}\phi\,dx_h\,dz
&=-\int_{\mathcal S}f\,\phi\,dx_h\,dz+\int_{\R^2}g\,\
\phi|_{z=-1}\,dx_h.\end{split}  \eeq
 Taking $\phi=\Lambda^{2s}_\eps u$ in \eqref{prop3.2} results in
\begin{eqnarray}\label{eq:ellip-est1}
\begin{split}
&\int_{\mathcal S}(1+Q^\eps[\sigma])\na^{\eps}\Lambda^s_{\eps}
u\cdot\na^{\eps}\Lambda^s_{\eps} u\,dx_h\,dz\\
&=-\int_{\mathcal S}\bigl([\Lambda^s_{\eps},
Q^\eps[\sigma]]\na^{\eps}u\cdot\na^{\eps}\Lambda^s_{\eps}
u+\Lambda^{s-1}_\eps f\,\ \Lambda^{s+1}_{\eps} u\bigr)\,dx_h\,dz+
\int_{\R^2}\Lambda^s_{\eps} g\,\ \Lambda^s_{\eps} u|_{z=-1}\,dx_h\\
&\leq\|[\Lambda^s_{\eps}, Q^\eps[\sigma]]\na^{\eps}u\|_2
\|\na^{\eps}\Lambda^s_{\eps} u\|_2+\|\Lambda^{s-1}_\eps
f\|_2\|\Lambda^{s+1}_{\eps} u\|_2+|\Lambda^s_{\eps}
g|_2|\Lambda^s_{\eps} u|_{z=-1}|_2.
\end{split}
\end{eqnarray}
Thanks to Proposition 2.3 in \cite{Lannes-Inven}, we have
\begin{eqnarray*}
\int_{\mathcal S}(1+Q^\eps[\sigma])\na^{\eps}\Lambda^s_{\eps}
u\cdot\na^{\eps}\Lambda^s_{\eps} u\,dx_h\,dz \geq
M(\sigma)^{-1}\|\Lambda^s_{\eps} \na^{\eps} u\|^2_2,
\end{eqnarray*}
While it follows from (\ref{lem:commutator estimate}) that
\begin{eqnarray*}
\begin{split}
\|[\Lambda^s_{\eps}, Q^\eps[\sigma]]\na^{\eps}u\|_2 \leq
M(\sigma)\Big(\|\Lambda^{s-1}_\eps\na^{\eps}
u\|_2+\langle\|\Lambda^{t_0}_\eps\na^{\eps}u\|_2|(\zeta,
b)|_{H^{s+1}_{\eps}}\rangle_{s>t_0+1}\Big),
\end{split}
\end{eqnarray*}
and thanks to $u(x_h,0)=0$, we have \beno
\sqrt{\eps}\|\Lambda^{s+1}_{\eps} u\|_2\le
C\|\Lam^s_{\eps}\na^{\eps} u\|_2, \quad |\Lambda^s_{\eps}
u|_{z=-1}|_2\leq \|\Lambda^s_{\eps} \p_z u\|_2. \eeno Plugging all
the above estimates into (\ref{eq:ellip-est1}) and using Young's
inequality yield that
\begin{align*}
\begin{split}
\|\Lambda^s_{\eps} \na^{\eps} u\|_2 \leq&
M(\sigma)\Big(\|\Lambda^{s-1}_\eps\na^{\eps}
u\|_2+\langle\|\Lambda^{t_0}_\eps\na^{\eps}u\|_2|(\zeta,
b)|_{H^{s+1}_{\eps}}\rangle_{s>t_0+1}\\
&+\eps^{-\f12}\|\Lambda^{s-1}_\eps f\|_2+|g|_{H^s_\eps}\Big),
\end{split}
\end{align*}
from which and an interpolation argument, we deduce that
\begin{align*}
\begin{split}
\|\Lambda^s_{\eps} \na^{\eps} u\|_2 \leq&
M(\sigma)\Big(\|\na^{\eps}u\|_2
+\langle\|\Lambda^{t_0}_\eps\na^{\eps}u\|_2|(\zeta,
b)|_{H^{s+1}_{\eps}}\rangle_{s>t_0+1}\\
&+\eps^{-\f12}\|\Lambda^{s-1}_\eps f\|_2+|g|_{H^s_\eps}\Big).
\end{split}
\end{align*}
Whereas taking $\phi=u$ in \eqref{prop3.2}, we get \beno
\|\na^{\eps} u\|_2 \leq M(\sigma)(\|f\|_2+|g|_2). \eeno
Consequently, we arrive at \beq\label{prop3.2a}
\begin{split}\|\Lambda^s_{\eps} \na^{\eps} u\|_2 \leq
M(\sigma)\Big(&\langle\|\Lambda^{t_0}_\eps\na^{\eps}u\|_2|(\zeta,
b)|_{H^{s+1}_{\eps}}\rangle_{s>t_0+1}+\eps^{-\f12}\|\Lambda^{s-1}_\eps
f\|_2+|g|_{H^s_\eps}\Big),
\end{split}\eeq and taking $s=t_0$ in \eqref{prop3.2a} gives \beno
\|\Lambda^{t_0}_\eps \na^{\eps} u\|_2\le M(\sigma)
\Big(\eps^{-\f12}\|\Lambda^{t_0-1}_\eps
f\|_2+|g|_{H^{t_0}_\eps}\Big), \eeno which implies the first
inequality of Proposition \ref{prop:elliptic estimate-general}.

To prove the second inequality, we only need to replace  the
estimate for the boundary term. Indeed thanks to Lemma
\ref{lem:trace theorem}, one has
\[
\bigl|\int_{\R^2} \Lambda^s_{\eps} g\Lambda^s_{\eps}
u|_{z=-1}\,dx_h\bigr| \leq
|g|_{H^{s-\f12}_\eps}|u|_{z=-1}|_{H^{s+\f12}_\eps}\leq
C\eps^{-\f14}|g|_{H^{s-\f12}_\eps}\|\Lambda^s_{\eps}\na^{\eps}u\|_2,
\]
which along with the proof of \eqref{prop3.2a} gives the second
inequality of Proposition \ref{prop:elliptic estimate-general}.

Finally, we get by using the elliptic equation and
\eqref{lem:product estimate} to obtain
\begin{align*}
\begin{split}
\|\Lambda^{s-1}_\eps\p_z^2 u\|_2\le&
M(\sigma)\Big(\|\Lambda^{s-1}_\eps f\|_2+
\sqrt{\eps}\|\Lambda^s_{\eps}\na^{\eps} u\|_2\\
&\quad +\langle(\|\Lambda^{t_0}_\eps
f\|_2+\sqrt{\eps}\|\Lambda^{t_0+1}_\eps \na^{\eps}
u\|_2)|(\zeta,b)|_{H^{s+1}_{\eps}}\rangle_{s>t_0+1}\Big).
\end{split}
\end{align*}
This together with the first inequality implies the third inequality
of the proposition. This finishes the proof of Proposition
\ref{prop:elliptic estimate-general}. \end{proof} \vspace{0.1cm}

\renewcommand{\theequation}{\thesection.\arabic{equation}}
\setcounter{equation}{0}
\section{Calculus of pseudo-differential operators with symbols of limited smoothness}

In this section, we shall adapt some results from \cite{Lannes-JFA}
on the calculus of pseudo-differential operators with symbols of
limited smoothness to our setting here. More precisely,  we shall
consider symbols of the form \beno
\sigma(x_h,\xi)=\Sigma(v(x_h),\xi) \eeno for $v\in C^0(\R^2)^p$ with
$p\in \N$, and $\Sigma$ is a function defined as follows (see
\cite{Lannes-JFA}):

\begin{defi}\label{def4.1} Let $m\in \N_0, p\in \N$, and $\Si$ be a function defined
on $\R^p_v\times \R^2_\xi$. We say $\Si(v,\xi)\in
C^\infty(\R^p,\dot\cM^m)$ if
\begin{itemize}
 \item[1)]
$\Si\in C^\infty(\R^p;C^{m})$ and $|\p_{\xi}^\be\Si(v,{\xi})|\le
C_\be(|v|)|{\xi}|^{m-\be}$ for any $\xi\in\R^2$, $|\be|\le m$;

\item[2)] for any $\alpha\in\N^p,\beta\in\N^2$, there exists a non-decreasing
function $C_{\alpha,\beta}(\cdot)$ such that
\[\sup_{|{\xi}|\ge \f14}|{\xi}|^{|\beta|-m}|(\p^\alpha_v\p^\beta_{\xi}\Si) (v,{\xi})|\le
C_{\alpha,\beta}(|v|).\] \end{itemize}
\end{defi}

For given $\Sigma,$ $v$ and $\eps\in [0,1],$ we consider
pseudo-differential operators, $\textrm{Op}_{\eps}(\sigma),$ defined
by \beno \textrm{Op}_{\eps}(\sigma)u(x_h){\eqdefa}
(2\pi)^{-2}\int_{\R^2}e^{ix_h\cdot{\xi}}\sigma(x_h,{\xi}^{\eps})\hat
u({\xi})d{\xi},\quad {\xi}^{\eps}=(\xi_1,\sqrt{\eps}\xi_2). \eeno

\begin{prop}\label{prop:psdo-estimate} {\sl Let $m\in \N_0$,
$p\in \N$. Then for given $\Sigma\in C^\infty(\R^p,$ $\dot\cM^m),$
$v\in H^{t_0}(\R^2)$ and $\sigma(x_h,{\xi})=\Sigma(v(x_h),{\xi}),$
one has \beno
\begin{split}
\bigl|\textrm{Op}_{\eps}(\sigma)u\bigr|_{H^s_{\eps}}\le&
C(|v|_{\infty})\Big(||D^{\eps}_{h}|^m
u|_{H^{s}_\eps}+\eps^{-\f14}|v|_{H^{t_0}_\eps}||D^{\eps}_{h}|^m
u|_{H^{s}_\eps}+\langle\eps^{-\f14}|v|_{H^{s}_\eps}||D^{\eps}_{h}|^m
u|_{H^{t_0}_\eps}\rangle_{s>t_0}\Big),\end{split} \eeno for all
$0\le s\le t_0$.}
\end{prop}

\begin{proof}\,Using the scaling argument, one can reduce the proof of
Proposition \ref{prop:psdo-estimate} to the case when $\eps=1$. We
first split $u$ as the low and high frequency part so that
\beq\label{prop4.1} u=u_{lf}+u_{hf}, \quad u_{lf}=\psi(D)u, \eeq
where $\psi\in C^\infty_0(\R^d)$ and $\psi\equiv1$ near the origin.
Let $\sigma_0({\xi}){\eqdefa}\Sigma(0,{\xi})$, it is easy to observe
that
$$
\bigl|\textrm{Op}(\sigma_0)u\bigr|_{H^s}\le C||D|^mu|_{H^s}.
$$  Whence without loss of generality, we may assume
that $\sigma_0({\xi})=0.$ While thanks to Corollary 30 of
\cite{Lannes-JFA}, we have
\begin{align*}
\begin{split}
\bigl|\textrm{Op}(\sigma)u_{hf}\bigr|_{H^s} &\ \le
C(|v|_{\infty})\big(|u_{hf}|_{H^{s+m}}+|v|_{H^{t_0}}|u_{hf}|_{H^{s+m}}+\langle|v|_{H^{s}}|u_{hf}|_{H^{t_0+m}}\rangle
_{s>t_0}\big)\\
&\ \le
C(|v|_{\infty})\big(||D|^mu|_{H^{s}}+|v|_{H^{t_0}}||D|^mu|_{H^{s}}+\langle|v|_{H^{s}}||D|^mu|_{H^{t_0}}\rangle
_{s>t_0}\big).
\end{split}
\end{align*}
On other hand, notice that \beno
\textrm{Op}(\sigma)u_{lf}(x_h)=(2\pi)^{-2}\int_{\R^2}e^{ix_h\cdot{\xi}}\sigma(x_h,{\xi})\psi({\xi})\hat
u({\xi})d{\xi}, \eeno and
$|e^{ix_h\cdot{\xi}}\sigma(x_h,{\xi})|_{H^s}\le
C\langle{\xi}\rangle^s|\sigma(\cdot,{\xi})|_{H^s}$, which along with
(\ref{lem:composition}) ensures that
\begin{align*}
\bigl|\textrm{Op}(\sigma)u_{lf}\bigr|_{H^s} \le& C\sup_{|{\xi}|\le
1}\bigl(|{\xi}|^{-m}|\sigma(\cdot,{\xi})|_{H^s}\bigr)
\int_{|{\xi}|\le 1}|{\xi}|^m|\hat u({\xi})|d{\xi}\\
\le& C(|v|_{\infty})|v|_{H^s}||D|^mu|_{2}.
\end{align*}
This finishes the proof of Proposition
\ref{prop:psdo-estimate}.\end{proof}

To handle the composition and commutator between two
pseudo-differential operators of limited-smooth symbols,  we recall
the following symbols for $n\in \N_0$:
 \beq\label{poisson}
 \begin{split}
 &\sigma_1\sharp_n
\sigma_2(x_h,\xi){\eqdefa} \sum_{|\alpha|\le
n}\f{(-i)^{|\alpha|}}{\alpha!}\partial_{\xi}^\alpha\sigma_1(x_h,\xi)\partial_{x_h}^\alpha\sigma_2(x_h,\xi)\quad\mbox{and}
\\
& \{\sigma_1,\sigma_2\}_n(x_h,\xi){\eqdefa} \sigma_1\sharp_n
\sigma_2(x_h,\xi)-\sigma_2\sharp_n \sigma_1(x_h,\xi).
\end{split} \eeq

\begin{prop}\label{prop:psdo-commuatator} {\sl Let $m_1, m_2, n\in\N_0$
with $m^+{\eqdefa}\max(m_1,m_2)$, $m^-{\eqdefa}$ $\min(m_1,m_2)\ge
n$. Let $\sigma^j(x_h,\xi)=\Si^j(v^j(x_h),\xi)$ $(j=1,2)$ with
$p_j\in\N$, $\Si^j\in C^\infty(\R^{p_j},\dot\cM^{m_j})$ and $v_j\in
H^{t_0+m^++1}(\R^2)$. Then for any $0\le s\le t_0+1$, there holds
\beno
\begin{split}
\big|\textrm{Op}_{\eps}&(\sigma^1)\circ\textrm{Op}_{\eps}(\sigma^2)u-
\textrm{Op}_{\eps}(\sigma^1\sharp_n\sigma^2)u\big|_{H^s_\eps}\\
&\le
C(|v|_{W^{n+1,\infty}})\Bigl\{||D^{\eps}_{h}|^{m^-}u|_{H^{s+m^+-n-1}_\eps}|v|_{W^{n+1,\infty}}
+\big(\eps^{-\f14}|v|_{H^{t_0+m^++1}_\eps}
\\
&\quad+\eps^{-\f12}|v|_{H^{t_0+m^++1}_\eps}^2\big)||D^{\eps}_{h}|^{m^-}u|_{2}
\\
&\quad+\langle\big(\eps^{-\f14}|v|_{H^{s+m^+}_\eps}
+\eps^{-\f12}|v|_{H^{s+m^+}_\eps}|v|_{H^{t_0+m^++1}_\eps}\big)||D^{\eps}_{h}|^{m^+}u|_{H^{t_0}_\eps}\rangle_{s>t_0+1}
\Bigr\}.
\end{split}
\eeno
In particular, we have
\beno
\begin{split}
\big|\big[\textrm{Op}_{\eps}&(\sigma^1),\textrm{Op}_{\eps}(\sigma^2)\big]u-
\textrm{Op}_{\eps}(\{\sigma^1,\sigma^2\}_n)u\big|_{H^s_\eps}\\
&\le
C(|v|_{W^{n+1,\infty}})\Bigl\{||D^{\eps}_{h}|^{m^-}u|_{H^{s+m^+-n-1}_\eps}|v|_{W^{n+1,\infty}}
\\
&\quad+\big(\eps^{-\f14}|v|_{H^{t_0+m^++1}_\eps}+\eps^{-\f12}|v|_{H^{t_0+m^++1}_\eps}^2\big)||D^{\eps}_{h}|^{m^-}u|_{2}\\
&\quad +\langle\big(\eps^{-\f14}|v|_{H^{s+m^+}_\eps}
+\eps^{-\f12}|v|_{H^{s+m^+}_\eps}|v|_{H^{t_0+m^++1}_\eps}\big)||D^{\eps}_{h}|^{m^+}u|_{H^{t_0}_\eps}\rangle_{s>t_0+1}
\Bigr\}. \end{split}\eeno Here $v=(v^1,v^2).$} \end{prop}

\begin{proof}\, Again using a scaling argument, one can reduce the proof of Proposition
\ref{prop:psdo-commuatator} to the case when $\eps=1$. As in the
proof of Proposition \ref{prop:psdo-estimate}, we split $u$ into
$u=u_{lf}+u_{hf}$ given by \eqref{prop4.1}. Then we get by applying
 Theorem 7 and Theorem 8 in \cite{Lannes-JFA} that
\beq\label{prop4.2a}
\begin{split}
\big|\textrm{Op}&(\sigma^1)\circ\textrm{Op}(\sigma^2)u_{hf}-\textrm{Op}(\sigma^1\sharp_n\sigma^2)u_{hf}\big|_{H^s}\\
&\le
C(|v|_{W^{n+1,\infty}})\Bigl(|u_{hf}|_{H^{s+m_1+m_2-n-1}}|v|_{W^{n+1,\infty}}
+\langle\big(|v|_{H^{s+m^+}}\\
&\qquad+|v|_{H^{s+m^+}}|v|_{H^{t_0+m^++1}}\big)|u_{hf}|_{H^{m^++t_0}}\rangle_{s>t_0+1} \Bigr)\\
&\le
C(|v|_{W^{n+1,\infty}})\Bigl(||D|^{m^-}u|_{H^{s+m^+-n-1}}|v|_{W^{n+1,\infty}}\\
&\qquad+
\langle\big(|v|_{H^{s+m^+}}+|v|_{H^{s+m^+}}|v|_{H^{t_0+m^++1}}\big)||D|^{m^+}u|_{H^{t_0}}\rangle_{s>t_0+1}\Bigr).
\end{split}
\eeq Setting $\sigma_0^1(\xi){\eqdefa}\Sigma^1(0,\xi)$ and
$\sigma_0^2(\xi){\eqdefa}\Sigma^2(0,\xi)$, we write
\begin{align*}
\textrm{Op}(\sigma^1)\circ\textrm{Op}(\sigma^2)u-\textrm{Op}(\sigma^1\sharp_n\sigma^2)
&=\textrm{Op}(\sigma^1)\circ\textrm{Op}(\sigma^2-\sigma_0^2)-\textrm{Op}(\sigma^1\sharp_n(\sigma^2-\sigma_0^2))
\\&\quad+\textrm{Op}(\sigma^1-\sigma_0^1)\circ\textrm{Op}(\sigma_0^2)-\textrm{Op}((\sigma^1-\sigma_0^1)\sharp_n \sigma_0^2).
\end{align*}
It follows from the proof of Proposition \ref{prop:psdo-estimate}
that \beno
\big|\textrm{Op}(\sigma^1\sharp_n(\sigma^2-\sigma_0^2))u_{lf}\big|_{H^s}\le
C(|v|_{W^{n,\infty}})|v|_{H^{s+n}} ||D|^{m_1+m_2-n}u_{lf}|_{2},
\eeno and it is easy to observe that \beno\begin{split}
\big|\textrm{Op}&(\sigma^1)\circ\textrm{Op}(\sigma^2-\sigma^2_0)u_{lf}\big|_{H^s}\\&\le
C(|v_1|_\infty)\Big(|\textrm{Op}(\sigma^2-\sigma^2_0)u_{lf}|_{H^{s+m_1}}+
|v_1|_{H^s}|\textrm{Op}(\sigma^2-\sigma^2_0)u_{lf}|_{H^{t_0+m_1}}\Big)\\
&\le C(|v|_\infty)\Big(|v_2|_{H^{s+m_1}}||D|^{m_2}u_{lf}|_{2}+
|v_1|_{H^s}|v_2|_{H^{t_0+m_1}}||D|^{m_2}u_{lf}|_{2}\Big),\end{split}
\eeno which implies that
\begin{align*}
\begin{split}
&\big|\textrm{Op}(\sigma^1)\circ\textrm{Op}(\sigma^2)u_{lf}-\textrm{Op}(\sigma^1\sharp_n\sigma^2)u_{lf}\big|_{H^s}\le
C(|v|_{W^{n,\infty}})\big(1+|v|_{H^{t_0+m^+}}\big)|v|_{H^{s+m^+}}|||D|^{m^-}u|_{2}.
\end{split}
\end{align*}
This along with \eqref{prop4.2a} concludes the proof of Proposition
\ref{prop:psdo-commuatator}.\end{proof}

\begin{prop}\label{prop:psdo-commuatator-special} {\sl Let $m_1, m_2, n\in\N_0$
with $m_1, m_2\ge n$ and $m^-{\eqdefa}$ $\min(m_1,m_2)$. Let
$\sigma^1(\xi)\in \dot\cM^{m_1}$, and
$\sigma^2(x_h,\xi)=\Si(v(x_h),\xi)$  with $p\in\N$, $\Si\in
C^\infty(\R^{p},\dot\cM^{m_2})$ and $v\in H^{t_0+m_1+1}(\R^2)$. Then
for any $0\le s\le t_0+1$, there holds \beno
\begin{split}
\big|&\big[\sigma^1(D^{\eps}_{h}),\textrm{Op}_{\eps}(\sigma^2)\big]u
-\textrm{Op}_{\eps}(\{\sigma^1,\sigma^2\}_n)u\big|_{H^s_\eps}\\
 &
\le
C(|v|_{W^{n+1,\infty}})\Bigl(||D^{\eps}_{h}|^{m_1+m_2}u|_{H^{s-n-1}_\eps}|v|_{W^{n+1,\infty}}
+\eps^{-\f14}|v|_{H^{s+m_1}_\eps}\big(||D^{\eps}_{h}|^{m_2}u|_{H^{t_0}_\eps}+||D^{\eps}_{h}|^{m^-
}u|_{2}\big)\Bigr). \end{split}\eeno} \end{prop}

\begin{proof}\,  Similar to the proof of Proposition \ref{prop:psdo-commuatator},
one first reduces the proof of this proposition to the case when
$\eps=1.$ For the high frequency
 part,  $u_{hf}$ of $u$,
we  use  Corollary 39 in \cite{Lannes-Inven} so that \beno
\begin{split}
\big|\big[\sigma^1(D),&\textrm{Op}(\sigma^2)\big]u_{hf}
-\textrm{Op}(\{\sigma^1,\sigma^2\}_n)u_{hf}\big|_{H^s}\\
&\le
C(|v|_{W^{n+1,\infty}})\bigl(|u|_{H^{s+m_1+m_2-n-1}}|v|_{W^{n+1,\infty}}
+|v|_{H^{s+m_1}}|u|_{H^{t_0+m_2}}\bigr).\end{split} \eeno The low
frequence part, $u_{lf}$ of $u$ can be obtained by exactly  the same
line as the proof to Proposition
\ref{prop:psdo-commuatator}.\end{proof}

\renewcommand{\theequation}{\thesection.\arabic{equation}}
\setcounter{equation}{0}

\section{The Dirichlet-Neumann operator}

The goal of this section is to study the Dirichlet-Neumann operator
defined by \eqref{DNO}, which will be the key ingredient used in the
proof of the first part of Theorem \ref{thm:KP approximation}.
Firstly thanks to the argument at the beginning of section 3, we
write (\ref{eq:elliptic equation-DN}) on $\Omega=\R^2\times[-1+\eps
b(x_h),\eps \zeta(t,x_h)]$ into a problem on the flat strip $\cS$:
\ben\label{eq:elliptic equation-DN-new} \left\{\begin{array}{ll}
\na_{h,z}\cdot P^\eps[\sigma]\na_{h,z}\psi^b=0,\quad\hbox{in}\quad \mathcal{S},\\
\psi^b|_{z=0}=\psi,\qquad\p_n^P \psi^b|_{z=-1}=0.
\end{array}\right.
\een where $\psi^b=\Phi\circ S$, $P^\eps[\sigma]=I+Q^\eps(\sigma),$
$Q^\eps(\sigma)$ and $S$ are given by \eqref{Qd} and
\eqref{Transform} respectively. Then we can write the
Dirichlet-Neumann operator as \beq\label{eq:DN-representation}
G[{\eps}\zeta]\psi=\p_n^P\psi^b|_{z=0}=-{\bf e_3}\cdot
P^\eps[\sigma]\na \psi^b|_{z=0}. \eeq

\subsection{Some basic properties}

For the convenience of the readers, we shall first recall some basic
properties of Dirichlet-Neumann operator from \cite{Lannes-Inven}.

\begin{prop}\label{prop:DN-basic properties} Let $ \zeta, b\in
H^{m_0}(\R^2)$ satisfy (\ref{ass:free surface}). Then we have
\begin{itemize}
 \item[(1)]
 The Dirichlet-Neumann operator $G[{\eps}\zeta]$ is self-adjoint: \beno
\big(u, G[\eps\zeta]v\big)=\big(v, G[\eps\zeta]u\big),\quad
\forall\, u,v\in H^\f12(\R^2); \eeno

\item[(2)] For all $ u,v\in H^\f12(\R^2)$,
\beno \bigl|\big(u, G[\eps\zeta]v\big)\bigr|\le \big(u,
G[\eps\zeta]u\big)^\f12 \big(v, G[\eps\zeta]v\big)^\f12; \eeno

\item[(3)] For $ u\in H^\f12(\R^2)$,  \beno M(\sigma)^{-1}|\mathfrak
{P}u|_2^2\leq \big(u, \f1\eps G[\eps\zeta]u\big)\le
M(\sigma)|\mathfrak {P}u|_2^2; \eeno

\item[(4)] For  $\underline{\bf v}\in H^{t_0+1}(\R^2)^2,$  $u\in
H^\f12(\R^2)$,  and  $\frak{\Delta}_{\eps}(\zeta)$ given by
\eqref{eq:ellipticoperator},  \beno
\begin{split}
&\bigl|\big(\underline{\bf v}\cdot\na_h^{\eps} u, \f1\eps
G[\eps\zeta]u\big)\bigr|
\le M(\sigma)|\underline{\bf v}|_{W^{1,\infty}}|\mathfrak {P}u|_2^2\quad\mbox{and}\\
&\bigl|\big(\big[\fd_\eps(\zeta)^k, \underline{\bf
v}\cdot\na_h^{\eps}\big]u,
 \f1\eps G[\eps\zeta]\big[\fd_\eps(\zeta)^k, \underline{\bf
 v}\cdot\na_h^{\eps}\big]u\big)\bigr|\\
&\qquad\le M(\sigma)|\underline{\bf v}|_{H^{t_0+2}}\Big(|\mathfrak
{P}u|_{H^{2k}_\eps}^2+|(\zeta,\underline{\bf v})|_{H^{2k+2}}
|\mathfrak {P}u|_{H^{t_0+1}}^2\Big). \end{split}\eeno \end{itemize}
\end{prop}

 \begin{proof}\, The second estimate in (4) can be
deduced by following the proof of Proposition 3.7 (i) in
\cite{Lannes-Inven}, and all the other estimates can be found in
\cite{Lannes-Inven}.
\end{proof}

\begin{prop}\label{prop:shape derivative} {\sl Let $ s\ge t_0,$ and
$\zeta, b\in H^{s+\f32}(\R^2)$ satisfy (\ref{ass:free surface}).
Then for any $\psi\in H^{s+\f32}$, the mapping $\zeta\mapsto
G[\eps\zeta]\psi$ is well-defined and differentiable in a
neighborhood of $\zeta$ in $H^{s+\f32}(\R^2).$ Moreover, for any
$h\in H^{s+\f32}(\R^2)$, there holds \beno \begin{split} &d_\zeta
G[\eps\zeta]\psi\cdot h=-\eps
G[\eps\zeta](hZ)-\eps^2\na_h^{\eps}\cdot(h\textbf{v})\qquad\mbox{with}\\
&{\bf v}=\na^{\eps}_h\psi-{\eps} Z \na^{\eps}_h\zeta\quad\mbox{
and}\quad Z=\frac1
{1+{\eps}^3|\na^{\eps}_h\zeta|^2}\big(G[{\eps}\zeta]\psi
+{\eps}^2\na^{\eps}_h\zeta\cdot\na^{\eps}_h\psi\big).
\end{split}\eeno}
\end{prop}

\begin{prop}\label{prop:shape derivative estimate} {\sl Let $ s\ge
t_0,$  and $\zeta, b\in H^{m_0}\cap H^{s+1}(\R^2)$ satisfy
(\ref{ass:free surface}). Then for any $u\in H^{s+\f12}(\R^2),$ $
j\in \{0,1,2\}$ and $\textbf{h}\in H^{t_0+1}\cap H^{s+1}(\R^2)^j$,
one has \beno
\begin{split}&\Big|\f1{\sqrt{\eps}}d_\zeta^jG[\eps\zeta]u\cdot
\textbf{h}\Big|_{H^{s-\f12}_{\eps}} \le
\eps^{\f34j}M(\sigma)\Big(|\mathfrak P
u|_{H^s_\eps}\prod_{k=1}^j|h_k|_{H^{t_0+1}_\eps}+|(\zeta,b)|_{H^{s+1}_{\eps}}\\
&\qquad\times|\mathfrak P u|_{H^{t_0}_\eps}
\prod_{k=1}^j|h_k|_{H^{t_0+1}_{\eps}}+|\mathfrak P u|_{H^{t_0}_\eps}
\sum_{k=1}^j|h_k|_{H^{s+1}_\eps}\prod_{l\neq
k}|h_l|_{H^{t_0+1}_\eps}\Big).\end{split} \eeno} \end{prop}

\begin{proof}\,We only present  the proof for the case when $j=0$, the other cases
can be handled in a  similar way (one may check the proof of
Proposition 3.3 in \cite{Lannes-Inven}). Indeed for any $v\in
\cS(\R^2)$, let $u^b$ and $u^\dag$ be defined by \eqref{eq:elliptic
equation-DN-new} and \eqref{lem:Poisson regularization}
respectively. Then applying \eqref{lem:product estimate} and the
fact that
\[
\|\Lam^{-\f12}_\eps\na^{\eps}v^\dag\|_2\le C|v|_2
\] gives
\begin{align*}
\begin{split}
\big(\Lam^{s-\f12}_\eps G[\eps
\zeta]u,v\big)&=\big(G[\eps\zeta]u,\Lam^{s-\f12}_\eps
v\big)=\int_\cS (1+Q^\eps[\si])\na^{\eps}u^b\cdot
\Lam^{s-\f12}_\eps\na^{\eps}v^\dag\,dx_h\,dz\\
&=\int_\cS\Lam^{s}_\eps(1+Q^\eps[\si])\na^{\eps}u^b\cdot\Lam^{-\f12}_\eps\na^{\eps}v^\dag\,dx_h\,dz\\
&\le C \|\Lam^{s}_\eps(1+Q^\eps[\si])\na^{\eps}u^b\|_2|v|_2\\
&\le M(\si)|v|_2\big(\|\Lam^{s}_\eps\na^{\eps}u^b\|_2
+\|\Lam^{t_0}_\eps\na^{\eps}u^b\|_2|(\zeta, b)|_{H^{s+1}_\eps}\big),
\end{split}
\end{align*}
which along with  Corollary \ref{cor:elliptic estimate} proves the
proposition for the case $j=0.$ \end{proof}

\begin{rmk}\label{rem:DN operator} We can also deduce  from the proof of Proposition
3.3 in \cite{Lannes-Inven} that \beno
\begin{split}
&\Big|\f1{\sqrt{\eps}}\Lam^{m-\f12}_\eps d_\zeta^jG[\eps\zeta]u\cdot
\textbf{h}\Big|_{H^{s}}
\le \eps^jM(\sigma)\Big(|\Lam^{m}_\eps\mathfrak P u|_{H^s}\prod_{k=1}^j|h_k|_{H^{t_0+m+1}}\\
&\qquad\quad +|\Lam^{m+1}_\eps(\zeta,b)|_{H^{s}}|\mathfrak P
u|_{H^{t_0+m}}\prod_{k=1}^j|h_k|_{H^{t_0+m+1}} \\
&\qquad\quad +|\mathfrak P
u|_{H^{t_0+m}}\sum_{k=1}^j|\Lambda_{\eps}^{m+1}h_k|_{H^{s}}\prod_{l\neq
k}|h_l|_{H^{t_0+m+1}}\Big),\end{split} \eeno for $m=0,1,2,3$. This
result is not sharp, but is enough for our applications in this
paper.
\end{rmk}

\begin{prop}\label{prop:DN-commutator-time} {\sl Let $ T>0,$
$b\in H^{m_0}(\R^2), \zeta\in C([0,T];H^{m_0}(\R^2))$ satisfy
(\ref{ass:free surface}) for some $h_0$ independent of $t$. Then for
any $u\in C^1([0,T];$ $ H^{\f12}(\R^2))$ and $t\in [0,T]$, one has
\beno \Big|\Big(\big[\p_t, G[\eps\zeta]\big]u(t), u(t)\Big)\Big| \le
\eps M(\sigma(t))|\na_h^{\eps}\p_t\zeta|_\infty|\mathfrak P
u(t)|_2^2. \eeno }\end{prop}

\subsection{The principle  part of the DN operator}\label{sec5.2} Recall that
$\sigma(t,x_h,z)=-\eps zb(x_h)+\eps(1+z)\zeta(t,x_h),$ we  rewrite
$P^\eps[\sigma]$ in \eqref{eq:elliptic equation-DN-new} as
\[
P^\eps[\sigma]=\left(\begin{matrix} P_1^\eps & {\bf p}^\eps\\
({\bf p}^\eps)^T & p_{d+1}^\eps\end{matrix}\right),
\]
with
\begin{align*}
\begin{split}
&P_1^\eps={\eps}(1+{\eps} \zeta-{\eps} b)
\left(\begin{matrix}1& 0\\
0 & \eps\end{matrix}\right),\\
&{\bf p}^\eps=-{\eps}^2\left(\begin{array}{ll}-z\p_xb+(z+1)\p_x\zeta\\
\eps(-z\p_yb+(z+1)\p_y\zeta)
\end{array}\right),\\
&p_{d+1}^\eps=\f{1+{\eps}\big({\eps} z\p_xb-{\eps}(z+1)\p_x
\zeta\big)^2+{\eps}^2 \big({\eps}
z\p_yb-{\eps}(z+1)\p_y\zeta\big)^2}{1+{\eps} \zeta-{\eps} b}.
\end{split}
\end{align*}
Then we have
\begin{align*}
\begin{split}
{\bf P}^\eps[\sigma]{\eqdefa}&\na_{h,z}\cdot P^\eps[\sigma]\na_{h,z}
=p_{d+1}^\eps\p^2_z+(2{\bf p}^\eps\cdot
\na_h+\big(\p_zp_{d+1}^\eps\\
&+\na_h\cdot {\bf
p}^\eps)\big)\p_z+P_1^\eps\Delta_h+\big((\na_h\cdot
P_1^\eps)+\p_z{\bf p}^\eps\big)\cdot\na_h.
\end{split}
\end{align*}

For simplicity, we shall neglect the subscript $\eps$ in what
follows.

We now define the approximate operator to ${\bf P}^\eps[\sigma]$ as
follows
\begin{eqnarray}\label{papp}
{\bf P}_{app}{\eqdefa} p_{d+1}(\p_z-\eta_-(x_h,z,D^{\eps}_{h}))
(\p_z-\eta_+(x_h,z,D^{\eps}_{h}))
\end{eqnarray}
where $\eta_\pm(x_h,z,D^{\eps}_{h})$ ($z\in [-1,0]$) are
pseudo-differential operators with symbols
\begin{align}
\eta_\pm(x_h,z,\xi^{\eps}) &=\frac 1{p_{d+1}} \Bigl(-i{\bf
p}\cdot\xi\pm\sqrt{p_{d+1}\xi\cdot
P_1\xi-({\bf p}\cdot\xi)^2}\Bigr)\label{eta}\\
&=\frac{1+\p_z\sigma}{1+{\eps}|\na^{\eps}_h\sigma|^2}
\left(i{\eps}\na^{\eps}_h\sigma\cdot\xi^{\eps}\pm
\sqrt{{\eps}(1+{\eps}|\na^{\eps}_h\sigma|^2)|\xi^{\eps}|^2
-{\eps}^2|\na^{\eps}_h\sigma\cdot\xi^{\eps}|^2}\right). \nonumber
\end{align}
It is easy to observe that there exits some constant $c_+>0$ so that
\beq \label{eq:eta} \sqrt{\eps} M(\sigma)|\xi^{\eps}|\ge
\textrm{Re}(\eta_+(x_h,z,\xi^{\eps})) \ge \sqrt{\eps}
c_{+}|\xi^{\eps}|, \eeq and there exists $\Sigma_\pm(v,\xi) \in
C^\infty(\R^3,\dot\cM^1)$ such that $\eta_\pm(x_h,z,\xi^{\eps})=$
$\sqrt{\eps}\ \Sigma_\pm(\na^{\eps}\sigma,\xi^{\eps}).$

As in \cite{lan1}, for $u\in\cS(\R^2),$ we define the approximate
solutions to \eqref{u^b} as \beq\begin{split}
 \label{eq:main symbol
of DN} &u^b_{app}(x_h,z){\eqdefa}\sigma_{app}(x_h,z,D^{\eps}_{h})u
\quad \mbox{with}\\
& \sigma_{app}(x_h,z,\xi)=\exp\big(-\int^0_z
\eta_+(x_h,s,\xi)ds)\big),\end{split}\eeq and we  define the symbol
for the approximate Dirichlet-Neumann operator as \beq \label{gxh}
g(x_h,\xi^{\eps}){\eqdefa}\sqrt{{\eps}(1+{\eps}^3|\na^{\eps}_h
\zeta|^2) |\xi^{\eps}|^2-{\eps}^4(\xi^{\eps}\cdot\na^{\eps}_h
\zeta)^2}. \eeq Then it follows from \cite{lan1} that
\beq\label{eq:main part of DN}
g(x_h,D^{\eps}_{h})\psi=\p^P_n\psi^b_{app}|_{z=0}. \eeq We'll see
that $g(x_h,D^\eps_h)$ is the principle part of the D-N operator.

The goal of this subsection is to prove the following proposition
which concerns the accuracy of the approximate Dirichlet-Neumann
operator.

\begin{prop}\label{prop:remainder} {\sl Let $s\ge 0,$   $u\in
H^{s+\f12}\cap H^{t_0+\f12}(\R^2)$, and $u^b$ be defined by
\eqref{u^b}. Let \beq\label{eq:remainder} R[\eps\zeta]u{\eqdefa}
G[\eps\zeta]u-g(x_h,D^{\eps}_{h})u=\p^P_n(u^b-u^b_{app})|_{z=0}.
\eeq Then we have \beq\label{prop5.5a}
|R[\eps\zeta]u|_{H^s_{\eps}}\le \sqrt{\eps}
M(\sigma)\Big(|\mathfrak{P}u|_{H^{s}_\eps} +|\mathfrak{P}
u|_{H^{t_0}_\eps}|(\zeta,b)|_{H^{s+3}_\eps}\Big). \eeq}
\end{prop}

\begin{rmk} This estimate is not a standard one since the gain is of
half instead of one derivative compared to similar estimates in
\cite{lan1,MZ}. This is due to the need of $O(\sqrt{\eps})$ term in
the r.h.s. of (\ref{eq:remainder}). In fact, we refer to
\cite{Lannes-stability} for how to gain the full derivative without
losing the $\sqrt \eps$ in the r.h.s. of (\ref{eq:remainder}).

\end{rmk}

We start the proof of this proposition by the following lemma.

\begin{lem}\label{lem:approximate} {\sl Under the assumptions of Proposition \ref{prop:remainder}, we have
\begin{align*}
\begin{split}
\|\Lam^s_{\eps}\na_h^{\eps} u^b_{app}\|_2&\leq
M(\sigma)\Big(|\mathfrak{P}u|_{H^{s}_\eps}
+|\mathfrak{P} u|_{H^{t_0}_\eps}|(\zeta,b)|_{H^{s+2}_\eps}\Big),\\
\|\Lam^s_{\eps}\p_z u^b_{app}\|_2&\leq \sqrt{\eps}
M(\sigma)\Big(|\mathfrak{P}u|_{H^{s}_\eps} +|\mathfrak{P}
u|_{H^{t_0}_{\eps}}|(\zeta,b)|_{H^{s+2}_\eps}\Big).
\end{split}
\end{align*}}
\end{lem}

\begin{proof}\,Thanks to \eqref{eq:main symbol of DN}, we
write
\begin{align}\label{lem5.1a}
\begin{split}
u^b_{app}=&\sigma_{app}(x_h,z,D^{\eps}_{h})\exp(-\f{c_+}2\sqrt{\eps} z|D^{\eps}_{h}|)
\exp(\f{c_+}2\sqrt{\eps} z|D^{\eps}_{h}|)u\\
{\eqdefa}&\widetilde{\sigma}_{app}(x_h,z,D^{\eps}_{h})\exp(\f{c_+}2\sqrt{\eps}
z|D^{\eps}_{h}|)u.
\end{split}
\end{align}
Note by (\ref{eq:eta}) that
$\widetilde{\sigma}_{app}(x_h,z,D^{\eps}_{h})$ is a
pseudo-differential operator of order zero, and
\begin{align*}
\begin{split}
\na_h^{\eps} u^b_{app}=&\Big(-\Big(\int^0_z
(\na_h^{\eps}\eta_+)(\cdot,s,\cdot)ds
\widetilde{\sigma}_{app}\Big)(x_h,z,D^{\eps}_{h})+
\widetilde{\sigma}_{app}(x_h,z,D^{\eps}_{h})\na_h^{\eps}\Big)\exp(\f{c_+}2\sqrt{\eps}
z|D^{\eps}_{h}|)u,
\end{split}
\end{align*}
 which along with Proposition \ref{prop:psdo-estimate} and \beno
 \begin{split}
&\|\Lam^s_\eps|D^{\eps}_{h}|\exp(\f {c_+}2\sqrt{\eps}
z|D^{\eps}_{h}|)u\|^2_2\\
&=\int^0_{-1}\int_{\sqrt{\eps}|\xi^{\eps}|\ge
1}\langle\xi^{\eps}\rangle^{2s}(\sqrt{\eps}|\xi^{\eps}|)\f{|\xi^{\eps}|^2}{\sqrt{\eps}|\xi^{\eps}|}
\exp(\f {c_+}2\sqrt{\eps} z|\xi^{\eps}|)|\hat u|^2d\xi
dz\\
&\quad+\int^0_{-1}\int_{\sqrt{\eps}|\xi^{\eps}|\le
1}\langle\xi^{\eps}\rangle^{2s}|\xi^{\eps}|^2\exp(\f
{c_+}2\sqrt{\eps} z|\xi^{\eps}|)|\hat
u|^2d\xi dz \\
&\le 2\int_{\sqrt{\eps}|\xi^{\eps}|\ge
1}\langle\xi^{\eps}\rangle^{2s}\f{|\xi^{\eps}|^2}{1+\sqrt{\eps}|\xi^{\eps}|}|\hat
u|^2\Big(\int^0_{-1}\exp(C_+\sqrt{\eps} z|\xi^{\eps}|)\sqrt{\eps}|\xi^{\eps}|dz\Big)d\xi\\
&\quad+ 2\int_{\sqrt{\eps}|\xi^{\eps}|\le
1}\langle\xi^{\eps}\rangle^{2s}\f{|\xi^{\eps}|^2}{1+\sqrt{\eps}|\xi^{\eps}|}|\hat
u|^2d\xi \le C|\mathfrak P u|^2_{H^s_\eps},\end{split}\eeno ensures
that
\begin{align*}
\|\Lambda_{\eps}^s\na_h^{\eps} u^b_{app}\|_2&\le
M(\sigma)\big(\|\Lambda_{\eps}^s|D^{\eps}_{h}|\exp(\f{c_+}2\sqrt{\eps}
z|D^{\eps}_{h}|)u\|_2\\
&\quad
+\|\Lambda_{\eps}^{t_0}|D^{\eps}_{h}|\exp(\f{c_+}2\sqrt{\eps} z|D^{\eps}_{h}|)u\|_2|(\zeta,b)|_{H^{s+2}_\eps}\big)\\
&\le M(\sigma)\big(|\mathfrak{P}u|_{H^{s}_\eps} +|\mathfrak{P}
u|_{H^{t_0}_\eps}|(\zeta,b)|_{H^{s+2}_\eps}\big).
\end{align*}
The other estimate of the lemma can be obtained in a similar
way.\end{proof}

\begin{lem}\label{lem:remainder-elliptic} {\sl Under the
assumptions of Proposition \ref{prop:remainder}, we denote
$u_r^b{\eqdefa} u^b-u^b_{app}.$
 Then one has
\begin{align*}
\begin{split}
&\|\Lambda^{s+1}_{\eps}\na^{\eps}u^b_r\|_2\leq
M(\sigma)\Big(|\mathfrak{P}u|_{H^{s}_\eps}
+|\mathfrak{P} u|_{H^{t_0}_\eps}|(\zeta,b)|_{H^{s+3}_\eps}\Big),\\
&\|\Lambda^s_{\eps}\na^{\eps}\p_zu^b_r\|_2\leq \sqrt{\eps}
M(\sigma)\Big(|\mathfrak{P}u|_{H^{s}_\eps} +|\mathfrak{P}
u|_{H^{t_0}_\eps}|(\zeta,b)|_{H^{s+3}_\eps}\Big).
\end{split}
\end{align*}
}
\end{lem}

\begin{proof}\,Thanks to the definition of $u^b_r$, we find out that it
solves \ben\label{eq:elliptic-remainder} \left\{\begin{array}{ll}
{\bf P}u^b_r={\bf P}(u^b-u^b_{app})=-({\bf P}-{\bf
P}_{app})u^b_{app}-{\bf P}_{app}
u^b_{app}\quad \hbox{in}\quad \mathcal S,\\
u^b_r|_{z=0}=0,\qquad\p^P_n u^b_r|_{z=-1}=-\p^P_nu^b_{app}|_{z=-1},
\end{array}\right.
\een where ${\bf P}_{app}$ is defined in \eqref{papp}. Then we get
by applying the first inequality of Proposition \ref{prop:elliptic
estimate-general} that
\begin{align}
\|\Lambda^{s+1}_{\eps} \na^{\eps}u^b_r\|_2 &\ \leq
M(\sigma)\Big(\eps^{-\frac12}
\|\Lambda^s_{\eps}(h^1_{app}+h^2_{app})\|_2+|\p^P_n
u^b_{app}|_{z=-1}|
_{H^{s+1}_{\eps}}\label{eq:elliptic-remainder-est1}\\
&\quad+\big(\eps^{-\frac12}\|\Lambda^{t_0-1}_\eps(h^1_{app}+h^2_{app})\|_2+|\p^P_nu^b_{app}|_{z=-1}|_{H^{t_0}_\eps})
|(\zeta,b)|_{H^{s+2}_\eps}\Big). \nonumber
\end{align}
Here \beno h^1_{app}=({\bf P}-{\bf
P}_{app})u^b_{app}\quad\mbox{and}\quad h^2_{app}={\bf
P}_{app}u^b_{app}. \eeno In what follows , we shall estimate term by
term the right hand side of
(\ref{eq:elliptic-remainder-est1}).\vspace{0.1cm}

\no{\bf Step 1.} The estimate of $h^1_{app}$.

Let \ben
&&\tau_1(x_h,z,D^{\eps}_{h}){\eqdefa}\eta_-(x_h,z,D^{\eps}_{h})\circ\eta_+(x_h,z,D^{\eps}_{h})-
(\eta_-\eta_+)(x_h,z,D^{\eps}_{h}),\nonumber\\
&&\tau_2(x_h,z,D^{\eps}_{h}){\eqdefa}(\p_z\eta_+)(x_h,z,D^{\eps}_{h}).
\label{tau} \een Then we write \beno \begin{split} {\bf P}-{\bf
P}_{app}=&-p_{d+1}\tau_1+p_{d+1}\tau_2+(\p_zp_{d+1}+\na_h\cdot{\bf
p}) \p_z+(\na_h\cdot P_1+\p_z{\bf p})\cdot \na_h.
\end{split}\eeno
While it is easy to observe that
\begin{eqnarray*}
\begin{split}
\na_h\cdot P_1=&{\eps}^2\left(\begin{array}{ll} \p_x(\zeta-b)\\
\eps\p_y(\zeta-b)
\end{array}\right)=- \p_z{\bf p},\\
\na_h\cdot{\bf p}=&\eps|D^{\eps}_{h}|^2\sigma={\eps}^2(-z|D^{\eps}_{h}|^2 b+(z+1)|D^{\eps}_{h}|^2\zeta),\\
\p_zp_{d+1}=&2{\eps}^3\f{(-z \na^{\eps}_h
b+(1-z)\na^{\eps}_h\zeta)\cdot\na^{\eps}_h(\zeta-b)}
{1+{\eps}(\zeta-b)},
\end{split}
\end{eqnarray*}
one has $\na_h\cdot P_1+\p_z{\bf p}=0$ and \beno
\p_zp_{d+1}+\na_h\cdot{\bf p}={\eps}^2F(b,\na_h^{\eps}
b,|D^{\eps}_{h}|^2b,\zeta,\na_h^{\eps} \zeta,|D^{\eps}_{h}|^2\zeta)
\eeno for some smooth function $F$. Then it follows from
(\ref{lem:product estimate})-(\ref{lem:composition}) that
\begin{align*}
\begin{split}
\|\Lambda^s_{\eps} h^1_{app}\|_2 \leq&
M(\sigma)\Big(\|\Lambda^s_{\eps}\tau_1(x_h,z,D^{\eps}_{h})u^b_{app}\|_2+
\|\Lambda^s_{\eps}\tau_2(x_h,z,D^{\eps}_{h})u^b_{app}\|_2\\
&+\eps^2\|\Lambda^s_{\eps}\p_z
u^b_{app}\|_2+\big(\|\Lambda^{t_0}_\eps\tau_1(x_h,z,D^{\eps}_{h})u^b_{app}\|_2\\&
+\|\Lambda^{t_0}_\eps\tau_2(x_h,z,D^{\eps}_{h})u^b_{app}\|_2
+\eps^{\f74}\|\Lambda^{t_0}_\eps\p_z
u^b_{app}\|_2\big)|(\zeta,b)|_{H^{s+2}_\eps}\Big).
\end{split}
\end{align*}
Applying Proposition \ref{prop:psdo-commuatator} and Lemma
\ref{lem:approximate} yields
\begin{align*}
\begin{split}
\|\Lambda^s_{\eps}\tau_1(x_h,z,D^{\eps}_{h})u^b_{app}\|_2 &\leq\eps
M(\sigma)\big(\eps\|\Lambda_{\eps}^s|D^{\eps}_{h}|u^b_{app}\|_2
+\eps^{\f34}\|\Lambda_{\eps}^{t_0}|D^{\eps}_{h}|u^b_{app}\|_2|(\zeta,b)|_{H^{s+2}_\eps}\big)\\
&\le \eps^{\f74}M(\sigma)\big(|\mathfrak{P}u|_{H^{s}_\eps}
+|\mathfrak{P} u|_{H^{t_0}_\eps}|(\zeta,b)|_{H^{s+2}_\eps}\big).
\end{split}
\end{align*}
Similarly applying Proposition \ref{prop:psdo-estimate} and Lemma
\ref{lem:approximate} gives
\begin{align*}
\|\Lambda^s_{\eps}\tau_2(x_h,z,D^{\eps}_{h})u^b_{app}\|_2 \le
\eps^{\f74}M(\sigma)\big(|\mathfrak{P}u|_{H^{s}_\eps} +|\mathfrak{P}
u|_{H^{t_0}_\eps}|(\zeta,b)|_{H^{s+2}_\eps}\big).
\end{align*}
As a consequence, we obtain \beq\label{eq:elliptic-remainder-est2}
\|\Lambda^s_{\eps} h^1_{app}\|_2
\leq\eps^{\f74}M(\sigma)\big(|\mathfrak{P}u|_{H^{s}_\eps}
+|\mathfrak{P} u|_{H^{t_0}_\eps}|(\zeta,b)|_{H^{s+2}_\eps}\big).
\eeq

\no{\bf Step 2.}\, The estimate of $h^2_{app}$.\,

Thanks to \eqref{eq:main symbol of DN}, we write \beno
\begin{split}
h^2_{app}=&{\bf P}_{app}u^b_{app}
=p_{d+1}(\p_z-\eta_-(x_h,z,D^{\eps}_{h}))
\tau_3(x_h,z,D^{\eps}_{h})\exp(\f{c_+}2z\sqrt{\eps}|D^{\eps}_{h}|)u,
\end{split}\eeno
where  $\widetilde\sigma_{app}$ is given by \eqref{lem5.1a} and
\[
\tau_3(x_h,z,D^{\eps}_{h}){\eqdefa}\textrm{Op}_{\eps}(\eta_+\widetilde\sigma_{app})
-\textrm{Op}_{\eps}(\eta_+)\circ\widetilde\sigma_{app}(x_h,z,D^{\eps}_{h}).
\]
Applying Proposition \ref{prop:psdo-estimate} gives
\begin{align*}
\begin{split}
\|\Lambda^s_{\eps} h^2_{app}\|_2\le&
M(\sigma)\Big(\|\Lam^s_{\eps}\na^{\eps}\tau_3(x_h,z,
D^{\eps}_{h})\exp(\f{c_+}2z\sqrt{\eps} |D^{\eps}_{h}|)u \|_2\\
&+\|\Lambda^{t_0}_\eps\na^{\eps}\tau_3(x_h,z,D^{\eps}_{h})\exp(\f{c_+}2z\sqrt{\eps}|D^{\eps}_{h}|)u
\|_2|(\zeta,b)|_{H^{s+2}_{\eps}}\Big).
\end{split}
\end{align*}
As in the proof of Proposition \ref{prop:psdo-estimate}, we split
$u$ into the low frequency and high frequency parts so that
$u=u_{lf}+u_{hf}$ with  $ u_{lf}=\psi(D^{\eps}_{h})u$. Then we
deduce from Proposition \ref{prop:psdo-commuatator} and the proof of
Lemma \ref{lem:approximate} that \beno
\begin{split}
\|\Lam^s_{\eps}\na^{\eps}\tau_3(x_h,z,D^{\eps}_{h})\exp(\f{c_+}2z\sqrt{\eps}|D^{\eps}_{h}|)u_{hf}\|_2
&\le
\eps^{\f74}M(\sigma)\Big(\|\Lambda^{s+1}_{\eps}\exp(\f{c_+}2z\sqrt{\eps}|D^{\eps}_{h}|)u_{hf}\|_2\\
&\quad+
\|\Lambda^{t_0+1}_\eps\exp(\f{c_+}2z\sqrt{\eps}|D^{\eps}_{h}|)u_{hf}\|_2|(\zeta,b)|_{H^{s+3}_\eps}\Big)\\
&\le \eps^{\f74}M(\sigma)\big(|\mathfrak{P}u|_{H^{s}_\eps}
+|\mathfrak{P} u|_{H^{t_0}_\eps}|(\zeta,b)|_{H^{s+3}_\eps}\big),
\end{split}
\eeno and it follows from a similar procedure as that used in
handling $u_{lf}$ in Proposition \ref{prop:psdo-commuatator} that
\begin{align*}
\begin{split}
\|\Lam^s_{\eps}\na^{\eps}\tau_3(x_h,z,D^{\eps}_{h})\exp(\f{c_+}2z\sqrt{\eps}|D^{\eps}_{h}|)u_{lf}\|_2
&\le
\eps^{\f32}M(\sigma)\||D^{\eps}_{h}|\exp(\f{c_+}2z\sqrt{\eps}|D^{\eps}_{h}|)u\|_2|(\zeta,b)|_{H^{s+3}_\eps}\\
&\le \eps^{\f32}M(\sigma)|\mathfrak{P}
u|_{2}|(\zeta,b)|_{H^{s+3}_\eps}.
\end{split}
\end{align*}
Whence  we obtain \beq\label{eq:elliptic-remainder-est3}
\|\Lambda^s_{\eps} h^2_{app}\|_2
\leq\eps^{\f32}M(\sigma)\big(|\mathfrak{P}u|_{H^{s}_\eps}
+|\mathfrak{P} u|_{H^{t_0}_\eps}|(\zeta,b)|_{H^{s+3}_\eps}\big).
\eeq

\no{\bf Step 3.}\, The estimate of
$\p^P_nu^b_{app}|_{z=-1}$.\,

Noticing that
\begin{align*}
&\p^P_nu^b_{app}|_{z=-1}=-{\bf e_3} \cdot(1+Q[\sigma])\na^{\eps}u^b_{app}|_{z=-1}\\
&={\eps}^2\na^{\eps}_h b\cdot\na^{\eps}_h
\sigma_{app}(x_h,-1,D^{\eps}_{h})u
-p_{d+1}|_{z=-1}(\eta_+\sigma_{app})(x_h,-1,D^{\eps}_{h})u,
\end{align*}
which together with  (\ref{lem:product
estimate})-(\ref{lem:composition}) implies that \beno
\begin{split}
|\p^P_nu^b_{app}|_{z=-1}|_{H^{s+1}_{\eps}}\leq&
M(\sigma)\Big(\eps^2||D^{\eps}_{h}|\sigma_{app}(x_h,-1,D^{\eps}_{h})u|
_{H^{s+1}_{\eps}}\\
&\ +|(\eta_+\sigma_{app})(x_h,-1,D^{\eps}_{h})u|_{H^{s+1}_{\eps}}
+\big({\eps}^2||D^{\eps}_{h}|\sigma_{app}(x_h,-1,D^{\eps}_{h})u|_{H^{t_0}_\eps}\\
&\
+|(\eta_+\sigma_{app})(x_h,-1,D^{\eps}_{h})u|_{H^{t_0}_\eps}\big)|(\zeta,b)|_{H^{s+2}_\eps}\Big).
\end{split}
\eeno It is easy to observe from the proof of Lemma
\ref{lem:approximate} that \beno
\begin{split}
\sqrt{\eps}||D^{\eps}_{h}|\sigma_{app}(x_h,-1,D^{\eps}_{h})u|_{H^{s+1}_{\eps}}
&\le
M(\sigma)\Big(|(1+\sqrt{\eps}|D^{\eps}_{h}|)^{-\f12}|D^{\eps}_{h}|u|_{H^{s}_\eps}\\
&\quad+||D^{\eps}_{h}|(1+\sqrt{\eps}|D^{\eps}_{h}|)^{-\f12}u|_{H^{t_0}_{\eps}}|(\zeta,b)|_{H^{s+3}_\eps}\Big)\\
&\le M(\sigma)\big(|\mathfrak{P}u|_{H^{s}_\eps} +|\mathfrak{P}
u|_{H^{t_0}_\eps}|(\zeta,b)|_{H^{s+3}_\eps}\big), \end{split} \eeno
and similarly we have \beno
|(\eta_+\sigma_{app})(x_h,-1,D^{\eps}_{h})u|_{H^{s+1}_{\eps}}\le
M(\sigma)\big(|\mathfrak{P}u|_{H^{s}_{\eps}} +|\mathfrak{P}
u|_{H^{t_0}_{\eps}}|(\zeta,b)|_{H^{s+3}_{\eps}}\big). \eeno
Therefore, we obtain that \ben\label{eq:elliptic-remainder-est4}
|\p^P_nu^b_{app}|_{z=-1}|_{H^{s+1}_{\eps}} \leq
M(\sigma)\big(|\mathfrak{P}u|_{H^{s}_{\eps}} +|\mathfrak{P}
u|_{H^{t_0}_{\eps}}|(\zeta,b)|_{H^{s+3}_{\eps}}\big). \een The above
arguments also imply that \ben\label{eq:elliptic-remainder-est5}
|\p^P_nu^b_{app}|_{z=-1}|_{H^{s}_{\eps}} \leq \sqrt{\eps}
M(\sigma)\big(|\mathfrak{P}u|_{H^{s}_{\eps}} +|\mathfrak{P}
u|_{H^{t_0}_{\eps}}|(\zeta,b)|_{H^{s+3}_{\eps}}\big). \een

Plugging
(\ref{eq:elliptic-remainder-est2})-(\ref{eq:elliptic-remainder-est4})
into (\ref{eq:elliptic-remainder-est1}) yields the first estimate of
the lemma. The second inequality of the lemma can be deduced from
the third inequality of Proposition \ref{prop:elliptic
estimate-general} and
(\ref{eq:elliptic-remainder-est2})-(\ref{eq:elliptic-remainder-est4}).
This finishes the proof of  Lemma
\ref{lem:remainder-elliptic}.\end{proof}

With the above two lemmas, we can complete the proof of Proposition
\ref{prop:remainder}.

\no{\bf Proof of Proposition \ref{prop:remainder}.}\,\,Thanks to
(\ref{eq:remainder}), for any $v\in \cS(\R^2)$, we get by applying
Green's identity  that
\begin{align*}
\begin{split}
\bigl(&\Lambda^s_{\eps} R[\eps\zeta]u,v)=\big(\p^P_n
u^b_r|_{z=0},\Lambda^s_{\eps} v\bigr)\\
&=-\big(\p^P_n u^b_r|_{z=-1},\Lambda^s_\eps v^\dag|_{z=-1}\big)+
\int_{\mathcal S}\bigl({\bf P}u^b_r\,\Lambda^s_{\eps}
v^\dag+(1+Q[\sigma])\na^{\eps}u^b_r\cdot\na^{\eps}\Lambda^s_{\eps} v^\dag\bigr)\,dx_h\,dz\\
&=\big(\Lambda^s_{\eps}\p^P_n
u^b_{app}|_{z=-1},\chi(-\sqrt{\eps}|D^{\eps}_{h}|)v\big)\\
&\quad- \int_{\mathcal S}\bigl(\Lambda^s_{\eps}{\bf
P}u^b_{app}\,v^\dag-
\Lam^s_{\eps}(1+Q[\sigma])\na^{\eps}u^b_r\cdot\na^{\eps}v^\dag\bigr)\,dx_h\,dz.
\end{split}
\end{align*}
As $\|v^\dag\|_2\le C|v|_2,$ applying Lemma \ref{lem:duality} below
ensures that \beno \begin{split}\big|\big(\Lambda^s_{\eps}
R[\eps\zeta]u,v\big)\big|\leq& C|v|_2
\Big(|\Lambda^s_{\eps}\p^P_nu^b_{app}|_{z=-1}|_2
+\|\Lambda^s_{\eps}{\bf P}u^b_{app}\|_2+\sqrt{\eps}
\|\Lambda^{s+1}_{\eps}
(1+Q[\sigma])\na^{\eps}u^b_r\|_2\Big),\end{split} \eeno which
together with (\ref{eq:elliptic-remainder-est2}),
(\ref{eq:elliptic-remainder-est3}),
(\ref{eq:elliptic-remainder-est5}) and Lemma
\ref{lem:remainder-elliptic} implies that
\begin{eqnarray*}
\big|\big(\Lambda^s_{\eps} R[\eps\zeta]u,v\big)\big|\le \sqrt{\eps}
M(\sigma)|v|_2 \big(|\mathfrak{P}u|_{H^{s}_{\eps}} +|\mathfrak{P}
u|_{H^{t_0}_{\eps}}|(\zeta,b)|_{H^{s+3}_{\eps}}\big).
\end{eqnarray*}
This proves \eqref{prop5.5a} by duality.\ef

\subsection{Commutator estimates}

In this subsection, we shall present several useful  commutator
estimates between the  Dirichlet-Neumann operator and the elliptic
operator $\fd_{\eps}(\zeta)$ defined by \eqref{eq:ellipticoperator}.

\begin{prop}\label{prop:DN-commutator} {\sl Let $ k\in \N,$   and
$\zeta, b\in H^{m_0}\cap H^{2k+2}(\R^2)$ satisfy (\ref{ass:free
surface}). Then for any $u\in H^{t_0+\f32}\cap H^{2k+\f12}(\R^2)$,
there holds \beq\label{eq:commutator-general} \big|\big[\f1\eps
G[\eps \zeta],\fd_{\eps}(\zeta)^k\big]u\big|_2\le \eps
M(\sigma)\Big(|\mathfrak Pu|_{H^{2k}_\eps}+|\mathfrak
Pu|_{H^{t_0+1}}|(\zeta,b)|_{H^{2k+2}_\eps}\Big). \eeq }\end{prop}

\begin{proof}\, Thanks to \eqref{Qd} and \eqref{eq:DN-representation}, for any $v\in \cS(\R^2)$, we get by
applying Green's identity  that \beq\label{prop5.6a}
\begin{split}
\big([G[\eps \zeta],&\fd_{\eps}(\zeta)^k]u,v\big)=\big(G[\eps \zeta]\fd_{\eps}(\zeta)^ku,v\big)-\big(\fd_{\eps}(\zeta)^kG[\eps \zeta]u,v\big)\\
&=\int_\cS\Bigl\{(1+Q[\si])\na^{\eps}(\fd_{\eps}(\zeta)^k u)^b\cdot
\na^{\eps}v^\dag- \fd_{\eps}(\zeta)^k(1+Q[\si])\na^{\eps}u^b\cdot\na^{\eps}v^\dag\\
&\qquad- \na^{\eps}\cdot\big(\fd_{\eps}(\zeta)^k(1+Q[\si])\na^{\eps}u^b\big)v^\dag\Bigr\}\,dx_h\,dz\\
&=\int_\cS\Bigl\{[(1+Q[\si])\na^{\eps},\fd_{\eps}(\zeta)^k]u^b\cdot\na^{\eps}v^\dag
\\
&\qquad+(1+Q[\si])\na^{\eps}((\fd_{\eps}(\zeta)^ku)^b-\fd_{\eps}(\zeta)^ku^b)\cdot\na^{\eps}
v^\dag\\
&\qquad-\big[\na^{\eps},\fd_{\eps}(\zeta)^k\big]\cdot\big((1+Q[\si])\na^{\eps}u^b\big)v^\dag\Bigr\}\,dx_h\,dz\\
&{\eqdefa} A_1+A_2+A_3.\end{split} \eeq

To deal with $A_1, A_2$, we need the following lemma, which can be
deduced from the proof of Lemma 3.1 in \cite{Lannes-Inven}.

\begin{lem}\label{lem:duality}{\sl For all $f\in L^2(\R^2)$ and
$\textbf{g}\in H^1(\cS)^3$, one has \beno
\bigl|\int_{\cS}\na^{\eps}f^\dag\cdot \textbf{g}\,dx_h\,dz\bigr|\le
C\sqrt{\eps}|f|_2\|\Lam_\eps \textbf{g}\|_2. \eeno} \end{lem}

Applying Lemma \ref{lem:duality} to $A_1$ gives \beno |A_1| \le
C\sqrt{\eps}|v|_2\|\Lam_\eps\big[(1+Q[\sigma])\na^{\eps},\fd_{\eps}(\zeta)^k\big]u^b\|_2,
\eeno but as \beno
\begin{split}
&\big[(1+Q[\sigma])\na^{\eps},\fd_{\eps}(\zeta)^k\big]u^b=(1+Q[\sigma])\big[\na^{\eps},\fd_{\eps}(\zeta)^k\big]u^b+\big[Q[\sigma],\fd_{\eps}(\zeta)^k\big]\na^{\eps}u^b,
\end{split}
\eeno from which,  Lemma \ref{lem:elliptic operator-commutator} and
Corollary \ref{cor:elliptic estimate}, we deduce that
\begin{align*}
\begin{split}
\|\Lam_{\eps}\big[(1+Q[\sigma])\na^{\eps},\fd_{\eps}(\zeta)^k\big]u^b\|_2
&\ \le \eps
M(\sigma)\big(\|\Lam^{2k}_\eps\na^{\eps}u^b\|_2+\|\Lam^{t_0+1}\na^{\eps}u^b\|_2
|(\zeta,b)|_{H^{2k+2}_{\eps}}\big)\\
&\ \le \eps^{\f32} M(\sigma)\big(|\mathfrak
Pu|_{H^{2k}_\eps}+|\mathfrak
Pu|_{H^{t_0+1}}|(\zeta,b)|_{H^{2k+2}_{\eps}}\big).
\end{split}
\end{align*}
As a consequence, we obtain \beq\label{prop5.6b} |A_1|\le
\eps^2M(\sigma)|v|_2\big(|\mathfrak Pu|_{H^{2k}_\eps}+|\mathfrak
Pu|_{H^{t_0+1}}|(\zeta,b)|_{H^{2k+2}_\eps}\big). \eeq Applying Lemma
\ref{lem:duality} again, we have \beno |A_2|\le \sqrt{\eps}
M(\sigma)|v|_2\|\Lam_{\eps}\na^{\eps}((\fd_{\eps}(\zeta)^ku)^b-\fd_{\eps}(\zeta)^ku^b)\|_2.
\eeno Thanks to \eqref{u^b}, we find that
$(\fd_{\eps}(\zeta)^ku)^b-\fd_{\eps}(\zeta)^ku^b$ solves
\[ \left\{\begin{array}{ll}
\na^{\eps}\cdot (1+Q[\sigma])\na^{\eps}((\fd_{\eps}(\zeta)^ku)^b-\fd_{\eps}(\zeta)^ku^b)=\textbf{g},\\
(\fd_{\eps}(\zeta)^ku)^b-\fd_{\eps}(\zeta)^ku^b|_{z=0}=0,\\
\p_n((\fd_{\eps}(\zeta)^ku)^b-\fd_{\eps}(\zeta)^ku^b)|_{z=-1}={\bf
e_{3}}\cdot
\big[(1+Q)\na^{\eps},\fd_{\eps}(\zeta)^k\big]u^b|_{z=-1},
\end{array}\right.
\]
where
\[ \textbf{g}\eqdefa -\big[\na^{\eps},\fd_{\eps}(\zeta)^k\big]\cdot(1+Q[\sigma])\na^{\eps}u^b
-\na^{\eps}\cdot\big[(1+Q[\sigma])\na^{\eps},\fd_{\eps}(\zeta)^k\big]u^b.\]
Then we deduce from Proposition \ref{prop:elliptic estimate} and
Proposition \ref{prop:elliptic estimate-general} that \beno
\begin{split}
&\|\Lam_{\eps}\na^{\eps}((\fd_{\eps}(\zeta)^ku)^b-\fd_{\eps}(\zeta)^ku^b)\|_2\\
&\ \le
M(\sigma)\big(\|\Lam_{\eps}[(1+Q[\sigma])\na^{\eps},\fd_{\eps}(\zeta)^k]u^b\|_2+
\eps^{-\f12}\|[\na^{\eps},\fd_{\eps}(\zeta)^k]\cdot(1+Q[\sigma])\na^{\eps}u^b\|_2\big)\\
&\ \le \eps^{\f32} M(\sigma)\big(|\mathfrak
Pu|_{H^{2k}_\eps}+|\mathfrak
Pu|_{H^{t_0+1}}|(\zeta,b)|_{H^{2k+2}_\eps}\big),
\end{split}
\eeno which gives \beq\label{prop5.6c} |A_2|\le
\eps^2M(\sigma)|v|_2\big(|\mathfrak Pu|_{H^{2k}_\eps}+|\mathfrak
Pu|_{H^{t_0+1}}|(\zeta,b)|_{H^{2k+2}_\eps}\big). \eeq To deal with
$A_3$, we apply Lemma \ref{lem:Poisson regularization}, Lemma
\ref{lem:elliptic operator-commutator} and Corollary
\ref{cor:elliptic estimate} to obtain
\begin{align*}
\begin{split}
|A_3|&\le c_2\|\big[\na^{\eps},\fd_{\eps}(\zeta)^k\big]\cdot\big((1+Q[\si])\na^{\eps}u^b\|_2|v|_2\\
&\le \eps^{\f52} M(\sigma)\big(|\mathfrak
Pu|_{H^{2k}_\eps}+|\mathfrak
Pu|_{H^{t_0+1}}|(\zeta,b)|_{H^{2k+2}_\eps}\big)|v|_2,
\end{split}
\end{align*}
which along with (\ref{prop5.6a}-\ref{prop5.6c}) concludes that
\beno \bigl|\big([G[\eps
\zeta],\fd_{\eps}(\zeta)^k]u,v\big)\bigr|\le
\eps^2M(\sigma)|v|_2\big(|\mathfrak Pu|_{H^{2k}_\eps}+|\mathfrak
Pu|_{H^{t_0+1}}|(\zeta,b)|_{H^{2k+2}_\eps}\big), \eeno and this
implies
 (\ref{eq:commutator-general}).\end{proof}

\begin{rmk}\label{rem:DN-commutator} It is easy to observe from the proof of Proposition
\ref{prop:DN-commutator} that \beno
\begin{split}
\big||D^{\eps}_{h}|^m\big[\Lam^s,\f1\eps G[\eps
\zeta]\big]u\big|_2\le \eps M(\si)\Big(|\Lam^{m+1}_\eps\mathfrak
Pu|_{H^{s-1}}+|\mathfrak Pu|_{H^{t_0+2}}|\Lam^{3}_\eps(\zeta,
b)|_{H^{s}}\Big) \end{split} \eeno  for $ m=0,1,$ which will be used
later in the lower order energy estimate.
\end{rmk}

In order to deal with the energy estimate for the linearized system
of \eqref{eq:Hamiltonian form-non}, we need the following sharper
commutator estimate.

\bthm{Theorem}\label{thm:DN-commutator-sharp} {\sl Let $s\ge 0, k\in
\N$, and $\zeta, b\in H^{m_0}\cap H^{2k+s+3}(\R^2)$ satisfy
(\ref{ass:free surface}). We denote $\rho(\zeta){\eqdefa}
(1+\eps^3|\na^\eps_h\zeta|^2)^{\f12}.$ Then for any $u\in
H^{2k+s-\f12}\cap H^{t_0+\f52}(\R^2)$, we have
\beq\label{eq:commutator-sharp} \begin{split} &\bigl|\big[\f
1\varepsilon \rho(\zeta)^{-1}G[{\eps}
\zeta],\fd_{\eps}(\zeta)^k\big]u\bigr|_{H^s_{\eps}} \leq \sqrt{\eps}
M(\sigma)\Big(|\mathfrak P u|_{H^{2k+s-1}_\eps}
+|\mathfrak{P}u|_{H^{t_0+2}}|(\zeta,b)|_{H^{2k+s+3}_\eps}\Big).
\end{split}\eeq }\ethm

\begin{rmk} Compared with (\ref{eq:commutator-general}), the
commutator estimate (\ref{eq:commutator-sharp}) gains one more
derivative. The key observation used to prove this theorem is that
the symbol of the principle part of the operator
$\rho(\zeta)^{-1}G[\eps\zeta]$ is the same as the square root of the
symbol of  $\fd_{\eps}(\zeta)$.
\end{rmk}

In what follows, we divide the proof of Theorem
\ref{thm:DN-commutator-sharp}  into two parts. In the first part, we
deal with the commutator estimate between the principle part of DN
operator and $\fd_{\eps}(\zeta)$.

\begin{prop}\label{prop:main part-commutator}{\sl Let $s\ge 0,$
$u\in H^{s+1}\cap H^{t_0+2}(\R^2),$ and $g(x_h,\xi^\eps)$ be
determined by \eqref{gxh}. Then under the assumptions of Theorem
\ref{thm:DN-commutator-sharp}, one has \beq\label{prop5.7}
\begin{split}
&\bigl|\big[\frac 1{\eps}
\rho(\zeta)^{-1}g(x_h,D^{\eps}_{h}),\fd_{\eps}(\zeta)\big]u\bigr|
_{H^{s}_\eps}\le {\eps}^\f94
M(\sigma)\Big(|\mathfrak{P}u|_{H^{s+\f12}_\eps}
+\langle|\mathfrak{P}u|_{H^{t_0+\f32}_\eps}|\zeta|_{H^{s+3}_\eps}\rangle_{s>t_0+1}\Big).
\end{split}
\eeq}
\end{prop}

\begin{proof} It is easy to observe from \eqref{poisson} and
\eqref{eq:main symbol of DN} that \beno \big\{\rho(\zeta)^{-1}
g(x_h,\xi^{\eps}),\fd_{\eps}(\zeta)(x,\xi^{\eps})\big\}_1=0, \eeno
where $\fd_{\eps}(\zeta)(x,\xi^{\eps})$ denotes the symbol of the
pseudo-differential  operator  with $\fd_{\eps}(\zeta)$ being
defined in \eqref{eq:ellipticoperator} so that $
\fd_{\eps}(\zeta)(x,\xi^{\eps})=|\xi^{\eps}|^2-{\eps}^3\rho(\zeta)^{-2}(\na^{\eps}_h
\zeta\cdot \xi^{\eps})^2. $ Then we have \beno
\begin{split}
\big[\rho(\zeta)^{-1}g(x_h,D^{\eps}_{h}),\fd_{\eps}(\zeta)\big]
=&\big[\rho(\zeta)^{-1}g(x_h,D^{\eps}_{h}),\fd_{\eps}(\zeta)\big]-\textrm{Op}_\eps\big\{\rho(\zeta)^{-1}
g(x_h,\xi),\fd_{\eps}(\zeta)(x,\xi)\big\}_1,
\end{split}\eeno
from which and Proposition \ref{prop:psdo-commuatator}, we infer
that \beno\begin{split}
&\big|\big[\rho(\zeta)^{-1}g(x_h,D^{\eps}_{h}),\fd_{\eps}(\zeta)\big]u\big|_{H^s_\eps}\le
\eps^{\f12}M(\sigma)\Big(\eps^{\f{11}4}||D^{\eps}_{h}|u|_{H^{s}_{\eps}}
+\eps^{\f{11}4}\langle|\zeta|_{H^{s+3}_{\eps}}||D^{\eps}_{h}|u|_{H^{t_0+1}_{\eps}}\rangle_{s>t_0+1}\Big),
\end{split}
\eeno which together with the fact that
\[
||D^{\eps}_{h}|u|_{H^s_{\eps}}\leq
\Big|\f{|D^{\eps}_{h}|(1+\sqrt{\eps}|D^{\eps}_{h}|)^{\f12}}
{(1+\sqrt{\eps}|D^{\eps}_{h}|)^{\f12}} u\Big|_{H^s_{\eps}} \leq
|\mathfrak{P}u|_{H^{s+\f12}_\eps}
\]
implies \eqref{prop5.7}.\end{proof}

Next let us turn to the commutator estimate between $R[\eps\zeta]$
and $\fd_{\eps}(\zeta)$.

\begin{lem}\label{lem:remainder-elliptic-comm}{\sl Let $s\ge 0,$  $u\in
H^{s+\f32}\cap H^{t_0+\f52}(\R^2).$ Let $u^b$ be given by
\eqref{u^b} and $u^b_r$ be given by Lemma
\ref{lem:remainder-elliptic}. We denote $w_r^b{\eqdefa}
(\fd_{\eps}(\zeta) u)^b_{r}-\fd_{\eps}(\zeta) u^b_{r}.$ Then one has
\begin{align}\label{lemma5.4}
\begin{split}
\|\Lambda^{s+1}_{\eps}&\na^{\eps}w^b_r\|_2+\eps^{-\f12}\|\Lambda^{s}_\eps\na^{\eps}\p_zw^b_r\|_2\\
&\le \eps
M(\sigma)\Big(|\mathfrak{P}u|_{H^{s+1}_\eps}+|\mathfrak{P}u|_{H^{t_0}}|(\zeta,b)|_{H^{s+5}_\eps}\\
&\
+\langle|\mathfrak{P}u|_{H^{t_0+1}}|(\zeta,b)|_{H^{s+5}_\eps}\rangle_{s>t_0}
+\langle|\mathfrak{P}u|_{H^{t_0+2}}|(\zeta,b)|_{H^{s+5}_\eps}\rangle_{s>t_0+1}\Big).
\end{split}
\end{align}
 }
\end{lem}

\begin{proof}\, Thanks to (\ref{eq:elliptic-remainder}), we have
\begin{align}\label{lemma5.4a}
\begin{split}
{\bf P}w^b_r=&\big[\fd_{\eps}(\zeta),{\bf P}-{\bf
P}_{app}\big]u^b_{app}
+\big({\bf P}-{\bf P}_{app}\big)\big(\fd_{\eps}(\zeta) u^b_{app}\\
&-(\fd_{\eps}(\zeta) u)^b_{app}\big)+\big(\fd_{\eps}(\zeta) {\bf
P}_{app}u^b_{app}-{\bf P}_{app}(\fd_{\eps}(\zeta) u)^b_{app}\big)
+\big[\fd_{\eps}(\zeta), {\bf P}\big]u^b_r\\{\eqdefa}
&h_1+h_2+h_3+h_4,
\end{split}
\end{align}
together with the boundary conditions \beno w^b_r|_{z=0}=0,\quad
\p_n^Pw^b_r|_{z=-1}=g_1+g_2 \eeno where \beno
\begin{split}
&g_1\eqdefa e_{3}\cdot\big[(1+Q[\sigma])\na^{\eps},\fd_{\eps}(\zeta)\big]u^b_r|_{z=-1},\\
&g_2\eqdefa -e_{3}\cdot\big[\fd_{\eps}(\zeta),
(1+Q[\sigma])\na^{\eps}\sigma_{app}(x_h,z,D^{\eps}_{h})\big]u|_{z=-1}.
\end{split}
\eeno

In what follows, we just consider the case of $s>t_0+1$, the other
cases can be handled in a similar way. In this case, we first get by
applying  Proposition \ref{prop:elliptic estimate-general} to
\eqref{lemma5.4a} that \beq\label{eq:rcomm-est-1}
\begin{split}
\|\Lambda^{s+1}_{\eps}&\na^{\eps}w^b_r\|_2 \le
M(\sigma)\Big(\eps^{-\f12}\|\Lam^s_{\eps}(h_1+\cdots+h_4)\|_2+\eps^{-\f14}|g_1|_{H^{s+\f12}_\eps}
\\&+|g_2|_{H^{s+1}_\eps}+\big(\eps^{-\f12}\|\Lambda^{t_0-1}_\eps(h_1+\cdots+h_4)\|_2+ |g_1+g_2|_{H^{t_0}_\eps}\big)|(\zeta,b)|_{H^{s+2}_\eps}\Big),
\end{split}
\eeq which reduces the estimate of \eqref{lemma5.4} to that of $h_1,
h_2, h_3, h_4$ and $g_1, g_2.$

\no{\bf Step 1.}\, The estimate of $h_1$.\,

 Recall from the proof of
Lemma \ref{lem:remainder-elliptic} that
\[
{\bf P}-{\bf
P}_{app}=-p_{d+1}\tau_1+p_{d+1}\tau_2+(\p_zp_{d+1}+\na_h\cdot{\bf
p})\p_z,
\]
 so we write \beno\begin{split} \big[\fd_{\eps}(\zeta),{\bf P}-{\bf
P}_{app}\big]=&-\big[\fd_{\eps}(\zeta),p_{d+1}\tau_1\big]
+\big[\fd_{\eps}(\zeta),p_{d+1}\tau_2\big]+\big[\fd_{\eps}(\zeta),
\p_zp_{d+1}+\na_h\cdot{\bf p}\big]\p_z.
\end{split}
\eeno To deal with $\big[\fd_{\eps}(\zeta),p_{d+1}\tau_1\big],$
thanks to \eqref{tau}, we can  split
$\big[\fd_{\eps}(\zeta),\tau_1\big]$ as
\begin{align*}
\big[\fd_{\eps}(\zeta),\tau_1\big]
=&\big(\big[\fd_{\eps}(\zeta),\eta_{-}\big]-\textrm{Op}_{\eps}\{\fd_{\eps}(\zeta),
\eta_-\}_1\big)\circ \textrm{Op}_\eps(\eta_+)\\
&+\textrm{Op}_{\eps}(\eta_-)\circ\big(\big[\fd_{\eps}(\zeta),\eta_{+}\big]-\textrm{Op}_{\eps}\{\fd_{\eps}(\zeta),\eta_+\}_1\big)\\
&+\big(\textrm{Op}_{\eps}\{\fd_{\eps}(\zeta),\eta_-\}_1\circ\textrm{Op}_{\eps}(\eta_+)-
\textrm{Op}_{\eps}\big(\{\fd_{\eps}(\zeta),\eta_-\}_1\eta_+\big)\big)\\
&+\big(\textrm{Op}_{\eps}(\eta_-)\circ\textrm{Op}_{\eps}\{\fd_{\eps}(\zeta),\eta_+\}_1-
\textrm{Op}_{\eps}\big(\eta_-\{\fd_{\eps}(\zeta),\eta_+\}_1\big)\big)\\
&+\big(\textrm{Op}_{\eps}\{\fd_{\eps}(\zeta),\eta_-\eta_+\}_1
-\big[\fd_{\eps}(\zeta),\textrm{Op}_{\eps}(\eta_-\eta_+)\big]\big),
\end{align*}
from which and Propositions
\ref{prop:psdo-estimate}-\ref{prop:psdo-commuatator} and Lemma
\ref{lem:approximate},  we deduce that
\begin{align*}
\begin{split}
\|\Lambda_{\eps}^s\big[\fd_{\eps}(\zeta),\tau_1\big]u^b_{app}\|_2 &\
\le \eps^2
M(\sigma)\big(\|\Lambda^{s+1}_{\eps}|D^{\eps}_{h}|u_{app}^b\|_2
+\eps^{-\f14}|(\zeta,b)|_{H^{s+4}_\eps}\|\Lambda^{t_0+2}_\eps|D^{\eps}_{h}|u_{app}^b\|_2\big)\\
&\ \le \eps^{\f74}M(\sigma)\big(|\mathfrak{P}u|_{H^{s+1}_\eps}
+|\mathfrak{P}u|_{H^{t_0+2}_\eps}|(\zeta,b)|_{H^{s+4}_\eps}\big).
\end{split}
\end{align*}
This ensures
\begin{align*}
\begin{split}
\|\Lambda_{\eps}^s\big[\fd_{\eps}(\zeta),p_{d+1}\tau_1\big]u^b_{app}\|_2
&\ \le
\|\Lambda_{\eps}^s\big[\fd_{\eps}(\zeta),p_{d+1}\big]\tau_1u^b_{app}\|_2
+\|\Lambda_{\eps}^sp_{d+1}\big[\fd_{\eps}(\zeta),\tau_1\big]u^b_{app}\|_2\nonumber\\
&\ \le \eps^{\f74}M(\sigma)\big(|\mathfrak{P}u|_{H^{s+1}_\eps}
+|\mathfrak{P}u|_{H^{t_0+2}_\eps}|(\zeta,b)|_{H^{s+4}_\eps}\big).
\end{split}
\end{align*}
Exactly following the same line, we  obtain
\begin{align*}
&\|\Lambda_{\eps}^s\big[\fd_{\eps}(\zeta),p_{d+1}\tau_2\big]u^b_{app}\|_2
+\|\Lambda_{\eps}^s\big[\fd_{\eps}(\zeta),(\p_zp_{d+1}+\na_h\cdot{\bf p})\big]\p_zu^b_{app}\|_2\\
&\le \eps^{\f74}M(\sigma)\big(|\mathfrak{P}u|_{H^{s+1}_\eps}
+|\mathfrak{P}u|_{H^{t_0+2}_\eps}|(\zeta,b)|_{H^{s+4}_\eps}\big).
\end{align*}
Consequently, we arrive at
\begin{align}\label{eq:rcomm-est-2}
\|\Lambda_{\eps}^s h_1\|_2 \le
\eps^{\f74}M(\sigma)\big(|\mathfrak{P}u|_{H^{s+1}_\eps}
+|\mathfrak{P}u|_{H^{t_0+2}_\eps}|(\zeta,b)|_{H^{s+4}_\eps}\big).
\end{align}

\no{\bf Step 2.} The estimate of $h_2$.\,

Set
\begin{align*}
\begin{split}
w^b_{app}{\eqdefa}&\fd_{\eps}(\zeta) u^b_{app}-(\fd_{\eps}(\zeta)
u)^b_{app}\\
=& \big[\fd_{\eps}(\zeta),
\widetilde{\sigma}_{app}(x_h,z,D^{\eps}_{h})\big]\exp\big(\f{c_+}2z\sqrt{\eps}
|D^{\eps}_{h}|\big)u\\
&+\widetilde{\sigma}_{app}(x_h,z,D^{\eps}_{h})\big[\fd_{\eps}(\zeta),
\exp\big(\f{c_+}2z\sqrt{\eps}|D^{\eps}_{h}|\big)\big]u,
\end{split}
\end{align*}
with $\widetilde{\sigma}_{app}$ being given by \eqref{lem5.1a}. Then
we deduce from the proof of Propositions
\ref{prop:psdo-estimate}-\ref{prop:psdo-commuatator} and Lemma
\ref{lem:approximate} that \beno
\|\Lambda^{s}_\eps(\na_h^{\eps},\p_z)w^b_{app}\|_2 \le \eps^\f12
M(\sigma)\big(|\mathfrak{P}u|_{H^{s+1}_\eps}
+|\mathfrak{P}u|_{H^{t_0+2}_\eps}|(\zeta,b)|_{H^{s+4}_\eps}\big).
\eeno And we get by applying Propositions
\ref{prop:psdo-estimate}-\ref{prop:psdo-commuatator}  that
\begin{align}\label{eq:rcomm-est-3}
\|\Lambda_{\eps}^s h_2\|_2 \le
\eps^{\f32}M(\sigma)\big(|\mathfrak{P}u|_{H^{s+1}_{\eps}}
+|\mathfrak{P}u|_{H^{t_0+2}_\eps}|(\zeta,b)|_{H^{s+4}_\eps}\big).
\end{align}

\no{\bf Step 3.} The estimate of $h_3$.\,

It is easy to observe from \eqref{eq:main symbol of DN} that $
h_3=\big[\fd_{\eps}(\zeta), {\bf P}_{app}\circ\sigma_{app}\big]u$,
since
\begin{eqnarray*}
\begin{split}
{\bf
P}_{app}\circ\sigma_{app}=&p_{d+1}(\p_z-\eta_-(x_h,z,D^{\eps}_{h}))
\tau(x_h,z,D^{\eps}_{h})\exp(\f{c_+}2z\sqrt{\eps}|D^{\eps}_{h}|),\\
\tau(x_h,z,D^{\eps}_{h})=&\textrm{Op}_{\eps}(\eta_+\tilde\sigma_{app})-\eta_+(x_h,z,D^{\eps}_{h})
\circ\tilde\sigma_{app}(x_h,z,D^{\eps}_{h}),
\end{split}
\end{eqnarray*}
we write
\begin{align*}
\begin{split}
h_3=& \big[\fd_{\eps}(\zeta)
,p_{d+1}\big](\p_z-\textrm{Op}_{\eps}(\eta_-))
\tau(X,z,D^{\eps}_{h})\exp(\f{c_+}2z\sqrt{\eps}|D^{\eps}_{h}|)u\\
&-p_{d+1}\big[\fd_{\eps}(\zeta),\textrm{Op}_{\eps}(\eta_-)\big]\tau(x_h,z,D^{\eps}_{h})\exp(\f{c_+}2z\sqrt{\eps}|D^{\eps}_{h}|)u\\
&+p_{d+1}(\p_z-\textrm{Op}_{\eps}(\eta_-))
\big[\fd_{\eps}(\zeta),\tau(x_h,z,D^{\eps}_{h})\big]\exp(\f{c_+}2z\sqrt{\eps}|D^{\eps}_{h}|)u\\
&+p_{d+1}(\p_z-\textrm{Op}_{\eps}(\eta_-))
\tau(x_h,z,D^{\eps}_{h})\big[\fd_{\eps}(\zeta),\exp(\f{c_+}2z\sqrt{\eps}|D^{\eps}_{h}|)\big]u.
\end{split}
\end{align*}
Applying Propositions
\ref{prop:psdo-estimate}-\ref{prop:psdo-commuatator} ensures that
\beq\label{eq:rcomm-est-4} \|\Lambda_{\eps}^s h_3\|_2\le \eps^\f32
M(\sigma)\big(|\mathfrak{P}u|_{H^{s+1}_{\eps}}
+|\mathfrak{P}u|_{H^{t_0+1}}|(\zeta,b)|_{H^{s+4}_\eps}\big). \eeq

\no{\bf Step 4.} The estimate of $h_4$.

Notice that
\begin{align*}
\begin{split}
h_4&=\big[\fd_{\eps}(\zeta), \na^{\eps}\cdot(1+Q[\sigma])\na^{\eps}\big]u^b_r\\
&=\big[\fd_{\eps}(\zeta), \na^{\eps}\cdot
(1+Q[\sigma])\big]\na^{\eps}u^b_r+ \na^{\eps}\cdot
(1+Q[\sigma])\big[\fd_{\eps}(\zeta), \na^{\eps}\big]u^b_r,
\end{split}
\end{align*}
which together with Lemma \ref{lem:elliptic operator-commutator}
 and Lemma \ref{lem:remainder-elliptic} implies that
\begin{align}\label{eq:rcomm-est-5}
\begin{split}
\|\Lambda_{\eps}^s h_4\|_2&\le \eps^2
M(\sigma)\big(\|\Lambda_{\eps}^{s+2}\na^{\eps}u^b_{r}\|_2
+|(\zeta,b)|_{H^{s+4}_\eps}\|\Lambda^{t_0+1}\na^{\eps}u^b_{r}\|_2\big)\\
&\quad+\eps
M(\sigma)\big(\|\Lambda_{\eps}^{s+1}\na^{\eps}u^b_{r}\|_2
+\|\Lambda_{\eps}^{s+1}(\na^{\eps})^2u^b_{r}\|_2+|(\zeta,b)|_{H^{s+4}_\eps}\|\Lambda^{t_0+1}\na^{\eps}u^b_{r}\|_2\big)\\
&\le \eps^\f32 M(\sigma)\big(|\mathfrak{P}u|_{H^{s+1}_{\eps}}
+|\mathfrak{P}u|_{H^{t_0+1}}|(\zeta,b)|_{H^{s+4}_\eps}\big).
\end{split}
\end{align}

\no{\bf Step 5.} The estimate of $g_1$ and $g_2$.\,

We first get by applying  Lemma \ref{lem:trace theorem} and Lemma
\ref{lem:elliptic operator-commutator} that
\begin{align*}
\begin{split}
|g_1|_{H^{s+\f12}_\eps}&\le \eps^{-\f14}\|\Lambda^{s}_\eps\na^{\eps}
\big[(1+Q[\sigma])\na^{\eps},\fd_{\eps}(\zeta)\big]u^b_r\|_2\\
&\le \eps^{\f34}
M(\sigma)\Big(\sqrt{\eps}\|\Lambda^{s+2}_\eps\na^{\eps}u^b_r\|_2
+\|\Lambda^{s+1}_{\eps}\na^{\eps}\p_z u^b_r\|_2\\
&\quad+\big(\sqrt{\eps}\|\Lambda^{t_0}\na^{\eps}u^b_r\|_2
+\|\Lambda^{t_0}\na^{\eps}\p_z
u^b_r\|_2\big)|(\zeta,b)|_{H^{s+5}_\eps}\Big),
\end{split}
\end{align*}
from which and Lemma \ref{lem:remainder-elliptic}, we infer that
\begin{align}\label{eq:rcomm-est-6}
|g_1|_{H^{s+\f12}_\eps}\le \eps^{\f54}
M(\sigma)\big(|\mathfrak{P}u|_{H^{s+1}_{\eps}}
+|\mathfrak{P}u|_{H^{t_0+1}}|(\zeta,b)|_{H^{s+5}_\eps}\big).
\end{align}

To deal with $g_2,$ we first rewrite it as
\begin{align*}
\begin{split}
g_2=&-e_{3}\cdot\big[\fd_{\eps}(\zeta), (1+Q[\sigma])\na^{\eps}\big]\sigma_{app}(x_h,z,D^{\eps}_{h})u|_{z=-1}\\
&-e_{3}\cdot(1+Q[\sigma])\na^{\eps}\big[\fd_{\eps}(\zeta),
\sigma_{app}(x_h,z,D^{\eps}_{h})\big]u|_{z=-1}{\eqdefa}
g_{21}+g_{22}.
\end{split}
\end{align*}
It follows  from Lemma \ref{lem:elliptic operator-commutator} and
the proof of Lemma \ref{lem:approximate} that
\begin{align*}\begin{split}
|g_{21}|_{H^{s+1}_\eps}\le& \eps M(\sigma)
\big(|\na^{\eps}\sigma_{app}(x_h,z,D^{\eps}_{h})u|_{z=-1}|_{H^{s+2}_\eps}\\&+
|\na^{\eps}\sigma_{app}(x_h,z,D^{\eps}_{h})u|_{z=-1}|_{H^{t_0+1}}|(\zeta,b)|_{H^{s+4}_\eps}\big)\\
\le& \eps M(\sigma)\big(|\mathfrak{P}u|_{H^{s+1}_{\eps}}
+|\mathfrak{P}u|_{H^{t_0+1}}|(\zeta,b)|_{H^{s+4}_\eps}\big),
\end{split}
\end{align*}
and similarly, one has
\begin{align*}
|g_{22}|_{H^{s+1}_\eps} \le \eps
M(\sigma)\big(|\mathfrak{P}u|_{H^{s+1}_{\eps}}
+|\mathfrak{P}u|_{H^{t_0+1}}|(\zeta,b)|_{H^{s+4}_\eps}\big).
\end{align*}
This gives \ben\label{eq:rcomm-est-8} |g_2|_{H^{s+1}_\eps}\le \eps
M(\sigma)\big(|\mathfrak{P}u|_{H^{s+1}_{\eps}}
+|\mathfrak{P}u|_{H^{t_0+1}}|(\zeta,b)|_{H^{s+4}_\eps}\big). \een

Plugging (\ref{eq:rcomm-est-2})-(\ref{eq:rcomm-est-6}) and
(\ref{eq:rcomm-est-8}) into (\ref{eq:rcomm-est-1}) results in \beno
\|\Lambda^{s+1}_{\eps}\na^{\eps}w^b_r\|_2\le \eps
M(\sigma)\big(|\mathfrak{P}u|_{H^{s+1}_{\eps}}
+|\mathfrak{P}u|_{H^{t_0+1}}|(\zeta,b)|_{H^{s+5}_\eps}\big). \eeno
On the other hand, it follows from the third inequality of
Proposition \ref{prop:elliptic estimate-general} that \beno
\begin{split}
\|\Lam^s_{\eps}\na^{\eps}\p_zw^b_r\|_2  &\ \le
M(\sigma)\Big(\|\Lam^s_{\eps}(h_1+\cdots+h_4)\|_2+\eps^\f14|g_1|_{H^{s+\f12}_\eps}+\sqrt{\eps}|g_2|_{H^{s}_\eps}\nonumber
\\&\quad+\big(\|\Lambda^{t_0}_\eps(h_1+\cdots+h_4)\|_2+\sqrt{\eps}
|g_1+g_2|_{H^{t_0+1}_\eps}\big)|(\zeta,b)|_{H^{s+2}_\eps}\Big)\nonumber\\
&\ \le \eps^{\f32} M(\sigma)\big(|\mathfrak{P}u|_{H^{s+1}_{\eps}}
+|\mathfrak{P}u|_{H^{t_0+2}}|(\zeta,b)|_{H^{s+5}_\eps}\big).
\end{split}\eeno
This completes the proof of Lemma
\ref{lem:remainder-elliptic-comm}.\end{proof}

\begin{lem}\label{lem:remainder-elliptic-comm-higher} {\sl Let $s\ge 0,
k\in \N$, and $u\in H^{2k+s-\f12}\cap H^{t_0+\f52}(\R^2)$. Then
there hold \beno
\begin{split}&\|\Lambda^{s+1}_{\eps}\na^{\eps}w^b_{r,k}\|_2\le \eps
M(\sigma)\big(|\mathfrak{P}u|_{H^{2k+s-1}_\eps}
+|\mathfrak{P}u|_{H^{t_0+2}}|(\zeta,b)|_{H^{2k+s+3}_\eps}\big),\\
&\|\Lam^s_{\eps}\na^{\eps}\p_zw^b_{r,k}\|_2 \le \eps^{\f32}
M(\sigma)\big(|\mathfrak{P}u|_{H^{2k+s-1}_\eps}
+|\mathfrak{P}u|_{H^{t_0+2}}|(\zeta,b)|_{H^{2k+s+3}_\eps}\big).
\end{split}
\eeno Here $w_{r,k}^b{\eqdefa}(|D^{\eps}_{h}|^{2k}
u)^b_{r}-|D^{\eps}_{h}|^{2k}u^b_{r}.$ }\end{lem}

\begin{proof}\,The proof of this lemma essentially follows from that of Lemma
\ref{lem:remainder-elliptic-comm}. Here we need to use Proposition
\ref{prop:psdo-commuatator-special} to deal with some commutator
estimates. We omit the details here.\end{proof}

\begin{prop}\label{prop:remainder-commutator} {\sl Let $s\ge 0$ and
$u\in H^{s+\f32}\cap H^{t_0+\f52}(\R^2)$. Then under the assumptions
of theorem \ref{thm:DN-commutator-sharp}, we have
\begin{align*}
\begin{split}
\big|\big[\frac 1{\eps} R[\eps\zeta],\fd_{\eps}(\zeta)\big]u\big|
_{H^{s}_\eps} &\ \le \sqrt{\eps}
M(\sigma)\Big(|\mathfrak{P}u|_{H^{s+1}_{\eps}}+|\mathfrak{P}u|_{H^{t_0}}|(\zeta,b)|_{H^{s+5}_\eps}\\
&\quad
+\langle|\mathfrak{P}u|_{H^{t_0+1}}|(\zeta,b)|_{H^{s+5}_\eps}\rangle_{s>t_0}
+\langle|\mathfrak{P}u|_{H^{t_0+2}}|(\zeta,b)|_{H^{s+5}_\eps}\rangle_{s>t_0+1}\Big).
\end{split}
\end{align*}}
\end{prop}

\begin{proof}\,\,Thanks to (\ref{eq:remainder}), for any $v\in
\cS(\R^2)$, we get by applying Green's identity  that
\begin{eqnarray*}
\begin{split}
\big(\Lambda_{\eps}^s[R[\eps\zeta],&\fd_{\eps}(\zeta)]u,v\big)=
\big(R[\eps\zeta]\fd_{\eps}(\zeta) u,\Lambda_{\eps}^s v\big)
-\big(\fd_{\eps}(\zeta) R[\eps\zeta]u,\Lam^s_{\eps} v\big)\\
=&\int_{\mathcal
S}\Bigl\{\na^{\eps}\cdot(1+Q[\sigma])\na^{\eps}(\fd_{\eps}(\zeta)
u)^b_{r}\Lambda _{\eps}^s
v^\dag+(1+Q[\sigma])\na^{\eps}(\fd_{\eps}(\zeta) u)^b_r\cdot
\na^{\eps}\Lam^s_{\eps} v^\dag\\
&-\na^{\eps}\cdot\big(\fd_{\eps}(\zeta)(1+Q[\sigma])\na^{\eps}u^b_{r}\big)
\Lambda_{\eps}^sv^\dag-\fd_{\eps}(\zeta)(1+Q[\sigma])\na^{\eps}u^b_r\cdot
\na^{\eps}\Lambda_{\eps}^s v^\dag\Bigr\}\,dx_h\,dz,
\end{split}
\end{eqnarray*}
which ensures
\begin{eqnarray*}
\begin{split}
\big(\Lambda_{\eps}^s[R[\eps\zeta],\fd_{\eps}(\zeta)]u,v\big) &\ =
\int_{\mathcal S}\Bigl\{\na^{\eps}\cdot(1+Q[\sigma])\na^{\eps}
\big((\fd_{\eps}(\zeta) u)^b_{r}\\
&\quad-\fd_{\eps}(\zeta) u^b_{r}\big)\Lam^s_{\eps}
v^\dag-\na^{\eps}\cdot\big[\fd_{\eps}(\zeta),(1+Q[\sigma])\na^{\eps}\big]u^b_{r}
\Lambda_{\eps}^sv^\dag\\
&\quad+\big[(1+Q[\sigma])\na^{\eps},
\fd_{\eps}(\zeta)\big]u^b_r\cdot \na^{\eps}\Lam^s_{\eps}
v^\dag\\
&\quad+(1+Q[\sigma])\na^{\eps} \big((\fd_{\eps}(\zeta)
u)^b_r-\fd_{\eps}(\zeta) u^b_r\big)\cdot
\na^{\eps}\Lam^s_{\eps} v^\dag\Bigr\}\,dx_h\,dz\\
&\ {\eqdefa} B_1+B_2+B_3+B_4.
\end{split}
\end{eqnarray*}
Here $v^\dag{\eqdefa} (1+z)\chi(z\sqrt{\eps}|D^{\eps}_{h}|)v$ with
$\chi$ being given by Lemma \ref{lem:Poisson regularization}.
\vspace{0.2cm}

Again we only consider the case of $s>t_0+1$, the other cases can be
handled in a  similar way. First of all, we get by applying Lemma
\ref{lem:duality} that
\begin{align*}
\begin{split}
|B_3|\le&
C\sqrt{\eps}|v|_2\|\Lambda^{s+1}_{\eps}\big[(1+Q[\sigma])\na^{\eps},
\fd_{\eps}(\zeta)\big]u^b_r\|_2+C|v|_2\|\Lam^s_{\eps}\big[(1+Q[\sigma])\na^{\eps},
\fd_{\eps}(\zeta)\big]u^b_r\|_2.
\end{split}
\end{align*}
Besides, Lemma \ref{lem:elliptic operator-commutator} and the first
inequality of Lemma \ref{lem:remainder-elliptic} ensures that \beno
\begin{split}
\|\Lambda^{s+1}_{\eps}\big[(1+Q[\sigma])\na^{\eps},
\fd_{\eps}(\zeta)\big]u^b_r\|_2
 &\ \le \eps
M(\sigma)\big(\|\Lambda_{\eps}^{s+2}\na^{\eps}u_r^b\|_2
+|(\zeta,b)|_{H^{s+4}_\eps}\|\Lambda^{t_0+1}\na^{\eps}u_r^b\|_2\big)\\
&\ \le \eps M(\sigma)\big(|\mathfrak{P}u|_{H^{s+1}_{\eps}}
+|\mathfrak{P}u|_{H^{t_0}}|(\zeta,b)|_{H^{s+4}_\eps}\big),\end{split}
\eeno and Lemma \ref{lem:elliptic operator-commutator} and the
second inequality of Lemma \ref{lem:remainder-elliptic} implies that
\beno\begin{split} \|\Lam^s_{\eps}\big[(1+Q[\sigma])\na^{\eps},
\fd_{\eps}(\zeta)\big]u^b_r\|_2
 &\ \le \eps
M(\sigma)\big(\|\Lambda_{\eps}^{s+1}\na^{\eps}u_r^b\|_2
+|(\zeta,b)|_{H^{s+3}_\eps}\|\Lambda^{t_0+1}\na^{\eps}u_r^b\|_2\big)\\
&\ \le \eps^{\f32} M(\sigma)\big(|\mathfrak{P}u|_{H^{s+1}_{\eps}}
+|\mathfrak{P}u|_{H^{t_0+1}}|(\zeta,b)|_{H^{s+3}_\eps}\big).
\end{split}\eeno Here we used the fact that
$\|\Lambda^{s+1}_{\eps}\na^{\eps}u^b_r\|_2\le
\|\Lambda^{s+1}_{\eps}\na^{\eps}\p_zu^b_r\|_2$ due to
$u^b_r|_{z=0}=0$. So we arrive at \beq\label{eq:rcomm-est-C}
|B_3|\le \eps^{\f32}
M(\sigma)|v|_2\big(|\mathfrak{P}u|_{H^{s+1}_{\eps}}
+|\mathfrak{P}u|_{H^{t_0+1}}|(\zeta,b)|_{H^{s+4}_\eps}\big). \eeq

Notice that  $B_2$ can be rewritten as
\begin{align*}
\begin{split}
B_2=&-\int_{\mathcal
S}\Bigl\{\sqrt{\eps}(\na^{\eps}_h,0)\cdot\big[\fd_{\eps}(\zeta),(1+Q[\sigma])\na^{\eps}\big]u^b_{r}
\Lambda_{\eps}^sv^\dag+{\bf
e_3}\cdot\big[\fd_{\eps}(\zeta),\p_zQ[\sigma]\na^{\eps}\big]u^b_{r}
\Lambda_{\eps}^sv^\dag\\
&\quad+{\bf
e_3}\cdot\big[\fd_{\eps}(\zeta),(1+Q[\sigma])\na^{\eps}\big]\p_zu^b_{r}
\Lambda_{\eps}^sv^\dag\Bigr\}\,dx_h\,dz,
\end{split}
\end{align*}
which together with  Lemma \ref{lem:elliptic operator-commutator}
and Lemma \ref{lem:remainder-elliptic} yields that
\beq\label{eq:rcomm-est-B} |B_2|\le \eps^{\f32}
M(\sigma)|v|_2\big(|\mathfrak{P}u|_{H^{s+1}_{\eps}}
+|\mathfrak{P}u|_{H^{t_0+1}}|\zeta|_{H^{s+4}_\eps}\big). \eeq

Finally we have by Lemma \ref{lem:duality} and Lemma
\ref{lem:remainder-elliptic-comm} that
\begin{align}
\begin{split}
|B_1|+|B_4| \le&
M(\sigma)|v|_2\Big(\sqrt{\eps}\|\Lambda^{s+1}_{\eps}\na^{\eps}w^b_r\|_2+\|\Lam^s_{\eps}\na^{\eps}\p_zw^b_r\|_2
\\
&+
\|\Lambda^{t_0}\na^{\eps}(w^b_r+\p_zw^b_r)\|_2|(\zeta,b)|_{H^{s+2}_\eps}\Big)\\
\le& \eps^{\f32}M(\sigma)|v|_2\big(|\mathfrak{P}u|_{H^{s+1}_{\eps}}
+|\mathfrak{P}u|_{H^{t_0+2}}|(\zeta,b)|_{H^{s+5}_\eps}\big).
\end{split}\label{eq:rcomm-est-A}
\end{align}

Summing up (\ref{eq:rcomm-est-C})-(\ref{eq:rcomm-est-A}), we
conclude Proposition \ref{prop:remainder-commutator} by the duality.
\end{proof}

\begin{rmk}\label{rem:commutator-sharp} It follows from the  proof
of Proposition \ref{prop:remainder-commutator} that
\begin{align*}\begin{split}
\big|\big[\frac 1{\eps}
R[\eps\zeta],\fd_{\eps}(\zeta)-|D^{\eps}_{h}|^2\big]&u\big|
_{H^{s}_\eps} \ \le\sqrt{\eps}
M(\sigma)\Big(|\mathfrak{P}u|_{H^{s+1}_{\eps}}+|\mathfrak{P}u|_{H^{t_0}_\eps}|(\zeta,b)|_{H^{s+5}_\eps}\\
&\quad
+\langle|\mathfrak{P}u|_{H^{t_0+1}_\eps}|(\zeta,b)|_{H^{s+5}_\eps}\rangle_{s>t_0}
+\langle|\mathfrak{P}u|_{H^{t_0+2}_\eps}|(\zeta,b)|_{H^{s+5}_\eps}\rangle_{s>t_0+1}\Big).
\end{split}
\end{align*}
\end{rmk}

\vspace{0.1cm}

Thanks to Lemma \ref{lem:remainder-elliptic-comm-higher} and the
proof of Proposition \ref{prop:remainder-commutator}, we  also
obtain that

\begin{prop}\label{prop:remainder-commutator-higher} {\sl Let $s\ge
0, k\in \N$, and $u\in H^{2k+s-\f12}\cap H^{t_0+\f52}(\R^2)$. Then
under the assumptions of theorem \ref{thm:DN-commutator-sharp}, one
has \beno \big|\big[\frac 1{\eps}
R[\eps\zeta],|D^{\eps}_{h}|^{2k}\big]u\big| _{H^{s}_\eps}\le
\sqrt{\eps} M(\sigma)\Big(|\mathfrak{P}u|_{H^{2k+s-1}_\eps}
+|\mathfrak{P}u|_{H^{t_0+2}}|\zeta|_{H^{2k+s+3}_\eps}\Big). \eeno}
\end{prop}

Now we are in a position to prove Theorem
\ref{thm:DN-commutator-sharp}.\vspace{0.1cm}

\no{\bf Proof of Theorem \ref{thm:DN-commutator-sharp}.}\, Thanks to
\eqref{eq:remainder}, we write
\begin{align}\label{th5.1}
\begin{split}
\big[\fd_{\eps}(\zeta)^k, \f 1{\eps} \rho(\zeta)^{-1}G[{\eps}
\zeta]\big]u &=\big[\fd_{\eps}(\zeta)^k, \f 1{\eps}
\rho(\zeta)^{-1}g(x_h,D^{\eps}_{h})\big]u\\
&+ \big[\fd_{\eps}(\zeta)^k, \f 1{\eps} \rho(\zeta)^{-1}R[{\eps}
\zeta]\big]u{\eqdefa} I_k+II_k.
\end{split}
\end{align}
We shall prove  by an induction argument on $k$ that
\beq\label{eq:comm-est1} |I_k|_{H^s_\eps}\le \sqrt{\eps}
M(\sigma)\Big(|\mathfrak P u|_{H^{2k+s-1}_\eps}
+|\mathfrak{P}u|_{H^{t_0+2}_\eps}|\zeta|_{H^{2k+s+3}_\eps}\Big).
\eeq The case of $k=1$ is a direct consequence of Proposition
\ref{prop:main part-commutator}. We assume that (\ref{eq:comm-est1})
holds for $k\le \ell-1$. We now turn to prove the case of $k=\ell$.
We write \beno \begin{split} I_k
=&\fd_{\eps}(\zeta)\big[\fd_{\eps}(\zeta)^{\ell-1},\f 1{\eps}
\rho(\zeta)^{-1}g(x_h,D^{\eps}_{h})\big]u+ \big[\fd_{\eps}(\zeta),\f
1{\eps}
\rho(\zeta)^{-1}g(x_h,D^{\eps}_{h})\big]\fd_{\eps}(\zeta)^{\ell-1}
u.\end{split} \eeno Using Lemma \ref{lem:elliptic operator} and the
induction assumption, we have by an interpolation argument that
\beno
\begin{split}
\big|\fd_{\eps}(\zeta)\big[\fd_{\eps}(\zeta)^{\ell-1},\f 1{\eps}
\rho(\zeta)^{-1} g(x_h,D^{\eps}_{h})\big]u\big|_{H^s_\eps}&\  \le
M(\sigma)\Big(\big|\big[\fd_{\eps}(\zeta)^{\ell-1},\f 1{\eps}
\rho(\zeta)^{-1}g(X,D^{\eps}_{h})\big]u\big|_{H^{s+2}_\eps}\\
&\qquad+\big|\big[\fd_{\eps}(\zeta)^{\ell-1},\f 1{\eps}
\rho(\zeta)^{-1}g(x_h,D^{\eps}_{h})
\big]u\big|_{H^{t_0}_{\eps}}|\zeta|_{H^{s+3}_\eps}\Big)\\
&\ \le \sqrt{\eps} M(\sigma)\Big(|\mathfrak P
u|_{H^{2\ell+s-1}_\eps}
+|\mathfrak{P}u|_{H^{t_0+2}_\eps}|\zeta|_{H^{2\ell+s+3}_\eps}\Big),
\end{split}
\eeno and applying  Proposition \ref{prop:main part-commutator} and
Lemma \ref{lem:elliptic operator} gives \beno
\begin{split}
\big|\big[\fd_{\eps}(\zeta),&\f 1{\eps}
\rho(\zeta)^{-1}g(x_h,D^{\eps}_{h})\big]\fd_{\eps}(\zeta)^{\ell-1}u\big|_{H^s_\eps}\\
&\le \eps^\f94M(\sigma)\Big(\big|\mathfrak P
\fd_{\eps}(\zeta)^{\ell-1}u\big|_{H^{s+\f12}_\eps}+\langle\big|\mathfrak P \fd_{\eps}(\zeta)^{\ell-1}u\big|_{H^{t_0+\f32}_\eps}|\zeta|_{H^{s+5}_\eps}\rangle_{s>t_0+1}\Big)\\
&\le \sqrt{\eps} M(\sigma)\Big(|\mathfrak P u|_{H^{2\ell+s-1}_\eps}
+|\mathfrak{P}u|_{H^{t_0+2}_{\eps}}|(\zeta,b)|_{H^{2\ell+s+3}_\eps}\Big).
\end{split}
\eeno
This proves (\ref{eq:comm-est1}) for $k=\ell$.

Noticing that $II_k$ can be rewritten as \beno
II_k=\big[\fd_{\eps}(\zeta)^k, \f 1{\eps}
\rho(\zeta)^{-1}\big]R[{\eps} \zeta]u+
\rho(\zeta)^{-1}\big[\fd_{\eps}(\zeta)^k, \f 1{\eps} R[{\eps}
\zeta]\big]u, \eeno applying Lemma \ref{lem:elliptic
operator-commutator} and Proposition \ref{prop:remainder} gives
\beno \begin{split} &\big|\big[\fd_{\eps}(\zeta)^k, \f 1{\eps}
\rho(\zeta)^{-1}\big]R[{\eps} \zeta]u\big|_{H^s_\eps}\le \eps^\f32
M(\sigma)\Big(|\mathfrak P u|_{H^{2k+s-1}_\eps}
+|\mathfrak{P}u|_{H^{t_0+2}_\eps}|(\zeta,b)|_{H^{2k+s+3}_\eps}\Big).
\end{split}
\eeno Similarly, we can prove by an induction argument that \beno
\begin{split}
&\big|\f 1{\eps}
\rho(\zeta)^{-1}\big[\fd_{\eps}(\zeta)^k-|D^{\eps}_{h}|^{2k},
R[{\eps} \zeta]\big]u\big|_{H^s_\eps} \\
&\quad\le \sqrt{\eps} M(\sigma)\Big(|\mathfrak P
u|_{H^{2\ell+s-1}_\eps}
+|\mathfrak{P}u|_{H^{t_0+2}_\eps}|(\zeta,b)|_{H^{2\ell+s+3}_\eps}\Big),
\end{split}
\eeno which together with Proposition
\ref{prop:remainder-commutator-higher} implies that $II_k$ also
satisfies \eqref{eq:comm-est1}. In fact, the case of $k=1$ comes
from Remark \ref{rem:commutator-sharp}.\ef

\renewcommand{\theequation}{\thesection.\arabic{equation}}
\setcounter{equation}{0}

\section{Large time existence for the linearized system}

\subsection{The linearized system}\label{sec6.1}
We first reformulate the original system (\ref{eq:Hamiltonian
form-non}) as \beq\label{eq:Hamiltonian form-non-new}
\p_tU+\mathcal{L}U+{\eps}\mathcal{A}[U]=0 \eeq where
\[U=(\zeta, \psi)^T,\quad\mathcal{A}[U]=(A_1[U], A_2[U])^T,\quad
\mathcal {L}=\left(\begin{matrix} 0 & -\frac1 {\eps} G[0]\\ 1&
0\end{matrix}\right),\] and \beno
\begin{split}
A_1[U]\eqdefa&-\frac 1{{\eps}^2}(G[{\eps}\zeta]\psi-G[0]\psi),\\
A_2[U]\eqdefa&\frac12 |\na^{\eps}_h\psi|^2 - \frac{\big(\eps^{-\f12}
G[{\eps}\zeta]\psi+{\eps}^{\frac32}\na^{\eps}_h\zeta\cdot
\na^{\eps}_h \psi\big)^2}{2(1+{\eps}^3|\na^
\eps_h\zeta|^2)}-\al\na^{\eps}_h \cdot \Big(\frac{\na^{\eps}_h
\zeta}{\sqrt{1+{\eps}^3|\na^{\eps}_h \zeta|^2}}\Big).
\end{split}\eeno

Motivated by \cite{lan1}, we shall use  Nash-Moser iteration Theorem
to prove the large time existence of solutions to
\eqref{eq:Hamiltonian form-non-new}. Toward this, a key step will be
to study the linearized system of \eqref{eq:Hamiltonian
form-non-new}. Indeed, we shall linearize the system
(\ref{eq:Hamiltonian form-non-new}) around an admissible reference
state in the following sense:

\begin{defi} {\sl Let $T > 0$. We say that
$\underline{U}=(\underline{\zeta},\underline{\psi})^T$ is an
admissible reference state to \eqref{eq:Hamiltonian form-non-new} on
$[0,\f T{\eps}]$ if there exists $h_0>0$ such that \beno
1+\eps\underline{\zeta}-\eps b\ge h_0\quad \mbox{uniformly on}\quad
[0,\f T{\eps}]\times \R^2. \eeno }\end{defi}

Given an admissible reference state
$\underline{U}=(\underline{\zeta},\underline{\psi})^T,$ one can
calculate by Proposition \ref{prop:shape derivative} (see also
\cite{Lannes-Inven}, \cite{MZ}) that the linearized operator of the
system (\ref{eq:Hamiltonian form-non-new}) equals  \beno
\begin{split}
&\underline{\mathcal{L}}=\p_t+\mathcal{L}+{\eps} d_{\underline{U}}\mathcal{A}\\
&=\p_t+\left(\begin{matrix} G[{\eps}\underline{\zeta}] (\cdot
\underline{Z})+{\eps}\na^{\eps}_h\cdot(\cdot {\bf\underline{v}})
&-\frac1 {\eps} G[{\eps}\underline{\zeta}]\cdot\\
{\eps}\underline{Z}G[{\eps}\underline{\zeta}](\cdot
\underline{Z})+(1+{\eps}^2\underline{Z}\na^{\eps}_h\cdot{\bf\underline{v}})\cdot
-\al\eps\underline{A} &
{\eps}{\bf\underline{v}}\cdot\na^{\eps}_h-\underline{Z}G[{\eps}\zeta]\end{matrix}\right),
\end{split}
\eeno where \beno
\begin{split}
&\underline{A}\eqdefa\na^{\eps}_h\cdot\left[\frac{\na^{\eps}_h}{(1+{\eps}^3|\na^{\eps}_h\underline{\zeta}|^2)^\frac12}
-\frac{{\eps}^3\na^{\eps}_h\underline{\zeta}(\na^{\eps}_h\underline{\zeta}\cdot\na^{\eps}_h)}
{(1+{\eps}^3|\na^{\eps}_h\underline{\zeta}|^2)^\frac32}\right],\\
&\underline {\bf v}\eqdefa\na^{\eps}_h\underline \psi-{\eps}
\underline Z \na^{\eps}_h\underline \zeta,\qquad\underline
Z\eqdefa\frac1 {1+{\eps}^3
|\na^{\eps}_h\underline\zeta|^2}\big(G[{\eps}\underline\zeta]\underline\psi
+{\eps}^2\na^{\eps}_h\underline\zeta\cdot\na^{\eps}_h\underline\psi\big).
\end{split}
\eeno This gives the linearized system of (\ref{eq:Hamiltonian
form-non-new}) as follows \beq\label{eq:linearized system}
\underline{\mathcal{L}}U={\eps} G,\quad U|_{t=0}=U^0. \eeq To solve
\eqref{eq:linearized system}, as in \cite{lan1, MZ} we introduce a
new variable $V{\eqdefa}(\zeta,\psi-{\eps}\underline{Z}\zeta)^T$
such that the principal symbol of the transformed linearized
operator
 is trigonalized.  Indeed  the system
(\ref{eq:linearized system}) can be equivalently written as
\beq\label{eq:linearized system-new} \underline{\mathcal{M}}V={\eps}
H,\quad V|_{t=0}=V^0, \eeq where
\[\underline{\mathcal{M}}{\eqdefa}\p_t+\left(\begin{matrix} {\eps}\na^{\eps}_h\cdot
(\cdot{\bf\underline{v}}) & -\frac1 {\eps} G[{\eps}\underline{\zeta}]\\
\underline{a}-\al\eps\underline{A}&
{\eps}{\bf\underline{v}}\cdot\na^{\eps}_h\end{matrix}\right),\] and
\[\underline{a}=1+{\eps}({\eps}{\bf\underline{v}}\cdot\na^{\eps}_h\underline{Z}
+\p_t\underline{Z}),\quad H=(G_1,G_2-{\eps}\underline{Z}G_1).\]

\subsection{Large time existence}

In this subsection, we shall prove the large time existence of
smooth enough solution for  the linearized system
\eqref{eq:linearized system} and establish the uniform  estimates
for thus obtained solutions on $[0,\f{T}\eps].$ Toward this, we
first introduce the following definition.

\begin{defi}\label{definitions} {\sl Let $s\in \R$ and $T>0$.
\begin{itemize}
 \item[(1)]
 We define the space $X^s$ as
\[
X^s{\eqdefa} \left\{U=(\zeta,\psi)^T: \zeta\in H^{2s+1}(\R^2),
\psi\in
 H^{2s+\f12}(\R^2)\right\}
\]
endowed with the norm \beno
\begin{split}
|U|_{X^s}{\eqdefa}&\sqrt{\eps}|\zeta|_{H^{2s+1}_\eps}+|\zeta|_{H^{2s}_\eps}+\sqrt{\eps}|\na_h^{\eps}
\zeta| _{H^{s}}+|\zeta|_{H^{s}}+|\mathfrak
P\psi|_{H^{2s}_\eps}+|\mathfrak P\psi|_{H^{s}}
+\eps|\psi|_{H^{s-1}}.
\end{split}
\eeno And $X^s_T{\eqdefa}C([0,\f T\eps]; X^s)$ endowed with its
canonical norm.

\item[(2)]  Let $\frak{X}^{s}$ be determined by Definition \ref{def1.1}.
The semi-normed space $(Y^s_T,|\cdot|_{Y^s_T})$ is defined as
\[
Y^s_T{\eqdefa}\bigcap^1_{j=0}C^j([0,\f T{\eps}];
\Lam^{2j}_\eps\frak{X}^s), \qquad
|U|_{Y^s_T}{\eqdefa}\sum^1_{j=0}\sup_{[0,\f T{\eps}]}|\p^j_t
\Lambda^{-2j}_\eps U| _{\frak{X}^{s}}.
\]

\item[(3)] For any $(G,U^0)\in X^s_T\times X^s$, we denote
\[
\mathcal I^s(t,U^0,G){\eqdefa}|U^0|^2_{X^s}+{\eps}
\int^t_0\sup_{[0,t']} |G(t'')|^2_{X^s}dt'.
\] \end{itemize}
}
\end{defi}

\begin{prop}\label{prop:linearized system} {\sl Let $k\in \N$,
$2k\ge m_0$, $T>0$ and $b\in H^{2k+5}(\R^2)$. Assume that
$\underline U= (\underline\zeta,\underline\psi)^T\in Y^{k}_T$ is an
admissible reference state on $[0,\f T\eps]$ for some $h_0>0$. Then
for any $(G,U^0)\in X^{k}_T\times X^{k}$, (\ref{eq:linearized
system}) has a unique solution $U\in X^{k}_T$ such that for all
$t\in [0,\f T{\eps}]$,
\[
|U(t)|^2_{X^{k}}\leq \underline C\Big(\mathcal I^{k}(t,U^0,G)+
|\Lambda^{5}_\eps\underline U|^2_{Y^{k}_T}\mathcal
I^{5}(t,U^0,G)\Big),
\]
where $ \underline
C=C(T,\f1{h_0},\al,\al^{-1},{\eps},|b|_{H^{2k+5}}, |\underline
U|_{Y^{m_0}_T})$.} \end{prop}

The proof of this proposition relies on the study of the
trigonalized linearized system (\ref{eq:linearized system-new}),
which we admit for the time being.

\begin{prop}\label{prop:linearized system-tri} {Under the
assumptions in Proposition \ref{prop:linearized system},  for any
given $(H,V^0)\in X^{k}_T\times X^{k}$, (\ref{eq:linearized
system-new}) has a unique solution $V\in X^{k}_T$ such that for all
$t\in [0,\f T{\eps}]$, \beq\label{prop6.2} |V(t)|^2_{X^{k}}\leq
\underline C\Big(\mathcal I^{k}(t,V^0,H) +|\Lambda^5_\eps\underline
U|^2_{Y^{k}_T}\mathcal I^{5}(t,V^0,H)\Big), \eeq with  $\underline
C$ being the same as in Proposition \ref{prop:linearized system}.}
\end{prop}

\no{\bf Proof of Proposition \ref{prop:linearized system}.}\
Recalling that $U=(V_1, V_2+\eps \underline{Z}V_1)$ and
$H=(G_1,G_2-\eps \underline{Z}G_1)$, it follows from Proposition
\ref{prop:shape derivative estimate} that
\begin{align*}
\begin{split}
|U(t)|^2_{X^k}\le& \underline{C}\Big(|V(t)|^2_{X^{k}}
+|\Lambda^5_\eps\underline U|^2_{Y^{k}_T}|V(t)|^2_{X^{t_0+1}}\Big)\\
\le& \underline C\Big(\mathcal I^{k}(t,V^0,H)
+|\Lambda^5_\eps\underline U|^2_{Y^{k}_T}\mathcal I^{5}(t,V^0,H)\Big)\\
\le& \underline C\Big(\mathcal I^{k}(t,U^0,G)+
|\Lambda^{5}_\eps\underline U|^2_{Y^{k}_T}\mathcal
I^{5}(t,U^0,G)\Big).
\end{split}
\end{align*}
\ef

In what follows, we shall use the energy method to prove Proposition
\ref{prop:linearized system-tri}. Notice that
$S=\left(\begin{array}{cc}
-\al\eps\underline{A} & 0\\
0& \f1\eps G[\eps\underline{\zeta}]
\end{array}\right)$ is  a symmetrizer
of $\underline{\cM},$  so that a natural energy functional for the
system (\ref{eq:linearized system-new}) is given by \beno
E^s(V)^2{\eqdefa}(\Lambda^s V, S\Lambda^s V)=(\Lambda^s V_1,
-\al\eps\underline{A}\Lambda^s V_1) +(\Lambda^s V_2, \f1\eps
G[\eps\underline{\zeta}]V_2). \eeno We shall see below that the
estimate of $E^s(V)^2$ will lead to deal with the following
commutators \beno \big(\big[\Lam^s, \underline
A\big]V_1,G[\eps\underline \zeta]\Lam^sV_2\big)\quad\hbox{and}\quad
\big(\big[\Lam^s, G[\eps\underline \zeta\big]V_2, \underline
A\Lam^sV_1\big). \eeno Since $\big[\Lam^s,\underline{A}]$ is a
pseudo-differential  operator of order $s+1$ and $ \big[\Lam^s,$
$G[\eps\zeta]\big]$ is a pseudo-differential operator of order $s$,
the above two commutators are dominated by (for example)
$|V_1|_{H^{s+\f32}}|V_2|_{H^{s+\f12}}$, which can not dominated by
the energy $E^s(V)$ due to
$$
E^s(V)\sim |V_1|_{H^{s+1}}+|V_2|_{H^{s+\f12}}\quad
(\textrm{essentially}).
$$
That is, we'll lose one half derivative in the process of energy
estimate if we choose to use this kind of symmetrizer. To overcome
this difficulty, motivated by \cite{MZ} we introduce a new energy
functional $\mathcal{E}^k(V)$ defined by \beq \label{energy}
\mathcal{E}^{k}(V)^2{\eqdefa}\mathcal{E}^{k}_{l}(V)^2+\mathcal{E}^{k}_{h}(V)^2
\eeq with $\mathcal{E}^{k}_{l}(V)$ and $\mathcal{E}^{k}_{h}(V)$
given by \beno
\begin{split}
\mathcal{E}^{k}_{l}(V)^2\eqdefa&\big(\Lambda^kV_1,(1-\alpha\eps\underline{A})\Lambda^kV_1\big)
+\big(\Lambda^kV_2,\frac1{\eps}
G[{\eps}\underline{\zeta}]\Lambda^kV_2\big)+\eps^2(\Lam^{k-1}V_2,
\Lam^{k-1}V_2),\\
\mathcal{E}^{k}_{h}(V)^2\eqdefa&\big(\fd_{\eps}(\underline\zeta)^k\rho^{-1}
V_1, \rho(1-\alpha\eps\underline{A})
\rho\fd_{\eps}(\underline\zeta)^k\rho^{-1}
V_1\big)+\big(\fd_{\eps}(\underline\zeta)^kV_2,\frac1{\eps}
G[{\eps}\underline{\zeta}] \fd_{\eps}(\underline\zeta)^kV_2\big)
\end{split}\eeno for $\fd_{\eps}(\underline\zeta)$
given by  \eqref{eq:ellipticoperator} and
$\rho=\rho(\underline\zeta)=\bigl(1+\eps^3|\na^\eps_h\underline\zeta|^2\bigr)^{\f12}.$
Let us also introduce \beq\label{energy2}
{E}^{k}(V)^2{\eqdefa}{E}^{k}_{l}(V)^2+{E}^{k}_{h}(V)^2 \eeq for
\beno \begin{split}&E_l^k(V)^2{\eqdefa}\eps|\na_h^{\eps}
V_1|_{H^{k}}^2+|V_1|_{H^{k}}^2
+|\mathfrak P V_2|_{H^k}^2+\eps^2|V_2|_{H^{k-1}}^2,\\
&E_h^k(V)^2{\eqdefa}\eps|V_1|_{H^{2k+1}_\eps}^2+|V_1|_{H^{2k}_\eps}^2+|\mathfrak
P V_2|_{H^{2k}_\eps}^2. \end{split}\eeno

We have the following relation between $\mathcal{E}^{k}(V)$ and ${E}^{k}(V)$:

\begin{lem}\label{lem:energy relation}{\sl Let $\mathcal{E}^{k}(V)$ and ${E}^{k}(V)$ be given
by \eqref{energy} and \eqref{energy2} respectively.  We have \beno
\begin{split}
&M(\sigma)^{-1}E_h^k(V)^2\le \mathcal E_h^k(V)^2+\eps
M(\sigma)\big(|V_1|^2_{H^{t_0}}+|\mathfrak P V_2|^2_{H^{t_0}}\big)
|\underline \zeta|^2_{H^{2k+1}_\eps},\\
&\mathcal E_h^k(V)^2\le M(\si)E_h^k(V)^2+\eps
M(\sigma)\big(|V_1|^2_{H^{t_0}}+|\mathfrak P V_2|^2_{H^{t_0}}\big)
|\underline \zeta|^2_{H^{2k+1}_\eps},\\
 &M(\sigma)^{-1}E_l^k(V)^2\le \mathcal E_l^k(V)^2\le
M(\sigma)E_l^k(V)^2. \end{split}\eeno }\end{lem}

\begin{proof}\,The first inequality can be deduced from Proposition
\ref{prop:DN-basic properties}, Lemma \ref{lem:elliptic operator}
and Lemma \ref{lem:elliptic operator-commutator}, while the
remaining ones are obvious.\end{proof}\vspace{0.2cm}

Now let us turn to the proof of Proposition \ref{prop:linearized
system-tri}.\\

\no{\bf Proof of Proposition \ref{prop:linearized system-tri}.\,}
With the {\it a priori} estimate \eqref{prop6.2}, it is   classical
to prove the existence part of the proposition. So we only present
the detailed proof of \eqref{prop6.2} for smooth enough solutions of
\eqref{eq:linearized system-new}.

\no{\bf Step 1.\,} High order energy estimate.

Recall that $\fd_{\eps}(\underline\zeta)$ is given by
\eqref{eq:ellipticoperator} and
$\rho=\rho(\underline\zeta)=\bigl(1+\eps^3|\na^\eps_h\underline\zeta|^2\bigr)^{\f12}.$
Let $\lambda>0$ to be determined later, and we denote $\underline G
{\eqdefa}G[{\eps} \underline\zeta]$. Then applying a standard energy
estimate to \eqref{eq:linearized system-new} yields
\beq\label{eq:energy-high}
\begin{split}
e^{{\eps}\lambda t}\f d{dt}\Big(e^{-{\eps}\lambda t}{\mathcal
E}^{k}_h(V(t))^2\Big) =&-{\eps}\lambda{\mathcal
E}^{k}_h(V(t))^2+2\big(\fd_{\eps}(\underline\zeta)^k\p_t(\rho^{-1}V_1\big),
\rho\big(1\\
&-\alpha\eps\underline{A})\rho\fd_{\eps}(\underline\zeta)^k\rho^{-1}V_1\big)+2\big(\fd_{\eps}(\underline\zeta)^k\p_tV_2,\f1{\eps}
\underline G\fd_{\eps}(\underline\zeta)^kV_2\big)\\
&+\Big[\big(\fd_{\eps}(\underline\zeta)^k\rho^{-1}V_1,\big[\p_t,\rho(1-\alpha\eps\underline
A)\rho\big]\fd_{\eps}(\underline\zeta)^k
\rho^{-1}V_1\big)\\
&+\big(\fd_{\eps}(\underline\zeta)^kV_2,\big[\p_t,\f1{\eps}\underline
G\big]\fd_{\eps}(\underline\zeta)^kV_2\big)\\
&+2\big(\big[\p_t,\fd_{\eps}(\underline\zeta)^k\big]\rho^{-1}V_1,\rho(1-\alpha\eps\underline
A)\rho\fd_\eps(\underline\zeta)^k\rho^{-1}V_1\big)\\
&+2\big(\big[\p_t,\fd_{\eps}(\underline\zeta)^k\big]V_2,\f1{\eps}\underline
G\fd_{\eps}(\underline\zeta)^k V_2\big)\Big]\\
{\eqdefa}&-{\eps}\lambda {\mathcal E}^{k}_{h}(V)^2+2D_1+2D_2+D_3.
\end{split}\eeq

\noindent$\bullet$ \textbf{The estimate of $D_1$}
\begin{align*}
\begin{split}
D_1=&\big(\fd_{\eps}(\underline\zeta)^k\rho^{-1}\p_tV_1,
\rho(1-\alpha\eps\underline{A})\rho\fd_{\eps}(\underline\zeta)^k\rho^{-1}V_1\big)\\
&+\big(\fd_{\eps}(\underline\zeta)^k(\p_t\rho^{-1})V_1,
\rho(1-\alpha\eps\underline{A})\rho\fd_{\eps}(\underline\zeta)^k\rho^{-1}V_1\big)
{\eqdefa}D_1^1+D_1^2.
\end{split}
\end{align*}
Applying Lemma \ref{lem:elliptic operator} gives
\begin{align*}
\begin{split}
|D_1^2| \leq&
\underline{C}\Big(|\fd_{\eps}(\underline\zeta)^k(\p_t\rho^{-1})V_1|_2
|\fd_{\eps}(\underline\zeta)^k\rho^{-1}V_1|_2+\eps|\fd_{\eps}(\underline\zeta)^k(\p_t\rho^{-1})V_1|_{H^1_\eps}
|\fd_{\eps}(\underline\zeta)^k\rho^{-1}V_1|_{H^1_\eps}\Big)\\
\leq&
\eps^3\underline{C}\Big(|V_1|^2_{H^{2k}_\eps}+\eps|V_1|^2_{H^{2k+1}_\eps}
+|V_1|^2_{H^{t_0}}|(\underline\zeta,\p_t\underline\zeta)|^2_{H^{2k+2}_\eps}
\Big).
\end{split}
\end{align*}
While thanks to \eqref{eq:linearized system-new}, we write
\begin{align}\label{prop6.2b}
\begin{split}
D_1^1=&\big(\fd_{\eps}(\underline\zeta)^k\rho^{-1}\f1{\eps}\underline
G V_2,
\rho(1-\alpha\eps\underline{A})\rho\fd_{\eps}(\underline\zeta)^k\rho^{-1}V_1\big)\\
&-\big(\fd_{\eps}(\underline\zeta)^k\rho^{-1}{\eps}\na^{\eps}_h\cdot(\underline{\bf
v}V_1\big),
\rho\big(1-\alpha\eps\underline{A})\rho\fd_{\eps}(\underline\zeta)^k\rho^{-1}V_1\big)\\
&+\big(\fd_{\eps}(\underline\zeta)^k\rho^{-1}{\eps} H_1,
\rho(1-\alpha\eps\underline{A})\rho\fd_{\eps}(\underline\zeta)^k\rho^{-1}V_1\big)\\
{\eqdefa}& D_1^{11}+D_1^{12}+D_1^{13}.
\end{split}
\end{align}
Let \beq\label{prop6.2aa} \na^{\eps}_{h,{\underline{\bf
v}}}f{\eqdefa}\f12(\na^{\eps}_h\cdot(\underline{\bf v} f)
+\underline{\bf v}\cdot\na^{\eps}_h f), \eeq  we write
\begin{align*}
\begin{split}
D_1^{12}=&-\big(\fd_{\eps}(\underline\zeta)^k\rho^{-1}{\eps}\na^{\eps}_{h,{\underline{\bf
v}}}V_1,\rho(1-\alpha\eps\underline{A})\rho\fd_{\eps}(\underline\zeta)^k\rho^{-1}V_1\big)\\
&-\f12\big(\fd_{\eps}(\underline\zeta)^k\rho^{-1}{\eps}(\na^{\eps}_h\cdot\underline{\bf
v})V_1,
\rho(1-\alpha\eps\underline{A})\rho\fd_{\eps}(\underline\zeta)^k\rho^{-1}V_1\big)\\
=&-{\eps}\big(\big[\fd_{\eps}(\underline\zeta)^k\rho^{-1},\na^{\eps}_{h,{\underline{\bf
v}}}\big]V_1,\rho(1-\alpha\eps\underline{A})
\rho\fd_{\eps}(\underline\zeta)^k\rho^{-1}V_1\big)\\
&-\f12{\eps}\big(\fd_{\eps}(\underline\zeta)^k\rho^{-1}V_1,\big[\rho(1-\alpha\eps\underline{A})
\rho,\na^{\eps}_{h,{\underline{\bf v}}}\big]\fd_{\eps}(\underline\zeta)^k\rho^{-1}V_1\big)\\
&
-\f12{\eps}\big(\fd_{\eps}(\underline\zeta)^k\rho^{-1}(\na^{\eps}_h\cdot
\underline{\bf v})V_1,
\rho(1-\alpha\underline{A})\rho\fd_{\eps}(\underline\zeta)^k\rho^{-1}V_1\big),
\end{split}
\end{align*}
which together with Lemma \ref{lem:elliptic operator}, Lemma
\ref{lem:elliptic operator-commutator} and Proposition
\ref{prop:shape derivative estimate} implies that
\begin{align*}
|D_1^{12}|\le {\eps}
\underline{C}\Big(|V_1|^2_{H^{2k}_\eps}+\eps|V_1|^2_{H^{2k+1}_\eps}
+|V_1|^2_{H^{t_0+1}}|(\underline\zeta,\mathfrak
P\underline{\psi})|^2 _{H^{2k+3}_\eps}\Big).
\end{align*}
And similarly, one has \beno \begin{split}
 |D_1^{13}|\leq &{\eps}
\underline{C}\Big(|H_1|^2_{H^{2k}_\eps}+\eps|H_1|^2_{H^{2k+1}_\eps}
+|V_1|^2_{H^{2k}_\eps}+\eps|V_1|^2_{H^{2k+1}_\eps}+(|H_1|^2_{H^{t_0}}+|V_1|^2_{H^{t_0+1}})
|\underline\zeta|^2_{H^{2k+2}_\eps}\Big).
\end{split}
\eeno Here we used the fact that $|\underline{\bf v}|_{H^s_\eps}\le
C|(\underline\zeta,\mathfrak
P\underline{\psi})|^2_{H^{s+1}_{\eps}}$. Consequently, we obatin
\beq\label{eq:energy-A} D_1=D_1^{11}+\cR_1 \eeq with $D_1^{11}$
given by \eqref{prop6.2b} and $\cR_1$ satisfying \beno \begin{split}
|\cR_1|\leq &{\eps}
\underline{C}\Big(|H_1|^2_{H^{2k}_{\eps}}+\eps|H_1|^2_{H^{2k+1}_{\eps}}
+|V_1|^2_{H^{2k}_{\eps}}+\eps|V_1|^2_{H^{2k+1}_{\eps}}\\
&+(|H_1|^2_{H^{t_0}}+|V_1|^2_{H^{t_0+1}})
|(\underline\zeta,\mathfrak P\underline{\psi},
\p_t\underline{\zeta})|^2 _{H^{2k+3}_{\eps}}\Big). \end{split} \eeno

\noindent$\bullet$ \textbf{The estimate of $D_2$}

Again thanks to \eqref{eq:linearized system-new}, we write
\begin{align}\label{prop6.2c}
\begin{split}
D_2=&-\big(\fd_{\eps}(\underline\zeta)^k(\underline{a}-\alpha\eps\underline
A)V_1,\f1{\eps}\underline G\fd_\eps(\underline\zeta)^k
V_2\big)-\big(\fd_{\eps}(\underline\zeta)^k{\eps}\underline{\bf
v}\cdot\na^{\eps}_h V_2,
\f1{\eps}\underline G\fd_{\eps}(\underline\zeta)^k V_2)\\
&+(\fd_{\eps}(\underline\zeta)^k{\eps} H_2,\f1{\eps}\underline
G\fd_\eps(\underline\zeta)^k V_2\big){\eqdefa}
D_2^{1}+D_{2}^2+D_2^{3}.
\end{split}
\end{align}
Applying Lemma \ref{lem:elliptic operator}, Lemma \ref{lem:elliptic
operator-commutator},  Proposition \ref{prop:DN-basic properties}
and Proposition \ref{prop:shape derivative estimate} yields
\begin{align*}
\begin{split}
D_{2}^2=&-{\eps}\big(\big[\fd_{\eps}(\underline\zeta)^k,\underline{\bf
v}\cdot\na^{\eps}_h\big]V_2, \f1{\eps} \underline
G\fd_{\eps}(\underline\zeta)^kV_2\big) -{\eps}\big(\underline{\bf
v}\cdot\na^{\eps}_h(\fd_{\eps}(\underline\zeta)^kV_2),
\f1{\eps} \underline G\fd_{\eps}(\underline\zeta)^kV_2\big)\\
\leq& {\eps} \underline{C}\Big(|\mathfrak PV_2|^2_{H^{2k}_\eps}
+|\mathfrak PV_2|^2_{H^{t_0+1}}|(\underline\zeta,\mathfrak
P\underline{\psi})|^2 _{H^{2k+3}_\eps}\Big).
\end{split}
\end{align*}
And similarly one has
\[
|D_2^{3}|\leq {\eps} \underline{C}\Big(|\mathfrak
PH_2|^2_{H^{2k}_\eps}+ |\mathfrak PV_2|^2_{H^{2k}_\eps} +(|\mathfrak
PH_2|^2_{H^{t_0}} +|\mathfrak PV_2|^2_{H^{t_0}})|\underline\zeta|^2
_{H^{2k+2}_\eps}\Big),
\]
which leads to \beq\label{eq:energy-B} D_2=D_2^1+\cR_2, \eeq with
$D_2^1$ given by \eqref{prop6.2c} and $\cR_2$ satisfying \beno
\begin{split}
|\cR_2|\leq& {\eps} \underline{C}\Big(|\mathfrak
PH_2|^2_{H^{2k}_\eps}+ |\mathfrak PV_2|^2_{H^{2k}_\eps}+(|\mathfrak
PH_2|^2_{H^{t_0}}+|\mathfrak
PV_2|^2_{H^{t_0+1}})|(\underline\zeta,\mathfrak
P\underline{\psi})|^2 _{H^{2k+3}_\eps}\Big).\end{split} \eeno

\noindent$\bullet$ \textbf{The estimate of $D_1+D_2$}

Thanks to \eqref{prop6.2b}, we rewrite $D_1^{11}$ as
\begin{align}
\label{prop6.2d}
\begin{split}
D_1^{11}=&\big(\fd_{\eps}(\underline\zeta)^k\rho^{-1}\f1{\eps}
\underline GV_2,\rho^2
\fd_{\eps}(\underline\zeta)^k\rho^{-1}V_1\big)-\big(\fd_{\eps}(\underline\zeta)^k\rho^{-1}\f1{\eps}\underline
G V_2, \alpha\eps\rho\underline
A\rho\fd_{\eps}(\underline\zeta)^k\rho^{-1}V_1\big)\\
=&\big(\fd_{\eps}(\underline\zeta)^k\f1{\eps} \underline GV_2,
\fd_{\eps}(\underline\zeta)^kV_1\big)+\big(\fd_{\eps}(\underline\zeta)^k\rho^{-1}\f1{\eps}
\underline GV_2,
\rho\big[\rho,\fd_{\eps}(\underline\zeta)^k\big]\rho^{-1}V_1\big)\\
&+\big(\big[\rho,\fd_{\eps}(\underline\zeta)^k\big]\rho^{-1}\f1{\eps}
\underline GV_2,
\fd_{\eps}(\underline\zeta)^kV_1\big)-\alpha\eps\big(\fd_{\eps}(\underline\zeta)^k\rho^{-1}\f1{\eps}\underline
G V_2, \rho \fd_{\eps}(\underline\zeta)^k\underline AV_1\big)\\
&-\alpha\eps\big(\fd_{\eps}(\underline\zeta)^k\rho^{-1}\f1{\eps}\underline
G V_2, \rho \big[\underline
A\rho,\fd_{\eps}(\underline\zeta)^k\big]\rho^{-1}V_1\big)\\
\eqdefa &\frak{G}_1+\frak{G}_2+\frak{G}_3+\frak{G}_4+\frak{G}_5.
\end{split}
\end{align}
It follows from Lemma \ref{lem:elliptic operator}-\ref{lem:elliptic
operator-commutator}, Proposition \ref{prop:DN-basic properties} and
Proposition \ref{prop:shape derivative estimate} that \beno
\begin{split}
\bigl|\frak{G}_2\bigr|+\bigl|\frak{G}_3\bigr|\leq &{\eps}
\underline{C}\Big(|V_1|^2_{H^{2k}_\eps}+|\mathfrak P V_2|^2
_{H^{2k}_\eps}+(|V_1|^2_{H^{t_0+1}}+|\mathfrak P V_2|^2
_{H^{t_0+1}})|\underline\zeta|^2_{H^{2k+2}_\eps}\Big). \end{split}
\eeno Since $\underline{A}(\rho \cdot)$ can be written as \beno
\underline{A}(\rho f)=-\fd_{\eps}(\underline\zeta) f+h_{\underline
A}\cdot\na^{\eps}_h f+h_{\underline A}'f \eeno for some smooth
function $h_A, h_A'$ depending on  $\eps^\f32\underline\zeta,$ $
\eps^\f32\na_h^{\eps}\underline\zeta,
\eps^\f32(\na_h^{\eps})^2\underline\zeta$, this implies that
$$
\big[\underline
A\rho,\fd_{\eps}(\underline\zeta)^k\big]=\big[h_{\underline
A}\cdot\na^\eps_h,\fd_{\eps}(\underline\zeta)^k\big]
+\big[h_{\underline A}',\fd_{\eps}(\underline\zeta)^k\big].
$$
Therefore, one has
\[\bigl|\frak{G}_5\bigr|\leq
{\eps}\underline{C}\Big(\eps|V_1|^2_{H^{2k+1}_{\eps}}+|\mathfrak P
V_2|^2 _{H^{2k}_{\eps}}+(|V_1|^2_{H^{t_0+1}}+|\mathfrak P V_2|^2
_{H^{t_0+1}})|\underline\zeta|^2_{H^{2k+2}_{\eps}}\Big).
\]
As a consequence, we obtain \beq\label{eq:energy-A11}
D_1^{11}=\frak{G}_1+\frak{G}_4+\cR_3 \eeq with $\frak{G}_1,$
$\frak{G}_4$ given by \eqref{prop6.2d} and $\cR_3$ satisfying \beno
\begin{split}
|\cR_3|\le &
{\eps}\underline{C}\Big(|V_1|^2_{H^{2k}_\eps}+\eps|V_1|^2_{H^{2k+1}_\eps}+|\mathfrak
P V_2|^2 _{H^{2k}_\eps}+(|V_1|^2_{H^{t_0+1}}+|\mathfrak P V_2|^2
_{H^{t_0+1}})|\underline\zeta|^2_{H^{2k+2}_\eps}\Big).
\end{split}\eeno

On the other hand, let $\underline a= 1+{\eps} \underline b$ with
$\underline b{\eqdefa}{\eps}\underline{\bf v}
\cdot\na^{\eps}_h\underline Z+\p_t\underline Z$. Thanks to
\eqref{prop6.2c}, we have \beq\label{prop6.2e}
\begin{split}
D_2^1=&-\big(\fd_{\eps}(\underline\zeta)^kV_1,\fd_{\eps}(\underline\zeta)^k\f1{\eps}\underline
GV_2\big) +\big(\fd_{\eps}(\underline\zeta)^k\alpha\eps\underline
AV_1, \rho\fd_{\eps}(\underline\zeta)^k\rho^{-1}\f1{\eps}\underline
GV_2\big)
\\&-\big(\fd_{\eps}(\underline\zeta)^kV_1, \big[\f1{\eps}\underline G,\fd_{\eps}(\underline\zeta)^k\big]V_2\big)\
-\big(\fd_{\eps}(\underline\zeta)^k{\eps}\underline bV_1,
\f1{\eps}\underline G\fd_{\eps}(\underline\zeta)^kV_2\big)
\\
&+\big(\fd_{\eps}(\underline\zeta)^k\alpha\eps\underline AV_1,
\rho\big[\rho^{-1}\f1{\eps}\underline
G,\fd_{\eps}(\underline\zeta)^k\big]V_2\big)
{\eqdefa}\frak{H}_{1}+\frak{H}_{2}+\frak{H}_{3}+\frak{H}_{4}+\frak{H}_{5}.\end{split}
\eeq Applying Lemma \ref{lem:elliptic operator} and Proposition
\ref{prop:DN-commutator} gives
\begin{align*}
\begin{split}
\bigl|\frak{H}_{3}\bigr|\leq &|\fd_{\eps}(\underline\zeta)^k
V_1|_2\big|\big[\f1{\eps}\underline G,
\fd_{\eps}(\underline\zeta)^k\big]V_2\big|_2\\
\leq&{\eps} \underline{C}\Big(|V_1|^2_{H^{2k}_{\eps}}+|\mathfrak P
V_2|^2 _{H^{2k}_{\eps}}+(|V_1|^2_{H^{t_0}}+|\mathfrak P V_2|^2
_{H^{t_0+1}})|\underline\zeta|^2_{H^{2k+2}_{\eps}}\Big),
\end{split}
\end{align*}
and  Theorem \ref{thm:DN-commutator-sharp} ensures that
\begin{align*}
\begin{split}
\bigl|\frak{H}_{5}\bigr|\leq&
\eps\underline{C}|\fd_{\eps}(\underline\zeta)^{k-1}\underline
AV_1|_{H^1_\eps}
\big|\big[\rho^{-1}\f1{\eps}\underline G,\fd_{\eps}(\underline\zeta)^k\big]V_2\big|_{H^1_\eps}\\
\leq& \eps\underline{C}\Big(\eps|V_1|^2_{H^{2k+1}_\eps} +|\mathfrak
PV_2|^2_{H^{2k}_{\eps}} +(|V_1|^2_{H^{t_0+2}_\eps}+|\mathfrak
PV_2|^2_{H^{t_0+2}}) |\underline \zeta|^2_{H^{2k+4}_\eps}\Big).
\end{split}
\end{align*}
While it follows from  Proposition \ref{prop:shape derivative} and
Proposition \ref{prop:shape derivative estimate} that \beno
|\underline b|_{H^s_{\eps}} \leq
\sqrt{\eps}\underline{C}\Big(|(\underline \zeta,\mathfrak
P\underline\psi)|_{H^{s+2}_\eps}+ |(\p_t\underline
\zeta,\p_t\mathfrak P\underline\psi)|_{H^{s+1}_{\eps}}\Big), \eeno
from which and Proposition \ref{prop:DN-basic properties}, Lemma
\ref{lem:elliptic operator} and Lemma \ref{lem:elliptic
operator-commutator}, we infer that
\begin{align*}
\begin{split}
 \bigl|\frak{H}_{4}\bigr|&\leq {\eps}
\underline{C}|\fd_{\eps}(\underline\zeta)^k(\underline
bV_1)|_{H^{1}_{\eps}}
|\mathfrak P\fd_{\eps}(\underline\zeta)^k V_2|_2\\
&\leq \eps\underline{C}\Big(\eps|V_1|^2_{H^{2k+1}_{\eps}}
+|\mathfrak PV_2|^2_{H^{2k}_{\eps}} +(|V_1|^2_{H^{t_0}}+|\mathfrak
PV_2|^2_{H^{t_0}})\\
&\qquad\times (|(\underline \zeta,\mathfrak
P\underline\psi)|^2_{H^{2k+3}_\eps} +|(\p_t\underline
\zeta,\p_t\mathfrak P\underline\psi)|^2_{H^{2k+3}_\eps})\Big).
\end{split}
\end{align*}
Consequently, we obtain
 \beq\label{eq:energy-B1}
D_2^1=\frak{H}_{1}+\frak{H}_2+\cR_4 \eeq with $\frak{H}_{1},$
$\frak{H}_2$ given by \eqref{prop6.2e} and $\cR_4$ satisfying
\begin{align*}
\begin{split}
|\cR_4|\le &
\eps\underline{C}\Big(|V|_{H^{2k}_\eps}+\eps|V_1|^2_{H^{2k+1}_{\eps}}
+|\mathfrak PV_2|^2_{H^{2k}_{\eps}}+(|V_1|^2_{H^{t_0}}+|\mathfrak PV_2|^2_{H^{t_0+1}})\\
&\times(|(\underline \zeta,\mathfrak
P\underline\psi)|^2_{H^{2k+3}_\eps} +|(\p_t\underline
\zeta,\p_t\mathfrak P\underline\psi)|^2_{H^{2k+3}_\eps})\Big).
\end{split}
\end{align*}

Summing up (\ref{eq:energy-A}), \eqref{eq:energy-B},
\eqref{eq:energy-A11}  and (\ref{eq:energy-B1}) yields that
\begin{align}\label{eq:energy-A+B}
\begin{split}
|D_1+D_2|&=|\cR_1+\cR_2+\cR_3+\cR_4|\\
&\le
\eps\underline{C}\Big(|V_1|_{H^{2k}_\eps}+\eps|V_1|^2_{H^{2k+1}_{\eps}}
+|\mathfrak PV_2|^2_{H^{2k}_{\eps}}+|H_1|^2_{H^{2k}_{\eps}}
+\eps|H_1|^2_{H^{2k+1}_{\eps}}\\
&\quad+|\mathfrak
PH_2|^2_{H^{2k}_{\eps}}+(|V_1|^2_{H^{t_0}}+|\mathfrak
PV_2|^2_{H^{t_0+1}}+|(H_1,\mathfrak PH_2)|_{H^{t_0}})\\
&\qquad\times (|(\underline \zeta,\mathfrak
P\underline\psi)|^2_{H^{2k+3}_\eps} +|(\p_t\underline
\zeta,\p_t\mathfrak P\underline\psi)|^2_{H^{2k+3}_\eps})\Big).
\end{split}
\end{align}

\noindent$\bullet$ \textbf{The estimate of $D_3$}

The estimate of $D_3$ is much easier. For example,   we get by
applying  Proposition \ref{prop:DN-commutator-time} to the second
term in $D_3$ that
\begin{align*}
\bigl|\big(\fd_{\eps}(\underline\zeta)^kV_2,\big[\p_t,\f1{\eps}\underline
G\big]\fd_{\eps}(\underline\zeta)^kV_2)\bigr| \leq&
\eps\underline{C}|\mathfrak P
\fd_{\eps}(\underline\zeta)^kV_2|^2_2\\
\leq& \eps\underline{C}\Big(|\mathfrak PV_2|^2_{H^{2k}_{\eps}}+
|\mathfrak PV_2|^2_{H^{t_0}}|\underline
\zeta|^2_{H^{2k+2}_{\eps}}\Big).
\end{align*}
Following exactly the same line, we can obtain similar estimates for
the other terms in $D_3.$ This gives
\begin{align}\label{eq:energy-C}\begin{split}
|D_3|\le&
\eps\underline{C}\Big(|V_1|_{H^{2k}_\eps}+\eps|V_1|^2_{H^{2k+1}_{\eps}}
+|\mathfrak PV_2|^2_{H^{2k}_{\eps}}+(|V_1|^2_{H^{t_0}}+|\mathfrak
PV_2|^2_{H^{t_0+1}}) |(\underline \zeta,\p_t\underline
\zeta)|^2_{H^{2k+3}_\eps}\Big).
\end{split}
\end{align}

Plugging (\ref{eq:energy-A+B}) and (\ref{eq:energy-C}) into
(\ref{eq:energy-high}) results in \beq\label{eq:energy-highest}
\begin{split}
e^{{\eps}\lambda t}\f d{dt}\Big(e^{-{\eps}\lambda t}\mathcal
E^{k}_{h}(V)^2\Big) &\leq -{\eps}\lambda \mathcal E^{k}_{h}(V)^2
+\eps\underline{C}\Big(|V_1|^2_{H^{2k}_{\eps}}+\eps|V_1|^2_{H^{2k+1}_{\eps}}
+|\mathfrak
PV_2|^2_{H^{2k}_{\eps}}\\
&\quad+|H_1|^2_{H^{2k}_{\eps}}+\eps|H_1|^2_{H^{2k+1}_{\eps}}
+|\mathfrak PH_2|^2_{H^{2k}_{\eps}}+(|V_1|^2_{H^{t_0}}+|\mathfrak
PV_2|^2_{H^{t_0+1}}\\
&\quad+|(H_1,\mathfrak PH_2)|^2_{H^{t_0}})(|(\underline
\zeta,\mathfrak P\underline\psi)|^2_{H^{2k+3}_{\eps}}
+|(\p_t\underline \zeta,\p_t\mathfrak
P\underline\psi)|^2_{H^{2k+3}_\eps})\Big).
\end{split}
\eeq

\no{\bf Step 2.\,} Lower order energy estimate.

Similar to \eqref{eq:energy-high}, we have
\begin{align}\label{eq:energy-low}
\begin{split}
e^{\eps\lam t}\f d{dt}\big(e^{-\eps\lam t}\mathcal
E^k_{l}(V(t))^2\Big)
 &=-\eps\lam \mathcal E^k_{l}(V)^2
+2(\Lam^k\p_tV_1,(1-\alpha\eps \underline A)\Lam^kV_1)\\
&\quad+2(\Lam^k\p_tV_2,\f1\eps\underline G\Lam^kV_2)+\Big[(\Lam^kV_1,[\p_t,1-\alpha\eps\underline
A]\Lam^kV_1)\\
&\quad+(\Lam^kV_2,[\p_t,\f1\eps \underline G]\Lam^k V_2)\Big]+2\eps^2\big(\Lam^{k-1} \p_tV_2,
 \Lam^{k-1} V_2\big)\\\
&{\eqdefa}-\eps\lam \mathcal E^k_{l}(V)^2+2E_1+2E_2+E_3+2E_4.
\end{split}
\end{align}
First of all, it is easy to show that \ben\label{eq:energylow-C}
|E_3|\le \eps\underline{C}\Big(\eps|\na_h^{\eps}
V_1|^2_{H^k}+|\mathfrak PV_2|^2_{H^k}\Big). \een

\noindent$\bullet$ \textbf{The estimate of $E_1$}

Thanks to \eqref{eq:linearized system-new}, we write
\begin{align*}
\begin{split}
E_1=&-\big(\Lam^k\eps \na_h^{\eps}\cdot (\underline{\bf
v}V_1),(1-\alpha\eps\underline
A)\Lam^kV_1\big)+\big(\Lam^k\f1\eps\underline GV_2,
(1-\alpha\eps\underline
A)\Lam^kV_1\big)\\
&+\big(\Lam^k\eps H_1,(1-\alpha\eps\underline
A)\Lam^kV_1\big){\eqdefa}E_1^1+E_1^2+E_1^3.
\end{split}
\end{align*}
It is easy to observe  that
\[
|E_1^3|\le \eps\underline{C}\Big(|H_1|_{H^k}^2+\eps|\na_h^{\eps}
H_1|_{H^k}^2+|V_1|_{H^k}^2+\eps|\na_h^{\eps} V_1|_{H^k}^2\Big).
\]
Whereas similar to \eqref{prop6.2aa}, we split $E_1^1$ further as
\begin{align*}\begin{split}
E_1^1=&-\eps\big(\big[\Lam^k,{\na^{\eps}}_{h,\underline {\bf
v}}\big]V_1,(1-\alpha\eps\underline
A)\Lam^kV_1\big)-\f12\eps\big(\Lam^k V_1,\big[1-\alpha\eps\underline
A,{\na^{\eps}}_{h,\underline {\bf
v}}\big]\Lam^kV_1\big)\\
&-\f12\eps\big(\Lam^k (\na_h^{\eps}\cdot\underline{\bf
v})V_1,(1-\alpha\eps\underline
A)\Lam^kV_1\big){\eqdefa}E_1^{11}+E_1^{12}+E_1^{13}.
\end{split}
\end{align*}
Note from Remark \ref{rem:DN operator} that \beno
|\Lam^2_{\eps}\underline {\bf v}|_{H^s}\le
\underline{C}\big(|\Lam^3_\eps \mathfrak P\underline
\psi|^2_{H^s}+|\Lam^4_\eps\underline \zeta|^2_{H^s}\big), \eeno and
\[[\Lam^k,{\na^{\eps}}_{h,\underline {\bf
v}}]=[\Lam^k,\underline{\bf v}]\cdot\na_h^{\eps}+\f12
[\Lam^k,\na_h^{\eps}\cdot\underline{\bf v}],\] so it follows from
Lemma2.4 that
\begin{align*}
|E_1^{11}| &\le \eps
\underline{C}\big(|[\Lam^k,{\na_h^{\eps}}_{\underline {\bf
v}}]V_1|_2+\eps^\f12|\na_h^{\eps}[\Lam^k,{\na_h^{\eps}}_{\underline
{\bf
v}}]V_1|_2\big)\big(|V_1|_{H^k}+\eps^\f12|\na_h^{\eps} V_1|_{H^k}\big)\\
&\le\eps\underline{C}\big(|V_1|_{H^k}+\eps^\f12|\na_h^{\eps}
V_1|_{H^k}+|V_1|_{H^{t_0+1}}|\Lam^2_\eps\underline {\bf
v}|_{H^k}\big)
\big(|V_1|_{H^k}+\eps^\f12|\na_h^{\eps} V_1|_{H^k}\big)\\
&\le\eps \underline{C}\Big(|V_1|^2_{H^k}+\eps|\na_h^{\eps}
V_1|^2_{H^k}+|V_1|^2_{H^{t_0+1}}\big(|\Lam^3_\eps \mathfrak
P\underline \psi|^2_{H^k}+|\Lam^3_\eps\underline
\zeta|^2_{H^k}\big)\Big).
\end{align*}
 Similarly, we have \beno\begin{split}
  |E_1^{12}+E_1^{13}|\le& \eps
\underline{C}\Big(|V_1|^2_{H^k}+\eps|\na_h^{\eps}
V_1|^2_{H^k}+|V_1|^2_{H^{t_0+1}}\big(|\Lam^3_\eps \mathfrak
P\underline \psi|^2_{H^k}+|\Lam^3_\eps\underline
\zeta|^2_{H^k}\big)\Big).
\end{split} \eeno To handle $E_1^2$, we write
\begin{align*}
\begin{split}
E_1^2=&\big(\f1\eps \underline G\Lam^kV_2,(1-\alpha\eps\underline
A)\Lam^k V_1\big)+\big(\big[\Lam^k,\f1\eps \underline
G\big]V_2,(1-\alpha\eps\underline A)\Lam^k V_1\big)
{\eqdefa}E_1^{21}+E_1^{22}. \end{split}
\end{align*}
As
\[
|\Lam^2_\eps\mathfrak
Pu|_{H^{k-1}}=|(\Lam^{2k}_\eps)^{\f1k}(\Lam^k)^{\f{k-1}k}\mathfrak
Pu|_2\le C\big(|\Lam^{2k}_\eps\mathfrak Pu|_2+|\Lam^k\mathfrak
Pu|_2\big),
\]
which along with  Remark \ref{rem:DN-commutator} ensures that
\begin{align*}
\begin{split}
|E_1^{22}| \le & \underline{C} \big(|\big[\Lam^k,\f1\eps \underline
G\big]V_2|_2+\eps^\f12|D^{\eps}_{h}\big[\Lam^k,\f1\eps
\underline G\big]V_2|_2)\big(|V_1|_{H^k}+\eps^\f12|\na_h^{\eps} V_1|_{H^k}\big)\\
\le&\eps \underline{C}\Big(|V_1|_{H^k}^2+\eps|\na_h^{\eps}
V_1|_{H^k}^2+ |\mathfrak PV_2|_{H^{2k}_\eps}^2+|\mathfrak
PV_2|_{H^k}^2 +|\mathfrak PV_2|_{H^{t_0+2}}^2|\Lam^3_\eps \underline
\zeta|_{H^k}^2\Big).
\end{split}
\end{align*}
As a consequence, we obtain \beq\label{eq:energy-low-A}
E_1=\big(\f1\eps \underline G\Lam^kV_2,(1-\alpha\eps\underline
A)\Lam^k V_1\big)+\frak{R}_1 \eeq with $\frak{R}_1$ satisfying
\begin{align*}
\begin{split}
|\frak{R}_1|\le& \eps
\underline{C}\Big(|V_1|_{H^k}^2+\eps|\na_h^{\eps} V_1|_{H^k}^2+
|\mathfrak PV_2|_{H^{2k}_\eps}^2+|\mathfrak
PV_2|_{H^k}^2+|H_1|_{H^k}^2\\
&\quad+\eps|\na_h^{\eps} H_1|_{H^k}^2+\big(|\mathfrak
PV_2|_{H^{t_0+2}}^2+|V_1|^2_{H^{t_0+1}}\big)|\Lam^4_\eps(\underline
\zeta, \mathfrak P\underline \psi)|^2_{H^k}\Big).
\end{split}
\end{align*}

\noindent$\bullet$ \textbf{The estimate of $E_2$}

Thanks to \eqref{eq:linearized system-new}, we write
\begin{align*}
\begin{split}
E_2=&-\big(\Lam^k(\underline a-\alpha\eps\underline
A)V_1,\f1\eps\underline G\Lam^kV_2\big)-\big(\Lam^k\eps\underline
{\bf
v}\cdot\na_h^{\eps} V_2,\f1\eps\underline G\Lam^kV_2\big)\\
&+\big(\Lam^k\eps H_2,\f1\eps\underline G\Lam^kV_2\big)  {\eqdefa}
E_2^1+E_2^2+E_2^3.
\end{split}
\end{align*}
Applying Proposition \ref{prop:DN-basic properties} gives
\[|E_2^3|\le \eps \underline{C}|\mathfrak PH_2|_{H^k}|\mathfrak PV_2|_{H^k}.\]
While notice that
\begin{align*}
\begin{split}
E_2^1=&-\big((1-\alpha\eps\underline A)\Lam^kV_1,\f1\eps\underline
G\Lam^kV_2\big)-\big(\big[\Lam^k,1-\alpha\eps\underline
A\big]V_1,\f1\eps\underline G\Lam^kV_2\big)\\
&-\big(\Lam^k\eps \underline b V_1, \f1\eps\underline
G\Lam^kV_2\big){\eqdefa}E_2^{11}+E_2^{12}+E_2^{13}.
\end{split}
\end{align*}
It follows from Lemma \ref{lem:commutator estimate}, Proposition
\ref{prop:DN-basic properties}, and  an interpolation argument that
\begin{align*}
\begin{split}
|E_2^{12}|\le& \underline{C}||D^{\eps}_{h}|\big[\Lam^k,\al\eps \underline A\big]V_1|_2|\mathfrak P V_2|_{H^k}\\
\le& \eps\underline{C}\Big(|\mathfrak PV_2|_{H^k}^2+\eps|
V_1|_{H^{2k+1}_{\eps}}^2+\eps|\na_h^{\eps}
V_1|_{H^k}^2+|V_1|_{H^k}^2+|V_1|_{H^{t_0+3}}^2|\Lam^3_{\eps}
\underline \zeta|_{H^k}^2\Big).
\end{split}
\end{align*}
Similarly, one has \beno
\begin{split}
|E_2^{13}|\le&\eps\underline{C}\Big(|\mathfrak
PV_2|_{H^k}^2+\eps|\na_h^{\eps}
V_1|_{H^k}^2+|V_1|_{H^k}^2\\
&+|V_1|_{H^{t_0+1}}^2(|\Lambda_{\eps}^2(\underline \zeta,\mathfrak
P\underline\psi)|^2_{H^{k}}+|\Lambda_{\eps}^2(\p_t\underline
\zeta,\p_t\mathfrak P\underline\psi)|^2_{H^{k}})\Big).\end{split}
\eeno Noting that
\begin{align*} E_2^2=-\eps\big(\big[\Lam^k,\underline {\bf v}\big]\cdot\na_h^{\eps} V_2,\f1\eps
\underline G\Lam^k V_2)-\eps\big(\underline{\bf
v}\cdot\na_h^{\eps}\Lam^k V_2, \f1\eps \underline G\Lam^k V_2\big),
\end{align*}
 from which,  Proposition \ref{prop:DN-basic properties} and Lemma
\ref{lem:commutator estimate}, we infer that
\begin{align*}
|E_2^2|\le \eps\underline{C}\Big(|\mathfrak PV_2|_{H^k}^2+|\mathfrak
PV_2|_{H^{2k}}^2 +|\mathfrak P
V_2|_{H^{t_0+1}}^2|\Lambda_{\eps}^3(\underline \zeta,\mathfrak
P\underline\psi)|^2_{H^{k}}\Big).
\end{align*}
Therefore, we obtain \beq\label{eq:energylow-B} E_2=-\big(\f1\eps
\underline G\Lam^kV_2,(1-\alpha\eps\underline A)\Lam^k
V_1\big)+\frak R_2 \eeq with $\cR_2$ satisfying
\begin{align*}
|\frak R_2|\le& \eps
\underline{C}\Big(|V_1|_{H^k}^2+\eps|\na_h^{\eps} V_1|_{H^k}^2+\eps|
V_1|_{H^{2k+1}_\eps}^2+ |\mathfrak PV_2|_{H^{2k}_\eps}^2+|\mathfrak
PV_2|_{H^k}^2+|\mathfrak
PH_2|_{H^k}^2\\
&\quad+\big(|\mathfrak
PV_2|_{H^{t_0+2}}^2+|V_1|^2_{H^{t_0+3}}\big)(|\Lam^3_\eps(\underline
\zeta, \mathfrak P\underline
\psi)|^2_{H^k}+|\Lam^2_\eps(\p_t\underline \zeta, \p_t\mathfrak
P\underline \psi)|^2_{H^k})\Big).
\end{align*}

\noindent$\bullet$ \textbf{The estimate of $E_4$}

Again thanks to \eqref{eq:linearized system-new}, we write
\begin{align*}
\begin{split}
E_4=&-\eps^2\big(\Lambda^{k-1}(\underline a-\alpha\eps\underline
A)V_1,\Lambda^{k-1} V_2\big)-\eps^2\big(\Lam^{k-1}\eps\underline
{\bf
v}\cdot\na_h^{\eps} V_2, \Lambda^{k-1} V_2\big)\\
&+\eps^3\big(\Lambda^{k-1} H_2, \Lambda^{k-1}
V_2\big){\eqdefa}E_4^1+E_4^2+E_4^3.
\end{split}
\end{align*}
It is easy to observe that \beno \begin{split} &|E_4^1|\le \eps
\underline{C}\Big(|V_1|^2_{H^k}+|V_1|^2_{H^{2k}_\eps}+\eps^2|V_2|^2_{H^{k-1}}
+|V_1|_{H^{t_0+1}}^2(|\Lambda_{\eps}^2(\underline \zeta,\mathfrak
P\underline\psi)|^2_{H^{k-1}}
\\
&\qquad\quad+|\Lambda_{\eps}^2(\p_t\underline \zeta,\p_t\mathfrak P\underline\psi)|^2_{H^{k-1}})\Big),\\
&|E_4^3|\le \eps
\underline{C}\Big(\eps^2|H_2|_{H^{k-1}}^2+\eps^2|V_2|^2_{H^{k-1}}\Big).
\end{split}
\eeno And one gets by using integration by parts that \beno
E_4^2=-\eps^3\big(\big[\Lam^{k-1},\underline {\bf
v}\big]\cdot\na_h^{\eps} V_2, \Lambda^{k-1} V_2\big)
+\f12\eps^3\big((\na_h^{\eps}\cdot\underline {\bf v})\Lam^{k-1}V_2,
\Lambda^{k-1} V_2\big), \eeno which together with Lemma
\ref{lem:commutator estimate} implies that \beno |E_4^2|\le \eps
\underline{C}\Big(\eps^2|V_2|^2_{H^{k-1}}
+\eps^2|V_2|_{H^{t_0+1}}^2|\Lambda_{\eps}^2(\underline
\zeta,\mathfrak P\underline\psi)|^2_{H^{k-1}}\Big). \eeno Then we
arrive at
\begin{align}\label{eq:energylow-D}
|E_4|\le& \eps \underline{C}\Big(|V_1|^2_{H^k}+|V_1|^2_{H^{2k}_\eps}
+\eps^2|V_2|^2_{H^{k-1}}+\big(|V_1|_{H^{t_0+1}}^2\\
&+\eps^2|V_2|_{H^{t_0+1}}^2\big)(|\Lambda_{\eps}^2(\underline
\zeta,\mathfrak P\underline\psi)|^2_{H^{k-1}}
+|\Lambda_{\eps}^2(\p_t\underline \zeta,\p_t\mathfrak
P\underline\psi)|^2_{H^{k-1}})\Big). \nonumber
\end{align}

Plugging (\ref{eq:energylow-C})-(\ref{eq:energylow-D}) into
(\ref{eq:energy-low}) results in
\begin{align} e^{{\eps}\lambda t}\frac
d{dt}\Big(e^{-{\eps}\lambda t}\mathcal E^k_l(V(t))^2\Big) &
\leq-{\eps}\lambda \mathcal E^k_l(V)^2+\eps
\underline{C}\Big(|V_1|_{H^k}^2+\eps|
V_1|_{H^{2k+1}_{\eps}}^2+\eps|\na_h^{\eps} V_1|_{H^k}^2+ |\mathfrak
PV_2|_{H^{2k}_{\eps}}^2\nonumber\\
&\quad+|\mathfrak
PV_2|_{H^k}^2+\eps^2|V_2|^2_{H^{k-1}}+|H_1|_{H^k}^2+|\mathfrak
PH_2|_{H^k}^2+\eps|\na_h^{\eps}
H_1|_{H^k}^2\nonumber\\
&\quad+\eps^2|H_2|^2_{H^{k-1}}+\big(|\mathfrak
P(V_2,H_2)|_{H^{t_0+2}}^2
+|(V_1,H_1)|^2_{H^{t_0+3}}\big)\label{eq:energy-lowest}\\
&\qquad\times\big(|\Lam^4_{\eps}(\underline \zeta, \mathfrak
P\underline \psi)|^2_{H^k}+|\Lam^3_{\eps}(\p_t\underline \zeta,
\p_t\mathfrak P\underline \psi)|^2_{H^k}\big)\Big).\nonumber
\end{align}

\no{\bf Step 3.} Full energy estimates.

Combining (\ref{eq:energy-highest}) with (\ref{eq:energy-lowest}),
we get by applying Lemma \ref{lem:energy relation} that
\begin{align*}
\begin{split}
&e^{{\eps}\lambda t}\frac d{dt}\Big(e^{-{\eps}\lambda t}\mathcal
E^k(V(t))^2\Big)
\leq-{\eps}\lambda E^k(V)^2+\eps \underline{C}E^k(V)^2\\
&\qquad+\eps \underline{C}E^k(H)^2+\eps
\underline{C}\big(E^{t_0+3}_l(V)^2+E^{t_0+3}_l(H)^2\big)
\big|\Lambda^5_{\eps}(\underline{\zeta},\underline{\psi})\big|^2_{Y^k_T}.
\end{split}
\end{align*}
Taking $\lambda=\underline{C}$ in the above inequality and applying
Lemma \ref{lem:energy relation} again yields
\begin{align}
E^k(V(t))^2&\le\underline{C}\cI^k(t,V_0,H)+\underline{C}(\cI^{t_0}(t,V_0,H)
+\eps E^{t_0}_l(V(t))^2)|\Lambda^3_\eps\underline{\zeta}|^2_
{X^k_T}\nonumber\\
&\quad+\eps\underline{C}\int_0^t\big(E^{t_0+3}_l(V)^2
+E^{t_0+3}_l(H)^2\big)d\tau\big|\Lambda^5_\eps
(\underline{\zeta},\underline{\psi})\big|^2_{Y^k_T}.
\label{eq:total-energy}
\end{align}
On the other hand, it follows from (\ref{eq:energy-lowest}) that
\begin{align*}
\begin{split}
E^k_l(V(t))^2\le &\underline{C}\cI^k(t,V_0,H)+\underline{C}\eps \int_0^tE^k(V(\tau))^2d\tau\\
&+\underline{C}\eps \int_0^t
\big(E^{t_0+3}_l(V)^2+E^{t_0+3}_l(H)^2\big)d\tau
\big|\Lambda^5_\eps(\underline{\zeta},\underline{\psi})\big|^2_{Y^k_T}.
\end{split}
\end{align*}
After taking $k=t_0+3$ in the above inequality and applying
Gronwall's inequality, we plug the resulting inequality into
(\ref{eq:total-energy})(where $k=5$) to yield that \beno
E^5(V(t))^2\le \underline{C}\cI^5(t,V_0,H), \eeno which together
with (\ref{eq:total-energy}) implies that \beno E^k(V(t))^2\le
\underline{C}\Big(\cI^k(t,V_0,H)
+\cI^5(t,V_0,H)\big|\Lambda^5_\eps(\underline{\zeta},\underline{\psi})\big|^2_{Y^k_t}\Big).
\eeno This completes the  proof of Proposition \ref{prop:linearized
system-tri}.\ef

\renewcommand{\theequation}{\thesection.\arabic{equation}}
\setcounter{equation}{0}
\section{Large time existence for the nondimensionalized water-wave  system}\label{well}

The goal of  this section is to use a modified Nash-Moser iteration
theorem in the Appendix and the uniform estimates obtained for the
linearized system \eqref{eq:linearized system}  to solve the
water-wave system \eqref{eq:Hamiltonian form-non} on
$[0,\f{T}\eps].$ As noticed in Remark\ref{rmk1.2}, there is no need
to prove Theorem\ref{thm:5-KP approximation} here.

We start the proof of Theorem \ref{thm:KP approximation} with the
following lemma:

\begin{lem}\label{lem:linear system} {\sl For all $U^0\in
\frak{X}^{s},$ we denote $S^{\eps}(t)$ the solution operator to the
linear system
 \beq\label{eq:linear system} \left\{\begin{array}{ll}
\p_tV+\f1{\eps}\mathcal L V=0,\\
V|_{t=0}=U^0.
\end{array}\right.
\eeq Then for all $T>0,$ $S^{\eps}(\cdot) U^0$ is well-defined in
$C([0,T];\frak{X}^{s})$. Moreover, for all $t\in[0,T]$, there holds
\beq\label{lem7.1} |S^{\eps}(t)U^0|_{\frak{X}^{s}}\leq
C(T,\f1{h_0},|b|_{H^{2s+5}})|U^0|_{\frak{X}^{s}}. \eeq}
\end{lem}

\begin{proof} Indeed \eqref{lem7.1}  can be deduced from the proof of
Proposition \ref{prop:linearized system-tri} in the particular  case
when $\underline{U}=(0,0)$.\end{proof}

With  $V\eqdefa S^{\eps}(t)U^0$ thus defined,  we shall seek for a
solution $U$ of (\ref{eq:Hamiltonian form-non}) under the form
$U=V+W$, which is equivalent to solve  the following system of $W$:
\beq\label{eq:nonlinear system} \left\{\begin{array}{ll}
\p_tW+\f1{\eps}\mathcal L W+\mathcal F[t,W]=h,\\
W|_{t=0}=(0,0)^T,
\end{array}\right.
\eeq where $\mathcal F[t,W]\eqdefa \mathcal A[V+W]-\mathcal A[V]$
and $h\eqdefa-\mathcal A[V]$.

\begin{lem}\label{lem:nonlinear operator} {\sl Let $T>0$ and $s\ge m_0$.
Then we have
\begin{itemize}
 \item[1)]
 The mapping $\mathcal L: X^{s+1}\rightarrow X^{s}$ is
well-defined and continuous, and the family of linear solution
operators $(S^{\eps}(t))_{0<{\eps} <1}$ is uniformly bounded in
$C([-T,T];\mathfrak L(X^{s},X^{s}))$;

\item[2)] For  $0\le j\le2$,  $W\in X^{s+2}(\R^2)$ and $(W_1,\cdots,W_j)\in
X^{s+2}(\R^2)^j,$ \beno
\begin{split}
\sup_{t\in[0,T]}|d_{W}^j&\mathcal F[t,W](W_1,\cdots,W_j)| _{X^{s}}
\leq C(s,T,|W|_{X^{m_0}})\\
&\quad\times\Big( \sum^j_{k=1}|\Lambda_{\eps}^2
W_k|_{X^{s}}\prod_{l\neq k}|W_l|_{X^{t_0+1}}
+|\Lambda_{\eps}^2W|_{X^{s}}\prod^j_{k=1}|W_k|_{X^{t_0+1}} \Big);
\end{split}
\eeno

\item[3)] For  $0\le j\le2$, $W\in X^{s}(\R^2)$ and $(W_1,\cdots,W_j)\in
X^{s}(\R^2)^j,$ \beno
\begin{split}
\sup_{t\in[0,T]}|\Lambda_{\eps}^{-2}d_{W}^j&\mathcal
F[t,W](W_1,\cdots,W_j)| _{X^{s}} \leq C(s,T,|W|_{X^{m_0}})\\
&\quad\times\Big( \sum^j_{k=1}| W_k|_{X^{s}}\prod_{l\neq
k}|W_l|_{X^{t_0+1}}+|W|_{X^{s}}\prod^j_{k=1}|W_k|_{X^{t_0+1}} \Big)
\end{split}
\eeno \end{itemize}}\end{lem}

This lemma is a direct consequence of Proposition \ref{prop:shape
derivative estimate} and Remark \ref{rem:DN operator}.

\noindent{\bf Proof of Theorem  \ref{thm:KP approximation}.}\, With
the above preparations, this proof is much similar to that of
Theorem 5.1 in \cite{Lannes-Inven}, so we only sketch its proof
here. Indeed rescaling the system (\ref{eq:Hamiltonian
form-non-new}) by using a new time variable $t'={\eps} t$, we only
need to show that there exists a $T>0$ independent of $\eps$ so that
 the following system has a unique solution on $[0,T]:$  \beq\label{eq:rescaling
water wave system} \left\{\begin{array}{ll}
\p_tU+\f1{\eps}\mathcal L U+\mathcal A[U]=0,\\
U|_{t=0}=U^0.
\end{array}\right.
\eeq As shown above, the solution of (\ref{eq:rescaling water wave
system}) can be equivalently decomposed into the sum of solution of
(\ref{eq:linear system}) and solution of (\ref{eq:nonlinear
system}), so the proof of this theorem relies on the well-posedness
of the nonlinear system (\ref{eq:nonlinear system}). Here we use the
Nash-Moser theorem \ref{thm:Nash-Moser} to solve it. Lemma
\ref{lem:nonlinear operator} ensures the first two assumptions of
Theorem \ref{thm:Nash-Moser} in the Appendix, and the third
assumption  of Theorem \ref{thm:Nash-Moser} follows from Proposition
\ref{prop:linearized system}. Then applying Theorem
\ref{thm:Nash-Moser} completes the proof of the theorem.\ef

\renewcommand{\theequation}{\thesection.\arabic{equation}}
\setcounter{equation}{0}
\section{Appendix. A Nash-Moser iteration theorem}

In order to solve the full water-wave system (\ref{eq:Hamiltonian
form-non}), here we present a variant of Nash Moser iteration
theorem in \cite{Lannes-IUMJ}. As far as one can see, we present
energy estimates with both scaled Sobolev spaces and standard Sobolev
spaces. One will find out easily that the Banach space
$\mathfrak{X}^s$ in our paper doesn't satisfy the definition of a
'Banach scale' in \cite{Lannes-IUMJ}, and that's the reason why a
modified Nash-Moser based on \cite{Lannes-IUMJ} is needed in our
paper.

 We shall
focus on the well-posedness of the singular evolution equations of
the form \ben\label{eq:singular equation} \left\{
\begin{array}{ll}\p_t \underline u^\eps+\f1\eps\cL^\eps(t)\underline u^\eps
+\cF^\eps[t,\underline u^\eps]=h^\eps,\\
\underline u^\eps|_{t=0}=\underline u^\eps_0,\end{array}\right. \een
where $\eps\in (0,\eps_0)$ is a small parameter, $\cL^\eps(t)$ is a
linear operator, while $\cF^\eps[t,\cdot]$ is a nonlinear mapping.
First of all, we introduce a family of Banach scale $\big(X^s,
|\cdot|_{X^s}\big)_{s\in\R}$  in the following sense:

\begin{defi}\label{defi9.1} {\sl We say that a family of Banach spaces $\big(X^s,
|\cdot|_{X^s}\big)_{s\in\R}$ are Banach scale if
\begin{itemize}
 \item[(1)]\,For all $s\le s'$, one has $X^{s'}\subset X^s$ and
$|\cdot|_{X^s}\le |\cdot|_{X^{s'}}$;

\item[(2)]\,There exits a family of smoothing operator $S_{\tht}$ $
(\tht\ge 1)$ satisfying $S_{2\tht}S_\tht=S_\tht$ and \beno \forall
s<s',\quad |(1-S_\tht)u|_{X^s}\le C_{s,s'}\tht^{s-s'}|u|_{X^{s'}};
\eeno

\item[(3)]\,There exists a linear positive operator $\Lambda$ such that
for $m\ge 0$,
$$
|S_\tht \Lambda^m u|_{X^s}\le C\tht^{m}|u|_{X^s}\quad\mbox{and}\quad
|\Lambda^m u|_{X^s}\le C|u|_{X^{s+m}};
$$

\item[(4)]\,The norms satisfy a convexity property \beno \forall s\le
s''\le s',\quad |u|_{X^{s''}}\le
C_{s,s',s''}|u|_{X^{s}}^\mu|u|_{X^{s'}}^{1-\mu}, \eeno where $\mu$
is determined by $\mu s+(1-\mu)s'=s''$. \end{itemize}}
\end{defi}

\no{\bf Notations.}\,If $X_1$ and $X_2$ be two Banach spaces, we
denote by $\mathfrak L(X_1,X_2)$ the set of all continuous mappings
from $X_1$ to $X_2$; If $X$ is a Banach space and $T>0$, $X_T$
stands for $C([0,T];X)$ with the norm $|\cdot|_{X_T}$; For $\cF\in
C([0,T];C^j(X_1,X_2))$, we denote by $d^j_u\cF$ the $j$-th order
derivatives of the mapping $u\mapsto \cF[\cdot,u]$. \vspace{0.1cm}

\noindent{\bf Assumption 8.1}. There exist $T>0$, $s_0\in\R$ and
$m\ge 0$ such that
\begin{itemize}
\item[(1)]\, For all $s\ge s_0$, one has
$(\cL^\eps(\cdot))_{0<\eps<\eps_0}$ is bounded in $C([0,T];$
$\mathfrak L(X^{s+m},X^s))$;

\item[(2)]\, For all $g\in X^s,$ the evolution operator $(U^\eps(\cdot))_{0<\eps<\eps_0}$
defined by
\[ U^\eps(t)g{\eqdefa}u^\eps(t),\quad \hbox{where}\quad
\p_tu^\eps+\f1\eps \cL^\eps(t)u^\eps=0,\quad u^\eps|_{t=0}=g\] is
bounded in $C([-T,T];\mathfrak L(X^s,X^s))$ for $s\ge s_0$.
\end{itemize}
\vspace{0.2cm}

\noindent{\bf Assumption 8.2}. There exist $T>0$, $s_0\in\R$ and
$m\ge 0$ such that for all $s\ge s_0$ and $0\le j\le 2$, one has
$\cF^\eps\in C([0,T];C^2(X^{s+m},X^s))$ and for all
$u,v_1,...,v_j\in X^{s+m}$, \beno
&&\sup_{[0,T]}\big|d_u^j\cF[t,u](v_1,...,v_j)\big|_{X^{s}}\le
C(s,T,|u|_{X^{s_0+m}})\\
&&\qquad\qquad\times\Big(\sum^j_{k=1}|\Lambda^{m}v_k|_{X^{s}}\prod_{l\neq
k}|v_l|_{X^{s_0+m}}+|\Lambda^{m}u|_{X^{s}}\prod^j_{k=1}|v_k|_{X^{s_0+m}}\Big).
\eeno Moreover for all $u,v_1,...,v_j\in X^{s}\cap X^{s_0+m}$, \beno
&&\sup_{[0,T]}\big|\Lambda^{-m}d_u^j\cF[t,u](v_1,...,v_j)\big|_{X^{s}}\le
C(s,T,|u|_{X^{s_0+m}})\\
&&\qquad\qquad\times\Big(\sum^j_{k=1}|v_k|_{X^{s}}\prod_{l\neq
k}|v_l|_{X^{s_0+m}}+|u|_{X^{s}}\prod^j_{k=1}|v_k|_{X^{s_0+m}}\Big).
\eeno

In order to state the third assumption, we need to introduce some
functional spaces as follows \beno
\begin{split}
&E^s_m{\eqdefa}\cap^1_{i=0}C^i([0,T];X^{s}),\quad |u|_{E^s_m}{\eqdefa}|u|_{X^s_T}+|\Lambda^{-m}\p_t u|_{X^{s}_T},\\
&F^s_m{\eqdefa}C([0,T];X^{s})\times X^{s+m},\quad
|(f,g)|_{F^s_m}{\eqdefa}|f|_{X^s_T}+|g|_{X^{s+m}},\\
&\mathfrak I^s_m(t,f,g){\eqdefa}|g|^2_{X^{s+m}}+\int^t_0\sup_{0\le
t''\le t'}|f(t'')|^2_{X^{s}}dt'. \end{split}\eeno

\noindent{\bf Assumption 8.3}. There exists $d_1\ge 0$ such that for
all $s\ge s_0+m$,  $u^\eps\in E^{s+d_1}_m$, and $(f^\eps,g^\eps)\in
F^{s}_m$, the IVP \beno
\p_tv^\eps+\f1\eps\cL^\eps(t)v^\eps+d_u\cF^\eps[t,u^\eps]v^\eps=f^\eps,\qquad
v^\eps|_{t=0}=g^\eps \eeno admits a unique solution $v^\eps\in
C([0,T];X^s)$ for all $\eps\in(0,\eps_0)$ and \beno
\begin{split}
|v^\eps|^2_{X^s_T}\le&
C(\eps_0,s,T,|u^\eps|_{E^{s_0+m+d_1}_m})\Big(\mathfrak
I^{s}_m(t,f^\eps,g^\eps)+ |\Lambda^{d_1}u^\eps|^2_{E^{s}_m}\mathfrak
I^{s_0+m}_m(t,f^\eps,g^\eps)\Big).\end{split} \eeno

In what follows, we shall always denote \beno
\begin{split}
&D_1{\eqdefa}d_1+m,\quad q{\eqdefa}D-m,\quad\hbox{and}\quad\\
&P_{min}{\eqdefa} D_1+\f
Dq\left(\sqrt{D_1}+\sqrt{2(D_1+q)}\right)^2.\end{split} \eeno Then
Nash-Moser iteration theorem is stated as follows.

\bthm{Theorem}\label{thm:Nash-Moser} {\sl Let $T>0$, $s_0$, $m$,
$d_1$ be such that Assumptions 8.1-8.3 are satisfied. Let $D>D_1$,
$P>P_{min}$, $s\ge s_0+m$, and let $(h^\eps,\underline
u^\eps_0)_{0<\eps <\eps_0}$ be bounded in $F^{s+P}_m$. Then there
exist $0<T'\le T$ such that (\ref{eq:singular equation}) has a
unique family of solutions $\{\underline u^\eps\}_{0<\eps<\eps_0}$
which are uniformly bounded in $C([0,T'];X^{s+D}).$ }\ethm

The proof of Theorem \ref{thm:Nash-Moser} essentially follows the
framework of \cite{Lannes-IUMJ}, and we omit the detailed proof
here.

\bigskip

\noindent {\bf Acknowledgments.} Part of this work was done when we
were visiting Morningside Center of Mathematics, CAS, in the summer
of 2009. We appreciate the hospitality and the financial support
from the center. The authors would like to thank David Lannes and
Thomas Alazard for profitable discussions. P. Zhang is partially
supported by NSF of China under Grant 10421101 and 10931007, and the
innovation grant from Chinese Academy of Sciences under Grant
GJHZ200829. Z. Zhang is supported by NSF of China under Grant
10990013 and 11071007.

\end{document}